\documentclass[11pt]{article}
\usepackage[utf8]{inputenc}
\usepackage{setspace}
\usepackage[margin=1.0in]{geometry}
\usepackage{amssymb, amsmath}
\usepackage{amsthm}
\usepackage{colordvi,verbatim,hyperref}
\usepackage{color,enumerate}
\usepackage{graphicx}
\usepackage{caption}
\usepackage{algorithm,algorithmic}
\usepackage{subcaption}
\usepackage[justification=centering]{caption}
\usepackage{authblk}
\usepackage[toc,page]{appendix}
\usepackage{lineno}
\usepackage{rotating}
\usepackage{tablefootnote}
\usepackage{mathtools}
\usepackage{natbib}

\theoremstyle{definition}
\newtheorem{defn}{Definition}

\newtheorem{thrm}{Theorem}

\allowdisplaybreaks
\begin{document}
\title{\textbf{Multi-Period Max Flow Network Interdiction with Restructuring for Disrupting Domestic Sex Trafficking Networks}\footnotetext[0]{Email addresses: d.kosmas@northeastern.edu (Daniel Kosmas), tcshark@clemson.edu (Thomas C. Sharkey), mitchj@rpi.edu (John E. Mitchell), k.maass@northeastern.edu (Kayse Lee Maass), mart2114@umn.edu (Lauren Martin)}}
\author[a]{Daniel Kosmas\footnote{Corresponding author (d.kosmas@northeastern.edu).}}
\author[b]{Thomas C. Sharkey} 
\author[c]{John E. Mitchell}
\author[a]{Kayse Lee Maass}
\author[d]{Lauren Martin}
\affil[a]{\footnotesize Department of Mechanical and Industrial Engineering, Northeastern University, 360 Huntington Avenue, 334 SN, Boston, MA 02115, USA}
\affil[b]{\footnotesize Department of Industrial Engineering, Clemson University, Freeman Hall, Clemson University
Clemson, SC 29634-0920, USA}
\affil[c]{\footnotesize Department of Mathematics, Rensselaer Polytechnic Institute, 110 8th Street, Troy, NY 12180, USA}
\affil[d]{\footnotesize School of Nursing, University of Minnesota, 308 SE Harvard St, Minneapolis, MN 55455, USA}
\date{}
\maketitle
\vspace{-1.5cm}
\begin{abstract}
    We consider a new class of multi-period network interdiction problems, where interdiction and restructuring decisions are decided upon before the network is operated and implemented throughout the time horizon. We discuss how we apply this new problem to disrupting domestic sex trafficking networks, and introduce a variant where a second cooperating attacker has the ability to interdict victims and prevent the recruitment of prospective victims. This problem is modeled as a bilevel mixed integer linear program (BMILP), and is solved using column-and-constraint generation with partial information. We also simplify the BMILP when all interdictions are implemented before the network is operated. Modeling-based augmentations are proposed to significantly improve the solution time in a majority of instances tested. We apply our method to synthetic domestic sex trafficking networks, and discuss policy implications from our model. In particular, we show how preventing the recruitment of prospective victims may be as essential to disrupting sex trafficking as interdicting existing participants.
    \newline Keywords: network interdiction, multi-period, sex trafficking, illicit networks
\end{abstract}

\section{Introduction}
\label{sec:intro}
Human trafficking is an egregious violation of human rights and dignity. The International Labour Organization estimated \$150 billion in annual profits from human trafficking \citep{ilo2014}, but the real impact is not well understood \citep{fedina2015use}. Victims of human trafficking are recruited through violence and force, fraudulent job opportunities, fake romantic interest, manipulation, offers of safe migration, or by traffickers exploiting their lack of access to basic needs. Traffickers compel work through strong arm tactics such as extreme physical and sexual violence, threats of unmanageable debts or quotas, confiscation of critical documents, and social isolation and confinement \citep{polaris, carpenter2016nature, martin2014mapping, preble2019under}. Traffickers also groom their victims into believing that performing the tasks asked of them will solidify a relationship with the trafficker \citep{cockbain2018offender}. This work, including physical labor, peddling and begging, transactional sex, and other illicit activities, would not be performed by the victim without force, fraud, or coercion \citep{polaris}. We focus on the specific case of sex trafficking, which is part of a complex and stigmatized commercial market \citep{dank2014estimating,marcus2016pimping,martin2017mapping}.

\cite{konrad2017overcoming} identified that current efforts to combat human trafficking can be supplemented by techniques from operations research. However, due to significant challenges in data collection in relation to sex trafficking \citep{fedina2019risk, gerassi2017design, weitzer2014new}, information about networks within sex trafficking operations that is needed to inform these techniques is limited. Thus the potential impacts of network disruptions are even less well-known. Qualitative studies have yielded some relevant insights. For example, sex trafficking operations have wide variation in level of complexity, the degree of hierarchy, and role of actors \citep{cockbain2018offender}. Individuals often move between networks and change roles over time \citep{cockbain2018offender, denton2016anatomy, martin2014benefit, martin2014mapping}. Thus, victims can move up in the hierarchy to escape violence or they can exit the network through escape, intervention (e.g. law enforcement or social services), or moving to a different network. The amount and frequency of ``turnover” of victims and the response of trafficking networks is not yet well understood \citep{caulkins2019call, konrad2017overcoming, martin2014benefit}. Pre-existing relationships among network members seem to be an important factor in establishing trust and maintaining a network \citep{cockbain2018offender}. Connections between legal and criminalized commercial sex, as well as between formal and informal networks, add complexity to how networks function \citep{cockbain2018offender, dewey2018routledge}. Social problems such as poverty, running away from home, homelessness and addictions make some people more vulnerable to being trafficked for sexual exploitation \citep{fedina2019risk, franchino2021vulnerabilities, ulloa2016prevalence}. These nuanced market and social factors shape how sex trafficking networks function, including how they recruit and retain victims \citep{cockbain2018offender, dewey2018routledge}. Modeling of networks and potential disruptions must account for this nuance and complexity. 

Policing disruptions in the US have been primarily directed toward arrest and prosecution of traffickers and identification of victims \citep{farrell2020measuring}, with limitations and uneven application \citep{farrell2015police}. Social service and healthcare responses have focused on finding victims and referring them to supportive services such as: housing, therapy, addiction treatment, mental health support, and job training \citep{hounmenou2019review, macy2021scoping, roby2017federal}.\cite{moynihan2018interventions} identified a variety of different (non-law enforcement) intervention strategies that promote the well-being of victims, such as focused health and/or social services and residential programs. While necessary to remediate the short and long-term harms to victims of sex trafficking, these approaches do not get ahead of trafficking before it happens and we lack data on their impact. \cite{franchino2021vulnerabilities} identified common risk factors that lead to victims being trafficked, including, but not limited to homelessness, negative mental health and an early introduction to drugs and alcohol. While rigorous evaluation studies of the impact of social services on vulnerabilities to trafficking are lacking, we hypothesize that providing resources to address these risk factors will help prevent victimization.

It is known that traffickers will adapt to anti-trafficking activities in order to minimize detection and maximize profits \citep{surtees2008traffickers}.  Effectively combating human trafficking will require the coordination of multiple different organizations with different intervention strategies. However, there are tensions between organizations that have limited the effectiveness of these coordination efforts \citep{foot2015collaborating}. Addressing these tensions can lead to more successful anti-trafficking efforts. \cite{foot2021outcome} explored how counter-human trafficking coalitions can lead to more positive outcomes of their efforts.  \cite{pajon2022importance} proposed that law enforcement needs to collaborate with other agencies, such as those that are able to ensure the safeguarding of victims, in order to more effectively investigate and prosecute traffickers. In this paper, we model how cooperating anti-trafficking stakeholders can more effectively disrupt human trafficking and assess the impact of this cooperation.

Because trafficking operations rely on the ability to obtain and control victims in order to generate profit, it is important to model recruitment dynamics. Research suggests that removing individual victims from a trafficking situation, while clearly necessary, might paradoxically result in more victims being recruited into trafficking after a disruption \citep{caulkins2019call,martin2014benefit}. Theorizing how networks may restructure after a disruption, specifically the removal of a victim from a trafficking situation, is necessary to inform the field about effective disruption. Our work here helps to mathematically understand the limitations of victim-level disruptions related to recruitment on the overall prevalence of human trafficking. Although, ethics dictate that we must continue to provide exit options for victims of trafficking and our results suggest a combination of exit options and preventative measures to recruitment are necessary for organizations whose goals are to disrupt trafficking in the long-term. 

Applying tools from operations research to disrupting human trafficking has been a focus of recent research \citep{caulkins2019call, dimas2021survey, konrad2017overcoming}. \cite{smith2020survey} suggests that network interdiction may prove useful in supporting anti-trafficking efforts. Network interdiction is a two player Stackelberg \citep{stackelberg1952theory} game that is commonly used to model adversarial scenarios involving networks. One player, the defender, seeks to operate the network to the best of their ability (shortest path, maximum flow, etc.). The other player, the attacker, tries to inhibit the defender's ability to operate the network by removing nodes or arcs, subject to certain constraints, before the defender has the ability to operate the network. In this work, we focus on max flow network interdiction, which has previously been applied to telecommunications \citep{baycik2018interdicting}, electrical power \citep{salmeron2009worst}, transportation \citep{alderson2011solving}, and illicit drug trafficking \citep{malaviya2012multi}. It has recently begun to be applied to disrupting human trafficking \citep{kosmas2022generating, mayorga2019countering, xie2022interdependent}. In applying max flow to human trafficking, the practical interpretation of flow depends on the scale of the trafficking operation. We consider domestic sex trafficking networks, where the flow has been interpreted as the ability of the traffickers to control and/or coerce their victims \citep{kosmas2022generating}. Interdictions for a domestic sex trafficking network can be practically interpreted as actions such as a trafficker being arrested by law enforcement or agencies providing a victim a service to address one of their vulnerabilities, such as access to affordable housing and health care. 

In this work, we introduce the multi-period max flow network interdiction problem with restructuring (MP-MFNIP-R), extending the work of \cite{kosmas2020interdicting} to include a temporal component. In this problem, interdiction and restructuring decisions are decided upon upfront and implemented throughout the time horizon. We discuss how we extend the work of \cite{kosmas2022generating} so that constraints on interdictions and restructuring for domestic sex trafficking networks can be formulated for a multi-period model. We additionally propose a variant with two cooperating attackers with different abilities to interdict the network, modeling how multiple anti-trafficking stakeholders would cooperate. We formulate MP-MFNIP-R as a bilevel mixed integer linear program (BMILP), which can be solved using column-and-constraint generation.  We then derive a column-and-constraint generation (C\&CG) algorithm to solve the BMILP, and propose augmentations to the C\&CG algorithm based on modeling choices for domestic sex trafficking networks. We test our model on validated synthetic domestic sex trafficking networks that are grounded in real-world experience \citep{kosmas2022generating}. These tests demonstrate the efficacy of our augmentations, as well as how recommended interdiction prescriptions change based on different modeling choices.

\subsection{Literature Review}
\cite{wood1993deterministic} originally proposed max flow network interdiction, and there have been many extensions proposed since the original work. \cite{derbes1997efficiently} was the first work to consider incorporating a temporal component in max flow network interdiction.  \cite{rad2013maximum} also considered a multi-period max flow network interdiction model where each arc has a traversal time for flow to travel across it, and derive a Benders' decomposition algorithm based on temporally repeated flows to solve this problem. \cite{zheng2012stochastic} proposed a stochastic version of the multi-period max flow network interdiction model, where the attacker has incomplete information on the network structure. \cite{soleimani2017solving} considered a multi-period interdiction model where flow is sent from source to sink instantaneously, and solve the problem with generalized Benders' decomposition. They also expand this model to include uncertainty in the arc capacities \citep{soleimani2018dynamic}. \cite{malaviya2012multi} and \cite{jabarzare2020dynamic} applied multi-period max flow network interdiction models to illicit drug trafficking networks. In all of these works, the network remains static, and is not allowed to change after interdictions have been implemented. Specialized methods proposed by these works, such as the Benders' decomposition based on temporally repeated flows, are no longer applicable if the underlying network changes in different time periods.

Understanding how networks ``react" to interdiction has been identified as a key feature necessary to applying network interdiction models to disrupting human trafficking \citep{caulkins2019call, konrad2017overcoming}. However, it has received little attention, since incorporating the ability to change the network after interdictions have been implemented proves to be computationally difficult even in a single time period. \cite{sefair2016dynamic} proposed a dynamic version of the shortest path interdiction problem, where the attacker and defender alternate between the attacker interdicting the network and the defender traversing an arc. \cite{holzmann2019shortest} introduced the shortest path interdiction problem with improvement (SPIP-I), where the defender has a limited budget to reduce the cost of traveling along certain arcs after the attacker has interdicted the network. \cite{kosmas2020interdicting} introduced the max flow network interdiction problem with restructuring (MFNIP-R), where the defender has a limited budget to add arcs to the network in response to the implemented interdictions. The works for \cite{holzmann2019shortest} and \cite{kosmas2020interdicting} do not include a temporal component.

Applying tools from operations research (OR) to disrupting human trafficking has been receiving more attention over the last few years. \cite{konrad2017overcoming} was among the first works suggesting how the OR and analytics community could support anti-human trafficking efforts. They suggest network interdiction may prove useful in combating human trafficking and note that there were modeling nuances that need to first be addressed, such as `the ability to accommodate dynamic changes' and that `trafficked humans are a ``renewable commodity."' \cite{caulkins2019call} additionally suggested that intervention strategies must also account for unintended consequences, highlighting an example of an intervention in seafood supply chains that use labor trafficking also could harm the legal industry. \cite{dimas2021survey} reviewed literature in OR and analytics for human trafficking that was published between 2010 and March 2021. They identified that the majority of the published works in this time period focused on machine learning classification/clustering methods. They also noted that many works in OR and analytics for human trafficking are broadly focused, and this broad focus can potentially lead to models not appropriately accounting for unique nuances for specific populations of victims and survivors. \cite{sharkey2021better} stated that combating human trafficking is a transdisciplinary challenge, and that, to ensure that the models developed by the OR and analytics communities are appropriately accounting for these unique nuances, researchers should employ a transdisciplinary research approach. By collaborating with subject-matter experts, both in and out of academia, models will be better developed to account for these nuances. \cite{martin2022learning} demonstrated the process in which they built a transdisciplinary research team to address sex trafficking and recommended that effective team-building was the key to establishing the respect and trust needed to develop a shared language across a diverse set of disciplines.

Two perspectives are currently being considered when developing network interdiction models for disrupting human trafficking: macroscopic and microscopic. Macroscopic models seek to disrupt the movement of trafficking victims from their origin location to where demand is, while microscopic models seek to disrupt the exploitation of trafficked individuals when they are at their destination. These differing perspectives support each other by helping address the limitations of the other perspective. Macroscopic models are currently limited by failing to include how victims are exploited after they reach their destination, which is the primary focus of microscopic models. In turn, microscopic models are limited by not fully accounting for how victims are moved within trafficking networks, which is the primary focus of macroscopic models.

To the best of our knowledge, three works explore the macroscopic perspective. \cite{mayorga2019countering} applied a network interdiction model to networks where victims were moved between illicit massage parlors in a geographic area. \cite{tezcan2020human} explored a multi-period network interdiction model where the probability of an interdiction being successful was dependent on the success of previous interdictions, and applied this model to human trafficking across the Nepal-India border. \cite{xie2022interdependent} proposed a multi-period interdependent network interdiction model on the sex trafficking supply chain, where flow needs to pass through the communication network before victims can be moved through the physical network. \cite{mayorga2019countering} and  \cite{xie2022interdependent} considered the flow through the network to be the victims themselves, whereas \cite{tezcan2020human} considered the flow to be the desirability of a trafficker to travel across an an arc. None of these works allow for the underlying network to change after interdictions. \cite{kosmas2022generating} is currently the only work investigating the microscopic perspective. They applied a network interdiction model with restructuring to domestic sex trafficking networks. In their work, they consider the flow through the network to be the ability of a trafficker to control their victims. We expand upon their work by extending the model they used to include a temporal component. This extension will allow policy-makers to better understand the long-term impacts, both intended and unintended, of their proposed anti-trafficking efforts to prevent exploitation.

To solve their model, \cite{kosmas2022generating} implemented a column-and-constraint generation algorithm. Column-and-constraint generation was originally proposed by \cite{zeng2014solving} to solve bilevel mixed integer linear programs that satisfy the relatively complete response property, meaning that for every feasible integer upper level and integer lower level solution, there is a feasible continuous upper level and continuous lower level solution (i.e., there is no pair of upper level and lower level integer solutions that make the bilevel problem infeasible). \cite{yue2019projection} expanded C\&CG to solve general BMILPs by incorporating implications constraints to remove lower level integer decisions that are infeasible with respect to the upper level integer decision that is being considered in the branch-and-bound procedure.  \cite{kosmas2020interdicting} adapted this procedure for network interdiction models with restructuring, where the restructuring decisions are monotonic with respect to the interdiction decisions. They do so by instead incorporating partial information from previously visited restructuring plans, identifying which components of the restructuring plans remain feasible as the interdiction decisions being considered in the branch-and-bound procedure change. \cite{kosmas2020interdicting} showed that this adaptation is necessary to solve models where new participants could be recruited into the illicit network. We apply the algorithm of \cite{kosmas2020interdicting} to solve the multi-period version of their problem, as well as suggest algorithmic augmentations that lead to significant computational gains based on modeling choices for disrupting domestic sex trafficking networks.

We summarize the key differences between the reviewed literature and this work. Previous max flow network interdictions models that include a temporal component only allow for the defender to respond to interdictions by shifting how flow is sent through the network. Our work expands on the capabilities of the defender by also allowing them to add arcs to the network over time. Prior network interdiction models that allow for the defender to add arcs to the network have focused on shortest path interdiction, not max flow interdiction. This work additionally distinguishes itself from prior work on applying network interdiction models to disrupting sex trafficking by including a temporal component in the model where flow is modeled as control. The network interdiction model we propose accounts for two of the modeling nuances identified by \cite{konrad2017overcoming}, dynamic changes and the ``re-usability" (over time) of trafficking victims, which no prior work has fully explored.

\subsection{Paper Organization}
This paper is organized as follows: Section \ref{sec:prob} formally introduces the multi-period max flow network interdiction problem with restructuring and proposes a bilevel mixed integer linear programming formulation of the problem. Section \ref{sec:modelInt} discusses modeling choices regarding interdiction decisions and Section \ref{sec:modelRest} discusses modeling choices regarding restructuring decisions for disrupting domestic sex trafficking networks. Section \ref{sec:modelderiv} derives an equivalent linear program that can be solved by column-and-constraint generation. Section \ref{sec:aug} proposes modeling-based augmentations to improve the solution time of the column-and-constraint generation procedure. Section \ref{sec:results} presents results on validated synthetic domestic sex trafficking networks, both comparing the quality of solution times with the proposed augmentations and discussing policy recommendations provided by our model. Section \ref{sec:conc} concludes the paper and discusses avenues for future research.

\section{Problem Description}
\label{sec:prob}
We first review how the model of  \cite{kosmas2022generating} is constructed. The networks considered are each trafficker's operation, as as well as the social network between traffickers, which helps to capture potential reactions after interdictions. The node set $N$ is partitioned into different sets, based on the role the node plays in the network. They first consider the roles of participants currently active in the network. These are traffickers, victims, and bottoms. A bottom is a victim who assists the trafficker in managing the trafficking operation as part of the activities they are forced to perform \citep{belles2018defining}. Bottoms are typically viewed as the most trusted or highest earning victim \citep{roe2015sexual}. Let $T$ be the set of traffickers, $B$ be the set of bottoms, and $V$ be the set of victims. Additionally, they consider participants who can be brought into the network, or those who can have their roles change. For example, if a trafficker is interdicted, one of their friends or family members may be able to take over the operations of the trafficking network \citep{dank2014estimating}. Alternatively, if the bottom is interdicted, the trafficker may promote another victim to take over the responsibilities of the previous bottom. Let $T^R$ be the set of back-up traffickers, $B^R$ be the set of victims that can be promoted to be a bottom, and $V^R$ be the set of prospective victims. We summarize all notation in Appendix \ref{note}.

In \cite{kosmas2022generating}, the traffickers operate the networks (typically referred to as the \emph{defender} in an interdiction problem), and the anti-trafficking stakeholder is trying to interdict the network (typically referred to as the \emph{attacker}). Both players have complete information about the game, having full knowledge about the network and each other's decisions and objectives. Each trafficker has limited ability to coerce their victims and acquire new victims. In their model, they consider the flow through the trafficking network to be the ability of a trafficker (or bottom) to coerce a victim into providing labor. Their model interdicts nodes instead of arcs, representing the removal of participants from the networks. After interdictions, they allow arcs from a set $A^R$ to be added to the network. The addition of these arcs is referred to as \emph{restructuring}. These arcs belong to two different sets $A^{R, out}$ and $A^{R, in}$, based on the role of the participants that is allowed to initiate that restructuring. As described in \cite{kosmas2022generating}, an ``out" restructuring (an arc belonging to $A^{R,out}$) is initiated by the trafficker. Examples of this include arcs that model a trafficker restructuring to a new victim after one of their victims has been interdicted or a trafficker assigning one of their victims to their bottom. An ``in" restructuring (an arc belonging to $A^{R,in}$) is initiated by the victim, such as a victim being recruited into a new operation after their trafficker has been interdicted. We will describe which belong to $A^{R,out}$ and $A^{R, in}$ in Section \ref{sec:modelRest}. In their work, $A \cap A^R = \emptyset$.

We now introduce the multi-period max flow network interdiction problem with restructuring (MP-MFNIP-R). MP-MFNIP-R is a two player game on a network $G = (N,A)$, with $N$ being the set of nodes and $A$ being the set of arcs currently in the network, and $A^R = A^{R, out} \cup A^{R, in}$ representing the set of restructurable arcs. Nodes and arcs (including restructurable arcs) are assigned capacities $u: N \cup A \cup A^R \rightarrow \mathbb{R}_+$, and a victim node $i$ that is promoted to be the new bottom will have their capacity increased by $\tilde{u}_i$. We denote that $\alpha \in N$ is the source node, and $\omega \in N$ is the sink node. Let $\tau$ be the number of time periods.

Gameplay for MP-MFNIP-R is as follows. First, the attacker decides upon an interdiction plan. Interdictions are decided upon at the beginning of the time horizon, and require a certain length of time before they are implemented, setting the capacity of the interdicted node to $0$. After the interdiction plan is decided upon, the defender decides upon a restructuring plan in response to the attacker's interdiction plan. Restructuring decisions are also decided upon at the beginning of the time horizon, and require a certain length of time before they are implemented, setting the capacity of each restructured arc to its non-zero capacity. After the interdiction and restructuring decisions are made, the defender operates the network, sending flow from source to sink instantaneously each time period. The goal of the defender is to maximize the amount of flow sent throughout the entire time horizon, and the goal of the attacker is to minimize the amount of flow sent throughout the entire time horizon.

To describe the mathematical program, we must first define the decision variables. Let $x_{it}$ be the flow across node $i$ at time $t$ for $i \in N$ and $t = 1, \ldots, \tau$, and $x_{ijt}$ be the flow across arc $(i,j)$ at time $t$ for $(i,j) \in A \cup A^R$ and $t = 1,\ldots, \tau$. Let

\small\begin{equation*}
    y_{i} = \begin{cases} 1 \text{ if node } i \text{ has been interdicted,}\\
    0 \text{ otherwise,}
    \end{cases}
\end{equation*}

and let 

\small\begin{equation*}
    \gamma_{it} = \begin{cases} 1 \text{ if node } i \text{ is interdicted in or before time period } t \text{, }\\
    0 \text{ otherwise.}
     \end{cases}
\end{equation*}

Additionally, let

\small\begin{equation*}
    z^{out}_{ij} = \begin{cases} 1 \text{ if arc } (i,j) \text{ has been restructured when } i \text{ is able to initiate the restructuring,}\\
    0 \text{ otherwise,}
     \end{cases}
\end{equation*}

and

\small\begin{equation*}
    z^{in}_{ij} = \begin{cases} 1 \text{ if arc } (i,j) \text{ has been restructured when } j \text{ is able to initiate the restructuring,}\\
    0 \text{ otherwise.}
     \end{cases}
\end{equation*}
Let 
\small\begin{equation*}
    \zeta^{out}_{ijt} = \begin{cases} 1 & \begin{aligned}
        \text{if arc } (i,j) \text{ has been restructured in or before time period } t \\\text{ when } i \text{ is able to initiate the restructuring,}
    \end{aligned}\\
    0 & \text{otherwise,}
     \end{cases}
\end{equation*}
and
\small\begin{equation*}
    \zeta^{in}_{ijt} = \begin{cases} 1 & \begin{aligned}
        \text{if arc } (i,j) \text{ has been restructured in or before time period } t \\\text{ when } j \text{ is able to initiate the restructuring,}
    \end{aligned}\\
    0 &\text{otherwise.}
     \end{cases}
\end{equation*}

We now define relevant parameters independent of the application to domestic sex trafficking. Let $\delta_{i}^y$ be the number of time periods needed before node $i$ can be interdicted. Let $\delta_{ij}^z$ be the number of time periods needed before arc $(i,j)$ can be restructured. Let $Y$ be the set of all feasible interdiction decisions, and for each $y \in Y$, let $Z(y)$ be the set of all feasible restructuring decisions responding to interdiction plan $y$. The constraints defining $Y$ will be further defined in Section \ref{sec:modelInt}, and the constraints defining $Z(y)$ will be further defined in Section \ref{sec:modelRest}.

We can now describe the bilevel programming formulation of MP-MFNIP-R.
\begin{singlespace}
\begin{subequations}
\label{minmax1}
\footnotesize\begin{align} 
    \min_{y, \gamma} \max_{x, z, \zeta} ~~~ & \sum_{t = 1}^{\tau} \sum_{i \in N: (\alpha,i) \in A \cup A^{R,out}} x_{\alpha it} & \nonumber\\
    \text{s.t. }& \sum_{(h,i) \in A \cup A^{R,out}} x_{hit} = x_{it} & \text{ for } i \in N \setminus \alpha, t = 1,\ldots,\tau \label{con1:inflow}\\
    & x_{it} = \sum_{(i,h) \in A \cup A^{R,out}} x_{iht} & \text{ for } i \in N \setminus \omega, t = 1,\ldots,\tau \label{con1:outflow}\\
    & 0 \le  x_{ijt} \le u_{ij} & \text{ for } (i,j) \in A, t = 1,\ldots,\tau \label{con1:arccap}\\
    & 0 \le  x_{ijt} \le u_{ij} \zeta^{out}_{ij} & \text{ for } (i,j) \in A^{R,out} \setminus A^{R,in}, t = 1,\ldots,\tau \label{con1:arcrescap1}\\
    & 0 \le  x_{ijt} \le u_{ij} (\zeta^{out}_{ijt}+\zeta^{in}_{ijt}) & \text{ for } (i,j) \in A^{R,in}, t = 1,\ldots,\tau \label{con1:arcrescap2}\\
    & 0 \le  x_{it} \le u_i(1 - \gamma_{it}) & \text{ for } i \in N\setminus\{j \in V: \exists h \in B \text{ with } (h,j) \in B^R \}, t = 1,\ldots,\tau \label{con1:nodecap1}\\
    & 0 \le  x_{jt} \le u_j(1 - \gamma_{jt}) + \Tilde{u_j} \zeta^{out}_{\alpha jt} & \text{ for } j \in V \text{s.t.} \exists i \in B, (i,j) \in B^R, t = 1,\ldots,\tau \label{con1:nodecap2}\\
    & (y, \gamma) \in Y  \\
    & (z, \zeta) \in Z(y)
\end{align}
\end{subequations}
\end{singlespace}

The objective function of \eqref{minmax1} is the sum of the flows out of the source node across all time periods. Constraints \eqref{con1:inflow} - \eqref{con1:outflow} are flow balance constraints. Constraints \eqref{con1:arccap} - \eqref{con1:arcrescap2} are the capacity constraints on the arcs and restructurable arcs, and constraints \eqref{con1:nodecap1} - \eqref{con1:nodecap2} are the capacity constraints on the nodes. It is worth noting that this model builds off the work done in \cite{kosmas2020interdicting} by adding the time dimension. We additionally note that traditional max flow models only have a single flow balance constraint for each node. However, when interdicting nodes instead of arcs, it is necessary to have the pair of constraints. An equivalence between this model and an interdiction model where arcs are interdicted is established in \cite{malaviya2012multi}.

\section{Modeling Interdictions}
\label{sec:modelInt}
We now describe constraints regarding interdiction, based on interdicting domestic sex trafficking networks. We first describe constraints linking $y$ and $\gamma$.

\begin{align}
    \gamma_{it} = 0 & \text{ for } i \in N \setminus \{\alpha,\omega\}, t \in 1, \ldots, \delta^y_i \label{con:noInt}\\
    \gamma_{it} = y_i & \text{ for } i \in N \setminus \{\alpha,\omega\}, t \in (\delta^y_i + 1), \ldots, \tau \label{con:yesInt}
\end{align}

Constraints \eqref{con:noInt} indicate that a node $i$ will be able to carry flow for the first $\delta^y_i$ time periods, regardless of interdiction decisions. Then, after $\delta^y_i$ time periods have elapsed, constraints \eqref{con:yesInt} enforce that the node will be unable to carry flow if it was interdicted. 

We additionally include a budget constraint to limit the overall number of interdictions implemented. Let $r_i$ be the cost to interdict node $i$ and let $b$ be the total budget of the attacker. We use the same budget constraints described in \cite{kosmas2022generating}, where the cost to interdict a trafficker can be decreased based on interdicting their bottom (if the trafficker has one) and victims. This models how, if victims or bottom are willing to cooperate with law enforcement, it is easier for law enforcement to build a successful case \citep{clawson2008prosecuting, david2008trafficking}. To represent this, we define additional variables $\tilde{r}_i$ to be the adjusted cost of interdicting trafficker $i$.  Additionally, for each trafficker $i \in T$, let $r_i^{min}$ be the minimum cost of interdicting trafficker $i$ after interdicting their victims, and let $d_{il}$ be the reduction in cost of interdicting trafficker $i$ if victim (or bottom) $l$ is also interdicted. We assume the discount in cost to interdict the trafficker is additive until the cost would be reduced $r_i^{min}$, in which case the cost remains at $r_i^{min}$. The following constraints capture the interdiction budget and costs to interdict each node.

\begin{align}
    &\sum_{i \in T} \Tilde{r}_i y_i + \sum_{j \in B \cup N} r_i y_i \le b & \label{con:budget}\\
    &\Tilde{r}_i \ge r_i - \sum_{l \in B \cup V} d_{il} y_i & \text{ for } i \in T
    \label{con:adjcost}\\
    &\Tilde{r}_i \ge r_i^{min} & \text{ for } i \in T \label{con:mincost}
\end{align}

Constraint \eqref{con:budget} enforces that the chosen interdictions respect the overall attacker budget. Constraints \eqref{con:adjcost}-\eqref{con:mincost} compute the adjusted cost to interdict a trafficker.

\subsection{Modeling Extension: Cooperating Attackers}
There are multiple stakeholders in the efforts to disrupt human trafficking, which have a variety of different means of disrupting the network. For example, law enforcement has the ability to arrest and prosecute participants in the trafficking network, while social service professionals have the ability to provide services to victims to help them leave the trafficking network while also helping reduce the vulnerabilities of prospective victims.

We want to model a second interdictor that only has the ability to interdict victims and prevent the ability to add prospective victims to the network. Let $y'_i$ be the indicator of whether or not the second attacker interdicts node $i$ for $i \in V \cup V^R$. Let $r'_i$ be the cost for the second interdictor to interdict node $i$, and let $b'$ be the budget of the second attacker. We include the following constraint regarding the second attacker:

\begin{align}
    \sum_{i \in V \cup V^R} r'_i y'_i \le b' \label{con:2ndBudget}
\end{align}

To update the model to incorporate the second attacker, we must update the existing constraints regarding interdicting victims, as well as include constraint \eqref{con:2ndBudget}. For $i \in V$, we adjust the constraints \eqref{con:yesInt} from $\gamma_{it} = y_i$ to $\gamma_{it} = y_i + y'_i$ for $t \in (\delta^y_i+1), \ldots, \tau$. For $i \in V^R$, we introduce the constraints $\gamma_{it} = y'_i$ for $t = 1, \ldots, \tau$. If a prospective victim $j$ were to be interdicted, then even if an arc $(i,j)$ were to be restructured from trafficker $i$, the capacity of the prospective victim node would be set to $0$ for $t = \delta^y_j, \ldots, \tau$, and thus would not increase the objective value in time periods after the victim is interdicted. We note that this formulation assumes that the two attackers are cooperating and fully aware of each other's actions, which may not be true in reality \citep{foot2015collaborating}. However, there has been little quantitative work in exploring cooperating attackers within network interdiction \citep{sreekumaran2021equilibrium, wilt2019measuring} and this first effort helps to shed light on how two attackers can cooperate and coordinate their efforts to have the most impact on the trafficking network. Future work on applying interdiction models to disrupting trafficking networks can better model the capabilities of different stakeholders, as well as their willingness to cooperate, and their knowledge of each other's activities.

\section{Modeling Restructuring}
\label{sec:modelRest}
For the set of restructurable arcs, we use the same sets outlined in \cite{kosmas2022generating}. This set includes traffickers recruiting each other's victims (or alternatively, a victim joining a different trafficker's operation after their trafficker has been interdicted), traffickers assigning or taking victims from their bottom, recruiting new victims not currently in the network, back-up traffickers taking over an interdicted trafficker’s operation, a victim being promoted to take the place of an interdicted bottom, and traffickers assigning victims to their newly promoted bottom. \cite{kosmas2022generating} decided upon these types of restructurable arcs by analyzing transcripts from meetings with their qualitative research team and survivor-centered advisory group that described how trafficking networks may react after interdictions.

We extend the constraints on the restructuring outlined in \cite{kosmas2022generating} to a multi-period setting. We first propose constraints that link when interdictions occur to when restructurings are allowed. Consider a feasible interdiction plan $\bar{y} \in Y$. For $i \in T$, let $\delta_i^{min,out} = \min_{(i,j) \in A^{R,out}} \delta_{ij}^z$, and $\delta_i^{max,out} = \max_{(i,j) \in A^{R,out}} \delta_{ij}^z$, which help to capture the earliest and the latest time which a trafficker may initiate a restructuring after an interdiction has occurred. For $j \in V$, let $\delta_j^{min,in} = \min_{(i,j) \in A^{R,in}} \delta_{ij}^z$, and $\delta_j^{max,in} = \max_{(i,j) \in A^{R,in}} \delta_{ij}^z$, which help to capture the earliest and the latest time which a victim may initiate a restructuring after an interdiction has occurred.

\small{
\begin{align}
    & \sum_{(i,j) \in A_R^{out}} z_{ij}^{out} \le \sum_{h \in N: (i,h) \in A} \bar{y}_{h} &\text{ for all } i \in T \label{con:outRes}\\
    & \sum_{(i,j) \in A_R^{in}} z_{ij\tau}^{in} \le \sum_{h \in N: (h,j) \in A} \bar{y}_{h} &\text{ for all } j \in V \label{con:inRes}\\
    & \sum_{\substack{j \in N :(i,j) \in A^{R, out}\\\delta^z_{ij} \ge t_1}} \zeta_{ij(t_2+\delta^z_{ij})}^{out} \le \sum_{h: (i,h) \in A, (\delta^y_h +1) \le t_2} \bar{y}_{h} &\begin{multlined}\text{ for } i \in T, t_1 = \delta^{min, out}_{i}, \ldots, \delta^{max, out}_{i}\\ t_2 = 1, \ldots, \tau- \delta^{max, out}_i \end{multlined} \label{con:timeOutRes}\\
    & \sum_{\substack{i \in N :(i,j) \in A^{R, in}\\ \delta^z_{ij} \ge t_1}} \zeta_{ij(t+\delta^z_{ij})}^{in} \le \sum_{h: (h,j) \in A, (\delta^y_h +1) \le t} \bar{y}_{h} &\begin{multlined}\text{ for } j \in V, t_1 = \delta^{min, in}_{j}, \ldots, \delta^{max, in}_{j}\\ t_2 = 1, \ldots, \tau- \delta^{max, in}_j \end{multlined} \label{con:timeInRes}
\end{align}
}

Constraints \eqref{con:outRes}-\eqref{con:inRes} determine the number of restructurings each trafficker and victim is allowed to make in response to the implemented interdictions. Constraints \eqref{con:timeOutRes}-\eqref{con:timeInRes} enforce that a restructured arc can only come online after the required amount of time has passed. We next define constraints to enforce the relationship between $z$ and $\zeta$.

\small{
\begin{align}
    & \zeta_{ijt}^{out} \ge \zeta_{ij(t-1)}^{out} &\text{ for } (i,j) \in A^{R,out}, t = 2, \ldots \tau \label{con:subseqOut}\\
    & \zeta_{ijt}^{in} \ge \zeta_{ij(t-1)}^{in} &\text{ for } (i,j) \in A^{R, in}, t = 2, \ldots \tau  \label{con:subseqIn}\\
    & \zeta_{ij\tau}^{out} = z_{ij}^{out} &\text{ for } (i,j) \in A^{R,out} \label{con:finOut}\\
    & \zeta_{ij\tau}^{in} = z_{ij}^{out} &\text{ for } (i,j) \in A^{R,in} \label{con:finIn}
\end{align}
}

Constraints \eqref{con:subseqOut}-\eqref{con:subseqIn} enforce that if an arc has turned online, it must remain online in the following time period. Constraints \eqref{con:finOut}-\eqref{con:finIn} ensure that if an arc was restructured, it will be online in the last time period. The combination of these constraints, as well as \eqref{con:outRes}-\eqref{con:timeInRes}, ensure that, once it is decided that an arc is to be restructured, it will first come online in a time period that it is allowed to, and remain online for the rest of the time horizon. They also ensure that the restructurings chosen occur in the appropriate time periods based on the chosen interdictions, and that the number of restructurings is limited by the number of relevant interdictions.

The next constraints limit the number of actions that each trafficker can take over the entire time horizon, independent of the number of interdictions that have occurred. This is to mirror the budget constraint of the attacker. Let $c_i^{out}$ be the number of actions trafficker $i$ can take, and $c_j^{in}$ be the number of actions victims $j$ can take We note that the constraints include the ability of each trafficker to bring new victims into their organization, assign new victims to their bottom, and to promote a victim to be a bottom.

\small{
\begin{align}
    & \sum_{j \in N: (i,j) \in A^{R,out}}z^{out}_{ij} + \sum_{\substack{k \in N, h \in N: (i,k) \in A \\(k,h) \in A^{R,out}}} z^{out}_{kh} + \sum_{\substack{l \in N: (i,l) \in A\\ (\alpha, l) \in A^{R,out}}} z_{\alpha l}^{out} \le c_i^{out} & \text{ for all } i \in T \label{con:trafBudget}
\end{align}
}

We additionally include constraints that limit the number of interactions with each victim (including prospective victims), preventing them from belonging to too many trafficking operations.

\small{
\begin{align}
    & \sum_{(i,j) \in A^{R,out}} z^{out}_{ij} + \sum_{(i,j) \in A^{R,in}} z^{in}_{ij} \le c_j^{in} & \text{ for } j \in V \cup V^R \label{con:vicBudget}
\end{align}
}

For trafficking operations that have a back-up trafficker, we include constraints indicating when a back-up trafficker can replace an interdicted trafficker.

\small{
\begin{align}
    & \zeta_{\alpha j t}^{out} = 0 & \text{ for } (i,j) \in T^R, t \in 1,\ldots, (\delta^y_i + \delta^{z, out}_{\alpha j}) \label{con:noBackup}\\
    & \zeta_{\alpha j t}^{out} = \bar{y}_i & \text{ for } (i,j) \in T^R, t = (\delta^y_i + \delta^{z, out}_{\alpha j} + 1), \ldots, \tau \label{con:yesBackup}
\end{align}
}

Constraints \eqref{con:noBackup} enforce that the arc representing the back-up trafficker will be offline for a number of time periods equal to the time to interdict the primary trafficker and restructure to the back-up trafficker. Constraints \eqref{con:yesBackup} allow for the arc to come online if the trafficker has been interdicted. 

The next set of constraints enforces when victims are allowed to be promoted to be a bottom.

\small{
\begin{align}
    & \zeta^{out}_{\alpha j t} = 0 & \text{ for } (i,j) \in B^R, t = 1, \ldots, (\delta^y_i + \delta^{z, out}_{\alpha j}) \label{con:noPromote}\\
    & \sum_{(i,j) \in B^R} \zeta^{out}_{\alpha j (t+\delta^{z, out}_{\alpha j})} \le \bar{y}_i & \text{ for } i \in B, t = (\delta_i^y+1),\ldots, \tau \label{con:yesPromote}\\
    & z^{out}_{\alpha j t} \le 1 - \bar{y}_j & \text{ for } j \in V \text{ s.t. } \exists i \in B, (i,j) \in B^R \label{con:intPromote}\\
    & \sum_{h \in V: (j,h) \in A^{R,out}} z^{out}_{jh} \le |V|z^{out}_{\alpha j} & \text{ for } j \in V \text{ s.t. } \exists i \in B, (i,j) \in B^R \label{con:vicNewBot}
\end{align}
}

Constraints \eqref{con:noPromote}-\eqref{con:yesPromote} act similarly to constraints \eqref{con:noBackup}-\eqref{con:yesBackup}, allowing a victim to be promoted to a bottom if the current bottom of the operations is interdicted and after the requisite time has passed. Constraints \eqref{con:intPromote} prevent an interdicted victim from being promoted to be a bottom. Constraints \eqref{con:vicNewBot} allows for a trafficker to assign some of their victims to the newly promoted bottom once they have been promoted. 

The last constraints enforce that any arc in $A^{R,out} \cap A^{R,in}$ cannot have both $z^{out}$ and $z^{in}$ be nonzero, preventing an arc from being restructured twice.

\small{
\begin{align}
    & z^{out}_{ij} + z^{in}_{ij} \le 1 & \text{ for all } (i,j) \in A^{R,out} \cap A^{R,in} \label{con:noOverlap}
\end{align}
}

\section{Model Derivation}
\label{sec:modelderiv}
We now describe the integer programming formulation of the multi-period max flow network interdiction problem with restructuring (MP-MFNIP-R) including the constraints on interdictions, and derive the column-and-constraint algorithm to solve it. The following program incorporates the constraints defined in Section \ref{sec:modelInt} and $Z(y)$ is defined by the constraints introduced in Section \ref{sec:modelRest}. Since the inclusion of a second attacker does not impact the model derivation, we only discuss the derivation of the model with a single attacker in this section. We will discuss the changes in the model by including the second attacker after presenting the final model.

\begin{singlespace}
\begin{subequations}
\label{minmax2}
\footnotesize\begin{align} 
    \min_{y, \Tilde{r}} \max_{x, z, \zeta} ~~~ & \sum_{t = 1}^{\tau} \sum_{i \in N: (\alpha,i) \in A \cup A^{R,out}} x_{\alpha it} & \nonumber\\
    \text{s.t. }& \sum_{(h,i) \in A \cup A^{R,out}} x_{hit} = x_{it} & \text{ for } i \in N \setminus \alpha, t = 1,\ldots,\tau \label{con:inflow}\\
    & x_{it} = \sum_{(i,h) \in A \cup A^{R,out}} x_{iht} & \text{ for } i \in N \setminus \omega, t = 1,\ldots,\tau \label{con:outflow}\\
    & 0 \le  x_{ijt} \le u_{ij} & \text{ for } (i,j) \in A, t = 1,\ldots,\tau \label{con:arcCap}\\
    & 0 \le  x_{ijt} \le u_{ij} \zeta^{out}_{ij} & \text{ for } (i,j) \in A^{R,out} \setminus A^{R,in}, t = 1,\ldots,\tau \label{con:arcOutCap}\\
    & 0 \le  x_{ijt} \le u_{ij} (\zeta^{out}_{ijt}+\zeta^{in}_{ijt}) & \text{ for } (i,j) \in A^{R,in}, t = 1,\ldots,\tau \label{con:arcInCap}\\
    & 0 \le  x_{it} \le u_i(1 - \gamma_{it}) & \text{ for } i \in N\setminus\{j \in V: \exists h \in B \text{ with } (h,j) \in B^R \}, t = 1,\ldots,\tau \label{con:nodecap}\\
    & 0 \le  x_{jt} \le u_j(1 - \gamma_{jt}) + \Tilde{u_j} \zeta^{out}_{\alpha jt} & \text{ for } j \in V \text{s.t.} \exists i \in B, (i,j) \in B^R, t = 1,\ldots,\tau \label{con:bottomCap}\\
    &\sum_{i \in T} \Tilde{r}_i y_i + \sum_{j \in B \cup T} r_i y_i \le b & \label{con:intBudgetInModel}\\
    &\Tilde{r}_i \ge r_i^{min} & \text{ for } i \in T\\
    &\Tilde{r}_i \ge r_i - \sum_{l \in B \cup V} d_{il} y_i & \text{ for } i \in T \\
    & \gamma_{it} = 0 & \text{ for } i \in N, t = 1,\ldots,\delta_i^y \\
    & \gamma_{it} = y_i & \text{ for } i \in N, t \in (\delta_i^y+1), \ldots, \tau \label{con:intTimeInModel}\\
    & y \in \{0,1\}^{|N|\setminus \{\alpha,\omega\}} & \\
    & \gamma \in \{0,1\}^{|N|\setminus \{\alpha,\omega\} \times \tau} & \\
    & z, \zeta \in Z(y)
\end{align}
\end{subequations}
\end{singlespace}

We follow the procedure outlined in \cite{kosmas2020interdicting} to derive a single-level minimization problem that can be solved as part of column-and-constraint generation with partial information. For the sake of brevity, we include the full model derivation in Appendix \ref{fullderiv}, and only state the final master problem that is solved.  

We first describe variables associated with the standard column-and-constraint generation procedure. Let $\eta$ be the variable representative of the objective value of the bilevel optimization problem, and let $n$ be the number of restructurings plans being considered in the optimization model. Let $\pi^{+^k}_{it}$ and $\pi^{-^k}_{it}$ be the dual variables associated with constraints \eqref{con:inflow} and \eqref{con:outflow}, respectively, for node $i$ at time $t$ for the $k^{th}$ restructuring plan. Let $\theta_{ijt}^k$ be the dual variables associated with constraints \eqref{con:arcCap}-\eqref{con:arcInCap} for arc $(i,j)$ at time $t$ for the $k^{th}$ restructuring plan, and let $\theta^k_{it}$ be the dual variable associated with constraints \eqref{con:nodecap}-\eqref{con:bottomCap} for node $i$ at time $t$ for the $k^{th}$ restructuring plan.

Now we describe parameters and variables specific to the partial information adaptation, which depends on each previously considered restructuring plan $z^k$. Let $z^{out,k}_{ij}$ be the parameter indicating if arc $(i,j)$ was restructured by node $i$ in the $k^{th}$ restructuring plan, and let $z^{in,k}_{ij}$ be the parameter indicating if arc $(i,j)$ was restructured by node $j$ in the $k^{th}$ restructuring plan. Let $w^{out,k}_{ijt}$ be the indicator of if arc $(i,j)$ can be restructured out of node $i$ in time period $t$ for restructuring plan $k$, and let $w^{in,k}_{ijt}$ be the indicator of if arc $(i,j)$ can be restructured in from node $j$ in time period $t$ for restructuring plan $k$. Let $\lambda^{out,k}_{it}$ be the indicator of if all of the ``out" restructurings for trafficker $i$ in time period $t$ for restructuring plan $k$ have been performed, and let $\lambda^{in,k}_{it}$ be the indicator of if all of the ``in" restructurings for victim $i$ in time period $t$ for restructuring plan $k$ have been performed.

\begin{singlespace}
\begin{subequations}
\label{tlfFin}
\footnotesize\begin{align}
    \min_{y, \gamma, w \pi,\theta} ~~~ & \eta \nonumber \\
    \text{s.t.} \hspace{.050cm } & \eta \ge \sum_{t = 1}^{\tau} [\sum_{i \in N \setminus \{\alpha,\omega\}} u_i \theta^k_{it} + \sum_{(i,j) \in A \cup A^{R,out} \cup A^{R,in}} u_{ij} \theta^k_{ijt} + \sum_{(i,j) \in B^R} \tilde{u_j} z^{out,k}_{\alpha jt} w^{out,k}_{\alpha jt} \theta^k_{\alpha jt} ]  \label{con:obj}\\
    & \omit\hfill$\text{ for } k = 1, \ldots, n \nonumber$\\  
    & \pi_{jt}^{+^k} + \theta^k_{\alpha jt} \ge 1  \label{con:sourceSide}\\
    & \omit\hfill$\text{ for } k = 1,\ldots,n, (\alpha,j) \in A, t = 1, \ldots, \tau$ \nonumber\\
    & \pi_{jt}^{+^k} - \pi_{it}^{-^k} + \theta^k_{ijt} \ge 0 \label{con:inCut}\\
    & \omit\hfill$\text{ for } k = 1,\ldots,n, (i,j) \in A \text{ s.t. } i \ne \alpha, j \ne \omega, t = 1, \ldots, \tau$ \nonumber\\
    & \pi_{it}^{-^k} - \pi_{it}^{+^k} + \theta^k_{it} \ge - \gamma_{it}   \label{con:nodeInCut}\\
    & \omit\hfill$\text{ for } k = 1,\ldots,n, i \in N\setminus \{\alpha,\omega\}, t = 1, \ldots, \tau$\nonumber\\
    & -\pi_{it}^{-^k} + \theta^k_{i\omega t} \ge 0 \label{con:sinkSide}\\
    & \omit\hfill$\text{ for } k = 1,\ldots,n, (i,\omega) \in A, t = 1, \ldots, \tau$ \nonumber\\
    & \pi_{jt}^{+^k} + \theta^k_{\alpha jt} \ge w^{out,k}_{\alpha jt} + z^{out,k}_{\alpha j} - 1 \label{con:sourceSinkNew}\\
    & \omit\hfill$\text{ for } k = 1,\ldots,n, (\alpha,j) \in A^{R, out}, t = 1, \ldots, \tau$\nonumber\\
    & \pi_{jt}^{+^k} - \pi_{it}^{-^k} + \theta^k_{ijt} \ge w^{out,k}_{ijt} + z^{out,k}_{ij} - 2  \label{con:inCutOut}\\
    & \omit\hfill$\text{ for } k = 1,\ldots,n, (i,j) \in  A^{R, out} \text{ s.t. } i \ne \alpha, j \ne \omega, t = 1, \ldots, \tau$\nonumber\\
    & \pi_{jt}^{+^k} - \pi_{it}^{-^k} + \theta^k_{ijt} \ge w^{in,k}_{ijt} + z^{in,k}_{ij} - 2  \label{con:inCutIn}\\
    & \omit\hfill$\text{ for } k = 1,\ldots,n, (i,j) \in  A^{R, in} \text{ s.t. } i \ne \alpha, j \ne \omega, t = 1, \ldots, \tau$\nonumber\\
    & \theta^k \ge 0 \\
    & \omit\hfill$\text{ for } k = 1, \ldots, n \nonumber$\\ 
    & \mu^{out}_i \lambda^{out,k}_{it} + \sum_{(i,j) \in A^{R,out}: z^{out,k}_{ij}=1} w^{out,k}_{ij(t+\delta^{z, out}_{ij})} \ge \sum_{(i,h) \in A} \gamma_{ht} \label{con:PIOut1}\\
    & \omit\hfill$\text{ for } k \in 1, \ldots, n, i \in T, t \in 1, \ldots, \tau-\delta^{max,out}_i$ \nonumber\\
    & \mu^{in}_j \lambda^{in,k}_{jt} + \sum_{(i,j) \in A^{R,in}:  z^{in,k}_{ij}=1} w^{in,k}_{ij(t+ \delta^{z,in}_{ij})} \ge \sum_{(h,j) \in A} \gamma_{ht} \label{con:PIIn1}\\
    & \omit\hfill$\text{ for } k \in 1, \ldots, n, j \in V, t \in 1, \ldots, \tau-\delta^{max,in}_i$ \nonumber\\
    & \lambda^{out,k}_{it} \le \frac{\sum_{(i,j) \in A^{R,out}:  z^{out,k}_{ij}=1} w^{out,k}_{ij(t+\delta^{z,out}_{ij})}}{\sum_{(i,j) \in A^{R,out}:  z^{out.k}_{ij}=1} z^{out,k}_{ij}} \label{con:PIOut2}\\
    & \omit\hfill$\text{ for } k \in 1, \ldots, n, i \in T, t = 1, \ldots, \tau$ \nonumber\\
    & \lambda^{in,k}_{it} \le \frac{\sum_{(i,j) \in A^{R,in}: z^{in,k}_{ij}=1} w^{in,k}_{ij(t+\delta^{z,in}_{ij})}}{\sum_{(i,j) \in A^{R,in}:  z^{in,k}_{ij}=1} z^{in,k}_{ij}} \label{con:PIIn2}\\
    & \omit\hfill$\text{ for } k = 1, \ldots, n, i \in V, t = 1, \ldots, \tau$ \nonumber\\
    & w^{out,k}_{ij} \ge w^{out,k}_{ij(t-1)} \label{con:subseqOutW}\\
    & \omit\hfill$\text{ for } k = 1, \ldots, n, (i,j) \in A^{R,out}, t \in 2 \ldots, \tau$  \nonumber\\
    & w^{in,k}_{ijt} \ge w^{in,k}_{ij(t-1)} \label{con:subseqInW}\\
    & \omit\hfill$\text{ for } k = 1, \ldots, n, (i,j) \in A^{R,in}, t \in 2 \ldots, \tau$ \nonumber\\
    & w^{out,k}_{\alpha j t} \ge y_{i} \label{con:PIbackupTraf}\\
    & \omit\hfill$\text{ for } k = 1,\ldots, n,  (j,i) \in T^R, t \in \delta^y_{i} + \delta^{z,out}_{\alpha j}, \ldots, \tau$ \nonumber\\
    & w^{out,k}_{\alpha j (t + \delta^{z,out}_{\alpha j})} \ge y_i \label{con:PInewBottom}\\
    & \omit\hfill$\text{ for } k = 1, \ldots, n, (j,i) \in B^R \text{ s.t. } z^{out,k}_{\alpha j}=1, t \in \delta^y_i+1, \ldots, \tau - \delta^{z,out}_{\alpha j}$ \nonumber\\
    & w^{out,k}_{ijt} \ge z^{out,k}_{ij} \label{con:freeRes1}\\
    & \omit\hfill$\text{ for } k = 1, \ldots, n, i \in B, j \in V \text{ s.t. } (i,j) \in A^{R,out}$ \nonumber\\
    & w^{out,k}_{ijt} \ge z^{out,k}_{ij} \label{con:freeRes2}\\
    & \omit\hfill$\text{ for } k = 1, \ldots, n, i \in V, j \in V \text{ s.t. } (i,j) \in A^{R,out}, \exists l \in B, (i,l) \in B^R$  \nonumber\\
    &\sum_{i \in T} \Tilde{r}_i y_i + \sum_{j \in B \cup T} r_i y_i \le b  \\
    &\Tilde{r}_i \ge r_i^{min} \\
    & \omit\hfill$\text{ for } i \in T$\nonumber\\
    &\Tilde{r}_i \ge r_i - \sum_{l \in B \cup V} d_{il} y_i \\
    &\omit\hfill$\text{ for } i \in T$ \nonumber\\
    & \gamma_{it} = 0 \\
    & \omit\hfill$\text{ for } i \in N, t = 1,\ldots,\delta_i^y$ \nonumber\\
    & \gamma_{it} = y_i \label{intConUpdate}\\ 
    & \omit\hfill$\text{ for } i \in N, t \in (\delta_i^y+1), \ldots, \tau$ \nonumber\\
    & y \in \{0,1\}^{|N|\setminus \{\alpha,\omega\}} \\
    & \gamma \in \{0,1\}^{|N|\setminus \{\alpha,\omega\} \times \tau}
\end{align}
\end{subequations}
\end{singlespace}

We focus on describing the constraints regarding partial information, since the other constraints are a formulation of the minimum cut problem or are the constraints regarding interdiction. Constraints \eqref{con:PIOut1}-\eqref{con:PIIn1} allow for restructurings from the $k^{th}$ restructuring plan to be implemented after interdictions that allow for them have occurred. When a node is allowed to restructure more arcs than the number of arcs that node restructured in the $k^{th}$ restructuring plan, constraints \eqref{con:PIOut2}-\eqref{con:PIIn2} allow for $\lambda = 1$ to ensure the feasibility of constraints \eqref{con:PIOut1}-\eqref{con:PIIn1}. Constraints \eqref{con:subseqOutW}-\eqref{con:subseqInW} enforce that an arc is online in time periods after the time period it was restructured in. Constraints \eqref{con:PIbackupTraf} enforce that a back-up trafficker is restructured to when the primary trafficker is interdicted. Likewise, constraints \eqref{con:PInewBottom} enforce that the victim who was promoted to be the bottom is still promoted if the current bottom is interdicted. Constraints \eqref{con:freeRes1}-\eqref{con:freeRes2} enforce that any restructured arcs that are independent of interdictions are considered to be restructured. We note that there are bilinear terms $w^{out,k}_{\alpha jt} \theta^k_{\alpha jt}$ in the objective function constraints, and these terms can be linearized using the McCormick inequalities \citep{mccormick1976computability}.

Recall from Section \ref{sec:modelInt}, minor adjustments to \eqref{tlfFin} need to be made to incorporate the second attacker. We would additionally include constraint \eqref{con:2ndBudget}, and constraints \eqref{intConUpdate} would be changed to $\gamma_{it} = y_i + y'_i$ for any nodes $i \in V \cup V^R$, which are the nodes that the second attacker is able to interdict.

\sloppy
After solving \eqref{tlfFin}, we identify optimal interdiction decisions $\bar{y}$ for known restructuring plans $(z^1,\zeta^1,\ldots, z^n, \zeta^n)$, as well as a lower bound on the objective value of \eqref{minmax2}, $\bar{\eta}$. We then need to determine the optimal restructuring plan responding to $(\bar{y},\bar{\gamma})$. To do so, we solve the defender's problem using $(\bar{y},\bar{\gamma})$ as data.

\begin{singlespace}
\begin{subequations}
\label{subprob}
\footnotesize\begin{align} 
    \max_{x, z, \zeta} ~~~ & \sum_{t = 1}^{\tau} \sum_{i \in N: (\alpha,i) \in A \cup A^{R,out}} x_{\alpha it} & \nonumber\\
    \text{s.t. }& \sum_{(h,i) \in A \cup A^{R,out}} x_{hit} = x_{it} & \text{ for } i \in N \setminus \alpha, t = 1,\ldots,\tau\\
    & x_{it} = \sum_{(i,h) \in A \cup A^{R,out}} x_{iht} & \text{ for } i \in N \setminus \omega, t = 1,\ldots,\tau \\
    & 0 \le  x_{ijt} \le u_{ij} & \text{ for } (i,j) \in A, t = 1,\ldots,\tau \\
    & 0 \le  x_{ijt} \le u_{ij} \zeta^{out}_{ij} & \text{ for } (i,j) \in A^{R,out} \setminus A^{R,in}, t = 1,\ldots,\tau \\
    & 0 \le  x_{ijt} \le u_{ij} (\zeta^{out}_{ijt}+\zeta^{in}_{ijt}) & \text{ for } (i,j) \in A^{R,in}, t = 1,\ldots,\tau \\
    & 0 \le  x_{it} \le u_i(1 - \bar{\gamma}_{it}) & \text{ for } i \in N\setminus\{j \in V: \exists h \in B \text{ with } (h,j) \in B^R \}, t = 1,\ldots,\tau \\
    & 0 \le  x_{jt} \le u_j(1 - \bar{\gamma}_{jt}) + \Tilde{u_j} \zeta^{out}_{\alpha jt} & \text{ for } j \in V \text{s.t.} \exists i \in B, (i,j) \in B^R, t = 1,\ldots,\tau \\
    & z, \zeta \in Z(\bar{y})
\end{align}
\end{subequations}
\end{singlespace}

Solving \eqref{subprob} provides the optimal restructuring plan $\bar{z}$ responding to $\bar{y}$, as well as an upper bound on the objective value of \eqref{minmax2}. Let $U = \min \{U, \sum_{t = 1}^{\tau} \sum_{i \in N: (\alpha,i) \in A \cup A^{R,out}} x_{\alpha it}\} $ be the upper bound on \eqref{minmax2}, let $L = \bar{\eta}$ be the lower bound on \eqref{minmax2}, and let $\epsilon \ge 0$ be the desired error tolerance. If $U - L \le \epsilon$, then we have identified the desired solution, and the algorithm will terminate. Otherwise, we set $z^{n+1} = \bar{z}$, and repeat the process. We formalize this in Algorithm \ref{alg:cncg_MMFNIPR}.

\begin{algorithm}[h]
\caption{C\&CG for MP-MFNIP-R}
\label{alg:cncg_MMFNIPR}
\begin{algorithmic}
\STATE{\textbf{Initialize:}} lower bound $L = - \infty$, upper bound $U = \infty$, optimal interdiction decisions $(y^*, \gamma^*)=0$, optimal restructuring decisions $(z^1,\zeta^1) = (z^*,\zeta^*) = 0$, error tolerance $\epsilon \ge 0$. iteration counter $m=1$.
 \WHILE{$U - L > \epsilon$}
\STATE{\textbf{Step 1.}} Solve \eqref{tlfFin} for optimal interdiction decision $(y^n,\gamma^n)$ and objective value $\eta$. Set $L = \eta$.
\IF{$U - L \le \epsilon$}
 \STATE{terminate; $(y^*, \gamma^*)$ and $(z^*,\zeta^*)$ are the optimal decisions with objective value $U$.}
\ENDIF
\STATE{\textbf{Step 2.}} Input $(y^*, \gamma^*)$ into \eqref{subprob} as data and solve \eqref{subprob} for optimal restructuring decisions $(z^{n+1},\zeta^{n+1})$ and objective value $\sum_{t = 1}^{\tau} \sum_{i \in N: (\alpha,i) \in A \cup A^{R,out}} x_{\alpha it}$.
\IF{$\sum_{t = 1}^{\tau} \sum_{i \in N: (\alpha,i) \in A \cup A^{R,out}} x_{\alpha it} < U$}
 \STATE{Let $y^* = y^n$, $\gamma^*=\gamma^n$, $z^* = z^{n+1}$, $\zeta^* = \zeta^{n+1}$, $U = \sum_{t = 1}^{\tau} \sum_{i \in N: (\alpha,i) \in A \cup A^{R,out}} x_{\alpha it}$}.
\ENDIF
\STATE{\textbf{Step 3.}} Include constraints corresponding to $(z^{n+1},\zeta^{n+1})$ in \eqref{tlfFin}, create variables $\pi^{n+1}$, $\theta^{n+1}$, $w^{n+1}$, set $n = n + 1$, return to Step 1. 
\ENDWHILE
\end{algorithmic}
\end{algorithm}

\subsection{Model Simplifications: Upfront Interdiction}
We note that, if $\delta_i^y = 0$ for all $i \in N$, meaning that all interdictions are implemented before the network is operated by the defender, we can reduce the number of variables in the constraints. We no longer need $\gamma$, as a node will either be online or offline for the entire time horizon. Similarly, we no longer need to index $w^{out}, w^{in}$ by time, as a restructured arc $(i,j)$ will be in the network in time period $\delta_{ij}^z + 1$. This is reflected by only including the equivalent constraints for \eqref{con:PIOut1}-\eqref{con:PIIn2} from $\delta_{ij}^{z,out}+1$ (or $\delta_{ij}^{z,in}$, respectively) to $\tau$. This will only enforce the relationship between whether or not the arc is in the cut (if the arc is restructured in the network) and which side of the cut the nodes are on in the time periods that the arc will appear in. The following model incorporates this modeling simplification.

\begin{singlespace}
\begin{subequations}
\label{tlfUpFront}
\footnotesize\begin{align}
    \min_{y, \pi,\theta} ~~~ \eta \nonumber\\
    \text{s.t.} \hspace{.050cm }& \eta \ge \sum_{t = 1}^{\tau} \sum_{i \in N \setminus \{\alpha,\omega\}} u_i \theta^k_{it} + \sum_{(i,j) \in A \cup A^{R,out} \cup A^{R,in}} u_{ij} \theta^k_{ijt} + \sum_{(i,j) \in B^R} \tilde{u_j} z^{out,k}_{\alpha jt} w^{out,k}_{\alpha j} \theta^k_{\alpha j t} \\
    & \omit\hfill$\text{ for } k = 1, \ldots, n$\nonumber\\  
    &\pi_{jt}^{+^k} + \theta^k_{\alpha jt} \ge 1 \\
    & \omit\hfill$\text{ for } k = 1, \ldots, n, (\alpha,j) \in A, t = 1, \ldots, \tau$\nonumber\\
    & \pi_{jt}^{+^k} - \pi_{it}^{-^k} + \theta^k_{ijt} \ge 0 \\
    & \omit\hfill$\text{ for } (i,j) \in A \text{ s.t. } k = 1, \ldots, n, i \ne \alpha, j \ne \omega, t = 1, \ldots, \tau$\\
    & \pi_{it}^{-^k} - \pi_{it}^{+^k} + \theta^k_{it} \ge - y_{i} \\
    & \omit\hfill$\text{ for } k = 1, \ldots, n, i \in N\setminus \{\alpha,\omega\}, t = 1, \ldots, \tau$\nonumber\\
    & -\pi_{it}^{-^k} + \theta^k_{i\omega t} \ge 0 \\
    & \omit\hfill$\text{ for } k = 1, \ldots, n, (i,\omega) \in A, t = 1, \ldots, \tau$\nonumber\\
    & \pi_{jt}^{+^k} + \theta^k_{\alpha jt} \ge w^{out,k}_{\alpha j} + z^{out,k}_{\alpha jt} - 1 \\
    & \omit\hfill$\text{ for } k = 1, \ldots, n, (\alpha,j) \in A^{R, out}, t = 1, \ldots, \tau$\nonumber\\
    & \pi_{jt}^{+^k} - \pi_{it}^{-^k} + \theta^k_{ijt} \ge w^{out,k}_{ij} + z^{out,k}_{ijt} - 2 \\
    & \omit\hfill$\text{ for } (i,j) \in  A^{R, out} \text{ s.t. } k = 1, \ldots, n, i \ne \alpha, j \ne \omega, t = 1, \ldots, \tau$\nonumber\\
    & \pi_{jt}^{+^k} - \pi_{it}^{-^k} + \theta^k_{ijt} \ge w^{in,k}_{ij} + z^{in,k}_{ijt} - 2 \\
    & \omit\hfill$\text{ for } (i,j) \in  A^{R, in} \text{ s.t. } k = 1, \ldots, n, i \ne \alpha, j \ne \omega, t = 1, \ldots, \tau$\nonumber\\
    & \theta^k \ge 0 \\
    & \omit\hfill$\text{ for } k = 1, \ldots, n \nonumber$\\ 
    & \text{Constraints } \eqref{con:PIOut1} - \eqref{con:freeRes2}\nonumber\\
    &\sum_{i \in T} \Tilde{r}_i y_i + \sum_{j \in B \cup T} r_i y_i \le b  \\
    &\Tilde{r}_i \ge r_i^{min} \\
    & \omit\hfill$\text{ for } i \in T$\nonumber\\
    &\Tilde{r}_i \ge r_i - \sum_{l \in B \cup V} d_{il} y_i \\
    & \omit\hfill$\text{ for } i \in T $\nonumber\\
    & y \in \{0,1\}^{|N|\setminus \{\alpha,\omega\}} 
\end{align}
\end{subequations}
\end{singlespace}

Furthermore, we are able to reduce the size of the overall problem due to upfront interdiction. Since the time periods that the restructurings occur in are no longer dependent on the time periods of the interdictions that allowed for those restructurings to occur in, we can divide the time horizon into network ``phases" based on when multiple sequential time periods do not have any new restructurings. Given the $k^th$ restructuring plan $z^k$, let $t_1^k, \ldots, t_f^k$ be the \emph{unique} time periods that restructurings in $z^k$ occur in. We can define network phase $s$ as $G^{sk} = (N(G), A(G) \cup \{(i,j) \in A^{R, out}: z^{out,k}_{ij} = 1, \delta_{ij}^z \le t_i^k\} \cup \{(i,j) \in A^{R, out}: z^{in,k}_{ij} = 1, \delta_{ij}^z \le t_i^k\})$. We then need to divide the time horizon $\{1, \ldots, \tau\}$ into $S^k$ intervals in which we are in the different network phases. Let $\tau_s^k = t_1^k - 1$ for $s=1$, $\tau_s^k = t_{s+1}^k - t_s^k$ for $s = 2, \ldots, t_f^k$, and $\tau_s^k = \tau - t_f^k - 1$ for $s = S^k$.

\begin{singlespace}
\begin{subequations}
\label{tlfphase}
\footnotesize\begin{align}
    \min_{y, \pi,\theta} ~~~ \eta \nonumber\\
    \text{s.t.} \hspace{.050cm }& \eta \ge \sum_{s = 1}^{S^k} \tau_s^k [\sum_{i \in N \setminus \{\alpha,\omega\}} u_i \theta^k_{is} + \sum_{(i,j) \in A \cup A^{R,out} \cup A^{R,in}} u_{ij} \theta^k_{ijs} + \sum_{(i,j) \in B^R} \tilde{u_j} z^{out,k}_{\alpha js} w^{out,k}_{\alpha j} \theta^k_{\alpha j} ] \\
    & \omit\hfill$\text{ for } k = 1, \ldots, n$\nonumber\\  
    &\pi_{js}^{+^k} + \theta^k_{\alpha js} \ge 1 \\
    & \omit\hfill$\text{ for } k = 1, \ldots, n, (\alpha,j) \in A, s = 1, \ldots, S^k$\nonumber\\
    & \pi_{js}^{+^k} - \pi_{is}^{-^k} + \theta^k_{ijs} \ge 0 \\
    & \omit\hfill$\text{ for } (i,j) \in A \text{ s.t. } k = 1, \ldots, n, i \ne \alpha, j \ne \omega, s = 1, \ldots, S^k$\nonumber\\
    & \pi_{is}^{-^k} - \pi_{is}^{+^k} + \theta^k_{is} \ge - y_{i} \\
    & \omit\hfill$\text{ for } k = 1, \ldots, n, i \in N\setminus \{\alpha,\omega\}, s = 1, \ldots, S^k $\nonumber\\
    & -\pi_{is}^{-^k} + \theta^k_{i\omega s} \ge 0 \\
    & \omit\hfill$\text{ for } k = 1, \ldots, n, (i,\omega) \in A, s = 1, \ldots, S^k$\nonumber\\
    & \pi_{js}^{+^k} + \theta^k_{\alpha js} \ge w^{out,k}_{\alpha j} + z^{out,k}_{\alpha js} - 1 \\
    & \omit\hfill$\text{ for } k = 1, \ldots, n, (\alpha,j) \in A^{R, out}, s = 1, \ldots, S^k$\nonumber\\
    & \pi_{js}^{+^k} - \pi_{is}^{-^k} + \theta^k_{ijs} \ge w^{out,k}_{ij} + z^{out,k}_{ijs} - 2 \\
    & \omit\hfill$\text{ for } (i,j) \in  A^{R, out} \text{ s.t. } k = 1, \ldots, n, i \ne \alpha, j \ne \omega, s = 1, \ldots, S^k$\nonumber\\
    & \pi_{js}^{+^k} - \pi_{is}^{-^k} + \theta^k_{ijs} \ge w^{in,k}_{ij} + z^{in,k}_{ijs} - 2 \\
    & \omit\hfill$\text{ for } (i,j) \in  A^{R, in} \text{ s.t. } k = 1, \ldots, n, i \ne \alpha, j \ne \omega, s = 1, \ldots, S^k$\nonumber\\
    & \theta^k \ge 0 \\
    & \omit\hfill$\text{ for } k = 1, \ldots, n \nonumber$\\ 
    & \text{Constraints } \eqref{con:PIOut1} - \eqref{con:freeRes2}\nonumber\\
    &\sum_{i \in T} \Tilde{r}_i y_i + \sum_{j \in B \cup T} r_i y_i \le b \\
    &\Tilde{r}_i \ge r_i^{min} \\
    & \omit\hfill$\text{ for } i \in T$\nonumber\\
    &\Tilde{r}_i \ge r_i - \sum_{l \in B \cup V} d_{il} y_i \\
    & \omit\hfill$\text{ for } i \in T $\nonumber\\
    & y \in \{0,1\}^{|N|\setminus \{\alpha,\omega\}} 
\end{align}
\end{subequations}
\end{singlespace}

\section{Modeling-Based Augmentations}
\label{sec:aug}
We seek to improve our C\&CG algorithm with modeling-based improvements. The first improvement is to augment previously visited restructuring plans. We do so by including additional arcs as ``previously restructured" while still maintaining feasibility.  This approach helps to better take advantage of the way partial information about previously visited restructured plans provide bounds on the current interdiction decisions. Note that, as long as a restructuring plan satisfies the constraints that are independent of $y$, then the $w$ variables will ensure that the constraints that are dependent on $y$ are satisfied, producing a feasible restructuring plan. As such, any previously visited restructuring plan can be augmented to include additional arcs \emph{as long as} the constraints independent of $y$ are satisfied. 

We provide a simple, yet effective augmentation. Every previously visited restructuring plan (including the initial iteration: the empty restructuring plan) is augmented to include every back-up trafficker, whether or not the back-up traffickers were restructured to. That is, for every $k = 1, \ldots, n$, $(j,i) \in T^R$, $z^{out,k}_{\alpha j} = 1$. Note that the only constraints on whether or not a back-up trafficker can be restructured are the constraints indicating whether or not the primary trafficker has been interdicted. Thus, this change will not impact the feasibility of any constraints independent of $y$. Since, for $(i,j) \in T^R$, $w^{out,k}_{\alpha j t} = 0$ for all $t$ if $y_i =0$, and $w^{out,k}_{\alpha j t} = 1$ for $t \ge \delta_i^y + \delta_{\alpha j}^{z,out} + 1$ if $y_i = 1$, the resulting restructuring sub plan will be feasible. This augmentation will allow the master problem to better project what the impact on flow will be if a trafficker is interdicted, which will in turn produce better lower bounds when traffickers are interdicted.

\sloppy The next improvement we propose is to project how new victims may be recruited. For each previously considered restructuring plan, we can identify if a prospective victim $j$ has been recruited or not by computing $\sum_{i \in T: (i,j) \in A^{R,out}} z^{out,k}_{ij}$. If $\sum_{i \in T: (i,j) \in A^{R,out}} z^{out,k}_{ij} = 1$, then prospective victim $j$ has been recruited in restructuring plan $k$. Otherwise, $\sum_{i \in T: (i,j) \in A^{R,out}} z^{out,k}_{ij} = 0$, meaning that prospective victim $j$ has not been recruited. We can thus identify the set of prospective victims that have not yet been recruited as $P^k = \{j \in V^R: \sum_{i \in T: (i,j) \in A^{R,out}} z^{out,k}_{ij} = 0\}$. Additionally, we can identify the \emph{latest} a victim can be recruited as $\bar{t} = \max_{i \in B \cup V} \delta_i^y + \max_{(i,j) \in A^{R,out}: j \in P^k} \delta_{ij}^{z,out}$. Likewise, we can identify which traffickers still have the ability to act (recall that constraints \eqref{con:trafBudget} restrict the number of acts each trafficker can take, which is independent of the interdiction decisions). Let $a_i^k = \sum_{j \in N: (i,j) \in A_R^{out}} z^{out}_{ij} + \sum_{\substack{l \in N, h \in N: (i,l) \in A \\(l,h) \in A_R^{out}}} z^{out}_{lh} + \sum_{\substack{l \in N: (i,l) \in A\\ (\alpha, l) \in A_R^{out}}} z_{\alpha l}^{out}$ be the number of actions taken by trafficker $i$ in restructuring plan $k$. Note that if $a_i^k =  c_i^{out}$, then trafficker $i$ cannot take any more actions in restructuring plan $k$. However, if $a_i^k < c_i^{out}$, then trafficker $i$ may still be able to recruit new victims. Let $C^k = \{i \in T: a_i^k < c_i^{out}\}$ be the set of traffickers that can still recruit in restructuring plan $k$. From these two sets, we can identify opportunities for recruitment. Let $A^{k,rec} = \{(i,j) \in A^{R,out}: i \in C^k, j \in P^k\}$ be the set of restructurable arcs $(i, j)$ where a trafficker $i$ that can still act recruits prospective victim $j$. For $k = 1, \ldots, n$, $(i,j) \in A^{k, rec}$, let $\psi_{ijt}^k$ be the indicator variable as to whether or not $(i,j)$ may be restructured in addition to restructuring plan $k$ in time period $t$. We want to define constraints that enforce that, if more interdictions have occurred that would allow for more restructurings to occur, arcs in $A^{k, rec}$ can additionally be restructured to augment that restructuring plan. In order to do so, we need to define two additional variables to ensure feasibility with big-$M$ style constraints. Let $\nu_{ki}$ be the indicator of whether or not trafficker $i$ has performed $c_i^{out}$ actions (including projected recruitment) in plan $k$, and let $\xi_{ki}$ be the indicator of whether or not trafficker $i$ has restructured all of the arcs they can in $A^{k,rec}$. We can define the following constraints: 

\begin{singlespace}
\begin{align}
    & B(\nu_{ki} + \xi_{ki}) + \sum_{j \in P^k: (i,j) \in A^{k,rec}} \psi_{ij\bar{t}}^k \ge \sum_{(i,j) \in A} y_j - \sum_{(i,j) \in A^{R,out}} w_{ij\bar{t}}^{out,k}  \label{con:ProjRec}\\ 
    & \omit\hfill$\text{ for } k \in 0,\ldots, n, i \in C^k $\nonumber\\
    & \sum_{j \in P^k: (i,j) \in A^{k,rec}} \psi_{ij\bar{t}} \le c_i^{out} - a_i^k   \label{con:TrafAct}\\ 
    & \omit\hfill$\text{ for } k \in 0,\ldots, n, i \in C^k$ \nonumber\\
    & \nu_{ki} \le \frac{1}{c_i^{out}}\left(\sum_{\substack{j \in N: (i,j) \in A_R^{out}\\ z^{out,k}_{ij}=1}} \bar{w}^{out,k}_{ij\tau} + \sum_{\substack{l \in N, h \in N: (i,l) \in A\\ (l,h) \in A_R^{out}\\ z^{out,k}_{lh}=1}} \bar{w}^{out,k}_{lh\tau} + \sum_{\substack{l \in N: (i,l) \in A\\ (\alpha, l) \in A_R^{out}\\ z^{out,k}_{\alpha l}=1}} \bar{w}_{\alpha l \tau}^{out,k} + \sum_{j \in P^k: (i,j) \in A^{k,rec}} \psi_{ij\bar{t}}^k\right)   \label{con:overflow1}\\ 
    & \omit\hfill$\text{ for } k = 1,\ldots, n, i \in C^k$ \nonumber\\ 
    & \xi_{ki} \le \frac{\sum_{j \in P^k: (i,j) \in A^{k,rec}} \psi_{ij\bar{t}}^k}{\left|j \in P^k: (i,j) \in A^{k,rec}\right|}  \label{con:overflow2}\\ 
    & \omit\hfill$\text{ for } k = 1,\ldots, n, i \in C^k$ \nonumber\\
    & \psi_{ijt}^k \ge \psi_{ij(t-1)}^k   \label{con:subseqProjRec}\\
    & \omit\hfill$\text{ for } (i,j) \in A^{k,rec}, t \in (\bar{t}+1),\ldots,\tau$ \nonumber
\end{align}
\end{singlespace}

Constraints \eqref{con:ProjRec} enforce that, for each previously considered restructuring plan, additional recruitment occurs when more interdictions occur than restructurings that were enabled by interdictions. Constraints \eqref{con:TrafAct} enforce that every trafficker takes at most $c_i^{out}$ actions. Constraints \eqref{con:overflow1} enforce that $c_i^{out}$ actions are taken before allowing $\nu_{ki} = 1$, which ensures the feasibility of \eqref{con:ProjRec} when more than $c_i^{out}$ interdictions disrupt nodes incident to trafficker $i$. Likewise, constraints \eqref{con:overflow2} enforce that all possible potential recruitment occurs before allowing $\xi_{ki} =1$, which ensures the feasibility of \eqref{con:ProjRec} when all $\psi_{ij\bar{t}}^k = 1$ for all $j \in P^k: (i,j) \in A^{k,rec}$, i.e., all potential recruitment opportunities are taken. Constraints \eqref{con:subseqProjRec} ensure that an arc stays online once it comes online.

We note that the introduction of these $\psi$ variables may cause constraints \eqref{con:vicBudget} to be infeasible. However, in our model, $u_j = 1$ for all $j \in V^R$. This allows us to show that, even if the constraints are violated, it will not impact the minimum cut in the network. First, we define what it means for a node to be in the minimum cut.

\begin{defn}
A node $i \in N$ is in the minimum cut at time $t$ if $\theta_{it}=1$.
\end{defn}

\begin{thrm}
    For every restructuring plan $k$, every prospective victim $v \in P^k$ and time period $t \ge \bar{t}$, at most one trafficker node $i$ that is not in the minimum cut will be brought into the minimum cut when $\psi^k_{ivt} = 1$, regardless of the number of trafficker nodes where $\psi^k_{ivt} = 1$.
\end{thrm}

\begin{proof}
Suppose that, with restructuring plan $k$ in time period $\bar{t}-1$, there are two non-interdicted trafficker nodes where $(i,v) \in A^{k, rec}$ and $(j,v) \in A^{k, rec}$, and $\psi^k_{ivt} = \psi^k_{jvt} = 1$. Let $f^k_{\bar{t}-1}$ be the maximum flow in through the network with restructuring plan $k$ at time period $\bar{t}-1$ as determined by the Ford-Fulkerson algorithm, and let $(\bar{\pi}^k_{\bar{t}-1}, \bar{\theta}^k_{\bar{t}-1})$ be the minimum cut in time period $\bar{t}-1$ as determined by the residual network $D_f$, and suppose that $\theta^k_{i(\bar{t}-1)}=\theta^k_{j(\bar{t}-1)}=0$. Thus, there are at least two $\alpha-\omega$ augmenting paths in time period $\bar{t}$, $(\alpha, i, v, \omega)$ and $(\alpha, j, v, \omega)$. Only one augmenting path through $v$ can be chosen, since $u_v=1$ will be amount of flow sent along the chosen path due to node and arc capacities being integral. Thus, if $j$ is not in the chosen path, then whether or not $j$ can be reached by $\alpha$ will not be impacted by the inclusion of arc $(j,v)$ after flow is sent along the augmenting path using $i$. Thus, at most one trafficker node will be included in the minimum cut after including arcs $(i,v)$ and $(j,v)$.
\end{proof}

Since at most one $i \in T$ will be included in the minimum cut, regardless of the number of arcs to recruitable victims that are ``included" via $\psi$ variables, we can construct a feasible restructuring plan. If there were to be a change in the minimum cut after ``including" an arc $(i,j)$ via $\psi^k_{ijt}=1$, then $(i,j)$ would be the arc that would be restructured. If no such change occurs, then any one arc can be chosen and the minimum cut will only increase by $1$, regardless of which arc is chosen. Removing the extra arcs introduced by $\psi$ will not impact the feasibility of any of the restructuring plans. We also note that, the only arcs that are added after the projected recruitment arcs will be restructuring to a back-up trafficker, promoting a victim to be the new bottom, and assigning the newly promoted bottom more victims. For a trafficker that was not interdicted, these arcs do not increase the total number of victims a trafficker can reach. Thus, if a victim is recruited into a given organization, no arcs restructured after their recruitment will cause them to be recruited into a different organization, and thus a feasible restructuring plan can be constructed. This augmentation can be implemented in both the standard model and simplified upfront interdiction models. To implement the prospective recruitment constraints in the network phase model, we would need to introduce an extra phase starting at $\bar{t}$ if no such phase where to exist.

\section{Computational Results}
\label{sec:results}
We implement the models derived for Algorithm \ref{alg:cncg_MMFNIPR} in AMPL with Gurobi 9.5 as the solver. Experiments were conducted on a laptop with an Intel\textsuperscript{\textregistered} Core\textsuperscript{TM} i5-8250 CPU @ 1.6 GHz - 1.8 GHz and 16 GB RAM running Windows 10. Each instance is limited to a run time of $2$ hours.

We test out model using $\tau = 7$ on the networks used in \cite{kosmas2022generating}. The networks used in \cite{kosmas2022generating} were the product of a network generator that was informed by qualitative literature, federal case file analysis, analyses of key stakeholder interviews and a secondary analysis of previously conducted interviews, and the generator was further validated by domain experts and a survivor-centered advisory board. Each network consists of $5$ single-trafficker operations. These networks are similar in size and structure to the networks presented in \cite{cockbain2018offender}. However, we note that this is an exploratory analysis, and that more analysis is needed to provide tailored recommendations for how different practitioners may want to apply our model. Table \ref{tab:netstats} reports the number of nodes, as well as the number of each type of node, in each network. The data used in this work is available from the corresponding author upon request.

\begin{table}[h]
    \centering
    \resizebox{\linewidth}{!}{\begin{tabular}{|c|c|c|c|c|}
        \hline
        Network & Number of Nodes & Number of Traffickers & Number of Bottoms & Number of Victims \\
        \hline
         1 & 28 & 5 & 3 & 20 \\
         \hline
         2 & 32 & 5 & 4 & 23 \\
         \hline
         3 & 35 & 5 & 5 & 25 \\
         \hline
         4 & 35 & 5 & 4 & 26\\
         \hline
         5 & 27 & 5 & 4 & 18\\
         \hline
    \end{tabular}}
    \caption{Sizes of generated networks from \cite{kosmas2022generating}}
    \label{tab:netstats}
\end{table}

As in \cite{kosmas2022generating}, we choose the cost to interdict a victim to be $2$, the cost to interdict a bottom to be $4$, and the cost to interdict a trafficker to be $8$. Interdicting a victim reduces the cost to interdict their trafficker by $1$, while interdicting a bottom reduces the cost by $3$, to a minimum of $4$. In the models with delayed interdiction, interdicting a victim takes $1$ time period, interdicting a bottom takes $2$ time periods, and interdicting a trafficker takes $3$ time periods. In models with a second attacker, the second attacker has a fixed budget of $10$, where interdicting a victim already in the trafficking network costs $3$, and preventing a victim from being recruited costs $1$.

Each trafficker may restructure up to $4$ arcs in their operation, while each victim may restructure $1$ arc if their trafficker is interdicted. Restructuring to a victim currently in the trafficking network takes $1$ time period, and recruiting a new victim takes $2$ time periods. A trafficker taking a victim from their bottom or assigning a victim to their bottom also takes $1$ time period. Trafficking operations with at least $4$ victims (including a bottom, if the operation has one) will have a back-up trafficker, and it takes $2$ time periods to restructure to them. For operations with a bottom, the set of victims that can be promoted to be the new bottom is determined by randomly selecting half of the victims the trafficker is adjacent to (rounding up), and promoting a victim to be a bottom takes $3$ time periods. We set the number of recruitable victims to be $40\%$ of the number of victims in the network, as was done in \cite{kosmas2022generating}.

\subsection{Comparison of Solution Methods}
We first compare the solve times of Algorithm \ref{alg:cncg_MMFNIPR} with and without the augmentations for models with a single attacker. Table \ref{tbl:compareTimeBaseD} compares the solve times for delayed interdictions. Table \ref{tbl:compareTimeBaseU} compares solves times for the base formulation of upfront interdiction, and Table \ref{tbl:compareTimeBaseUS} compares the solve times for the network phase formulation of upfront interdiction. In each table, we report the solve time, as well as the number of restructuring plans visited (in parentheses). We note that the number of plans visited is equivalent to the number of iterations of the C\&CG algorithm. Unsolved instances are marked with an $*$ and bold entries indicate which method solved the instance faster.

\begin{sidewaystable}[]
\begin{center}
\resizebox{\textwidth}{!}{
\begin{tabular}{|c|c|c|c|c|c|c|c|c|c|c|}
\hline
Budget & Data1, Base   & Data1, Aug            & Data2, Base   & Data2, Aug             & Data3, Base   & Data3, Aug             & Data4, Base   & Data4, Aug           & Data5, Base  & Data5, Aug           \\ \hline
8      & 7.328 (6)     & \textbf{3.313 (4)}    & 96.312 (14)   & \textbf{41.484 (8)}    & 146.703 (12)  & \textbf{48.812 (8)}    & 430.016 (17)  & \textbf{36.453 (4)}  & 0.812 (3)    & \textbf{0.453 (2)}   \\ \hline
12     & 318.421 (19)  & \textbf{43.656 (8)}   & 764.110 (17)  & \textbf{61.922 (6)}    & 7200* (32)    & \textbf{4047.500 (18)} & 853.891 (15)  & \textbf{224.391 (6)} & 0.891 (3)    & \textbf{0.484 (2)}   \\ \hline
16     & 18.375 (6)    & \textbf{16.187 (5)}   & 1871.375 (19) & \textbf{122.703 (6)}   & 7200* (20)    & \textbf{1783.781 (9)}  & 373.859 (10)  & \textbf{45.937 (3)}  & 39.031 (11)  & \textbf{27.719 (10)} \\ \hline
20     & 119.782 (11)  & \textbf{45.516 (6)}   & 2663.765 (18) & \textbf{418.422 (8)}   & 2665.813 (12) & \textbf{129.266 (4)}   & 316.156 (9)   & \textbf{83.750 (4)}  & 12.750 (7)   & \textbf{5.656 (4)}   \\ \hline
24     & 20.157 (6)    & \textbf{6.860 (4)}    & 2166.250 (15) & \textbf{916.703 (8)}   & 7200* (15)    & \textbf{4757.234 (9)}  & 3631.657 (15) & \textbf{788.032 (8)} & 11.656 (6)   & \textbf{3.344 (3)}   \\ \hline
28     & 283.906 (12)  & \textbf{33.015 (5)}   & 7200* (18)    & \textbf{6889.297 (14)} & 7200* (14)    & \textbf{2707.797 (9)}  & 4164.390 (15) & \textbf{677.562 (8)} & 232.110 (16) & \textbf{57.266 (8)}  \\ \hline
32     & 3101.609 (28) & \textbf{699.203 (12)} & 7200* (14)    & \textbf{336.813 (5)}   & 7200* (14)    & \textbf{5753.937 (10)} & 3811.922 (14) & \textbf{643.750 (8)} & 719.649 (18) & \textbf{85.296 (8)}  \\ \hline
36     & 616.703 (17)  & \textbf{441.172 (9)}  & 7200* (16)    & \textbf{3680.265 (12)} & 7200* (12)    & \textbf{399.141 (5)}   & 241.766 (7)   & \textbf{54.219 (4)}  & 169.546 (11) & \textbf{23.579 (6)}  \\ \hline
40     & 237.422 (17)  & \textbf{31.594 (6)}   & 6058.062 (18) & \textbf{1031.157 (10)} & 985.781 (8)   & \textbf{74.172 (4)}    & 1101.593 (10) & \textbf{420.906 (6)} & 186.516 (13) & \textbf{20.937 (5)}  \\ \hline
\end{tabular}}
\caption{Comparison of run times (seconds) and number of plans visited of algorithm \ref{alg:cncg_MMFNIPR} for delayed interdictions with and without augmentations, one attacker}
\label{tbl:compareTimeBaseD}

\vspace{2\baselineskip}
\resizebox{\textwidth}{!}{
\begin{tabular}{|c|c|c|c|c|c|c|c|c|c|c|}
\hline
Budget & Data1, Base        & Data1, Aug            & Data2, Base   & Data2, Aug             & Data3, Base   & Data3, Aug             & Data4, Base   & Data4, Aug            & Data5, Base   & Data5, Aug            \\ \hline
8      & 2.719 (5)          & \textbf{1.312 (3)}    & 19.812 (9)    & \textbf{9.157 (4)}     & 390.796 (25)  & \textbf{119.454 (9)}   & 1803.500 (27) & \textbf{495.203 (6)}  & 0.781 (3)     & \textbf{0.515 (2)}    \\ \hline
12     & \textbf{4.687 (5)} & 6.875 (5)             & 774.750 (22)  & \textbf{47.500 (5)}    & 4644.547 (31) & \textbf{1191.187 (14)} & \textbf{492.172 (16)}  & 526.922 (6)  & 1.735 (4)     & \textbf{0.532 (2)}    \\ \hline
16     & 2.750 (3)          & \textbf{2.062 (3)}    & 276.156 (14)  & \textbf{6.453 (3)}     & 5407.688 (27) & \textbf{191.313 (7)}   & 318.343 (13)  & \textbf{128.219 (5)}  & 25.344 (12)   & \textbf{5.187 (5)}    \\ \hline
20     & 76.485 (12)        & \textbf{5.907 (4)}    & 2759.157 (22) & \textbf{93.922 (7)}    & 1967.437 (19) & \textbf{575.296 (6)}   & 239.688 (11)  & \textbf{92.391 (4)}   & 40.015 (13)   & \textbf{3.344 (3)}    \\ \hline
24     & 28.140 (7)         & \textbf{6.625 (3)}    & 3071.187 (21) & \textbf{44.328 (4)}    & 7200* (22)    & \textbf{4597.032 (10)} & 3712.719 (20) & \textbf{1016.203 (8)} & 11.781 (6)    & \textbf{1.797 (2)}    \\ \hline
28     & 166.532 (13)       & \textbf{37.218 (6)}   & 6786.156 (27) & \textbf{874.359 (8)}   & 7200* (19)    & \textbf{3297.032 (13)} & 2710.734 (17) & \textbf{261.094 (5)}  & 125.844 (14)  & \textbf{14.093 (5)}   \\ \hline
32     & 1823.703 (28)      & \textbf{530.125 (14)} & 2172.750 (16) & \textbf{208.234 (6)}   & 7200* (19)    & \textbf{2789.437 (9)}  & 3801.25 (18)  & \textbf{254.750 (5)}  & 1135.031 (25) & \textbf{200.625 (14)} \\ \hline
36     & 284.125 (14)       & \textbf{98.782 (7)}   & 3969.797 (18) & \textbf{1945.579 (9)}  & 7200* (17)    & \textbf{513.078 (5)}   & 1825.578 (13) & \textbf{86.031 (4)}   & 408.438 (17)  & \textbf{101.500 (9)}  \\ \hline
40     & 345.781 (16)       & \textbf{22.750 (7)}   & 7200* (19)    & \textbf{3442.734 (12)} & 3256.641 (13) & \textbf{507.516 (5)}   & 7200* (17)    & \textbf{917.297 (8)}  & 385.359 (18)  & \textbf{61.985 (8)}   \\ \hline
\end{tabular}}
\end{center}
\caption{Comparison of run times (seconds) and number of plans visited of algorithm \ref{alg:cncg_MMFNIPR} for the base formulation with upfront interdictions with and without augmentations, one attacker}
\label{tbl:compareTimeBaseU}
\end{sidewaystable}

\begin{sidewaystable}[]
\begin{center}
\resizebox{\textwidth}{!}{
\begin{tabular}{|c|c|c|c|c|c|c|c|c|c|c|}
\hline
Budget & Data1, Base        & Data1, Aug            & Data2, Base   & Data2, Aug            & Data3, Base   & Data3, Aug            & Data4, Base           & Data4, Aug           & Data5, Base        & Data5, Aug          \\ \hline
8      & 2.094 (6)          & \textbf{0.891 (3)}    & 8.953 (10)    & \textbf{2.250 (4)}    & 71.203 (19)   & \textbf{29.531 (8)}   & 765.563 (24)          & \textbf{65.765 (5)}  & \textbf{0.266 (2)} & 0.312 (2)           \\ \hline
12     & \textbf{2.203 (5)} & 2.344 (4)             & 169.360 (20)  & \textbf{16.891 (4)}   & 868.422 (28)  & \textbf{230.922 (12)} & \textbf{166.172 (15)} & 228.344 (8)          & 0.875 (4)          & \textbf{0.360 (2)}  \\ \hline
16     & \textbf{0.828 (3)} & 1.140 (3)             & 107.765 (16)  & \textbf{3.031 (3)}    & 1647.062 (25) & \textbf{122.203 (7)}  & 102.234 (12)          & \textbf{40.438 (5)}  & 5.969 (10)         & \textbf{1.922 (4)}  \\ \hline
20     & 35.562 (12)        & \textbf{2.422 (4)}    & 1021.813 (22) & \textbf{22.641 (6)}   & 691.657 (19)  & \textbf{116.187 (6)}  & 81.281 (11)           & \textbf{30.562 (4)}  & 13.125 (10)        & \textbf{1.343 (3)}  \\ \hline
24     & 6.422 (6)          & \textbf{3.438 (4)}    & 742.937 (19)  & \textbf{33.484 (5)}   & 7200* (30)    & \textbf{727.485 (10)} & 1443.250 (20)         & \textbf{166.578 (7)} & 8.375 (7)          & \textbf{0.610 (2)}  \\ \hline
28     & 40.844 (11)        & \textbf{22.125 (7)}   & 2382.922 (25) & \textbf{114.781 (6)}  & 7200* (26)    & \textbf{420.937 (9)}  & 1007.984 (18)         & \textbf{57.828 (5)}  & 100.375 (17)       & \textbf{6.375 (5)}  \\ \hline
32     & 279.938 (22)       & \textbf{148.234 (12)} & 1057.359 (18) & \textbf{55.313 (5)}   & 7200* (24)    & \textbf{556.641 (9)}  & 947.078 (18)          & \textbf{85.594 (5)}  & 377.656 (23)       & \textbf{31.703 (8)} \\ \hline
36     & 80.093 (14)        & \textbf{13.438 (5)}   & 1031.610 (16) & \textbf{345.234 (9)}  & 5662.719 (22) & \textbf{115.781 (5)}  & 378.156 (13)          & \textbf{34.031 (5)}  & 285.000 (21)       & \textbf{21.515 (6)} \\ \hline
40     & 118.047 (17)       & \textbf{14.578 (7)}   & 7200* (32)    & \textbf{452.969 (11)} & 2914.656 (17) & \textbf{96.766 (5)}   & 2294.422 (19)         & \textbf{340.813 (9)} & 157.781 (19)       & \textbf{31.672 (8)} \\ \hline
\end{tabular}
}
\end{center}
\caption{Comparison of run times (seconds) and number of plans visited of algorithm \ref{alg:cncg_MMFNIPR} for the network phase formulation with upfront interdictions with and without augmentations, one attacker}
\label{tbl:compareTimeBaseUS}
\end{sidewaystable}

From these tables, we can see that, in the models with a single attacker, the modeling-based augmentations allow the C\&CG algorithm to almost always solve the problem significantly quicker for all formulations. Out of $135$ instances tested, only $6$ instances are solved faster without the augmentations. In all but two of these instances, the solve time for both methods is under $10$ seconds, indicating that those instances can already be easily solved. For all the instances for delayed interdictions, the method with augmentations always outperform the base method. The improvement in solve time by including the augmentations is quite significant. In $114$ out of $135$ instances, the solve time of the method with augmentations is at most $50\%$ that of the base method, and in $39$ out of $135$ instances, the solve time is at most $10\%$ that of the base method. Additionally, the base method is unable to solve 19 instances, whereas the method with augmentations solves every instance. This improvement is due to requiring significantly fewer iterations to solve the problem when the augmentations are included. Although the augmentations increase the size of the master problem, the augmentations were able to provide higher quality lower bounds in each iteration, thus solving the overall problem significantly quicker. In comparing Tables \ref{tbl:compareTimeBaseU} and \ref{tbl:compareTimeBaseUS}, we can additionally see that the network phase formulation for upfront interdiction drastically outperforms the base formulation. Every instance that both formulations solve is solved quicker with the network phase formulation, with $68$ out of $90$ instances being solved in at most $50\%$ of the solve time of the base formulation. Additionally, the network phase formulation solves two more instances than the base formulation.

We now present our results for the models with two cooperating attackers. Table \ref{tbl:compareTimeBaseD2} compares solve times with delayed interdictions, Table \ref{tbl:compareTimeBaseU2} compares solves times for the base formulation of upfront interdiction, and Table \ref{tbl:compareTimeBaseUS2} compares the solve times for the network phase formulation of upfront interdiction.

\begin{sidewaystable}[]
\begin{center}
\resizebox{\textwidth}{!}{
\begin{tabular}{|c|c|c|c|c|c|c|c|c|c|c|}
\hline
Budget & Data1, Base         & Data1, Aug           & Data2, Base            & Data2, Aug             & Data3, Base           & Data3, Aug           & Data4, Base          & Data4, Aug             & Data5, Base         & Data5, Aug          \\ \hline
8      & 203.547 (12)        & \textbf{110.922 (7)} & 224.718 (13)           & \textbf{84.875 (6)}    & 566.860 (11)          & \textbf{524.313 (7)} & 64.922 (6)           & \textbf{28.281 (4)}    & \textbf{30.219 (9)} & 62.968 (8)          \\ \hline
12     & 147.344 (10)        & \textbf{86.906 (6)}  & 1331.750 (16)          & \textbf{1044.469 (10)} & \textbf{527.109 (10)} & 567.937 (7)          & 124.734 (6)          & \textbf{114.157 (5)}   & 46.922 (10)         & \textbf{16.485 (5)} \\ \hline
16     & \textbf{9.047 (5)}  & 9.532 (4)            & 361.203 (8)            & \textbf{218.828 (6)}   & \textbf{923.531 (8)}  & 3700.844 (10)        & \textbf{46.657 (5)}  & 164.375 (6)            & 79.687 (11)         & \textbf{8.687 (4)}  \\ \hline
20     & 95.140 (9)          & \textbf{16.312 (4)}  & \textbf{809.766 (13)}  & 2674.500 (13)          & 1369.375 (8)          & \textbf{349.516 (6)} & 2296.250 (18)        & \textbf{947.781 (8)}   & 19.562 (7)          & \textbf{7.328 (3)}  \\ \hline
24     & 24.125 (7)          & \textbf{17.141 (5)}  & 7200* (18)             & 7200* (13)             & 3920.375 (14)         & \textbf{504.828 (7)} & 7200* (14)           & 7200 (11)*             & \textbf{37.766 (8)} & 50.032 (8)          \\ \hline
28     & \textbf{16.813 (6)} & 25.390 (6)           & 7200* (26)             & \textbf{5060.297 (11)} & \textbf{1148.625 (9)} & 1748.016 (9)         & 595.468 (8)          & \textbf{525.344 (7)}   & 92.422 (8)          & \textbf{69.859 (6)} \\ \hline
32     & 66.515 (11)         & \textbf{52.094 (8)}  & \textbf{1509.125 (18)} & 2249.719 (11)          & 3430.563 (15)         & \textbf{342.312 (6)} & 3707.141 (14)        & \textbf{3673.765 (12)} & 14.235 (8)          & \textbf{2.375 (2)}  \\ \hline
36     & \textbf{24.344 (7)} & 30.531 (7)           & \textbf{116.140 (6)}   & 158.922 (7)            & 1557.259 (12)         & \textbf{148.578 (5)} & \textbf{407.844 (9)} & 749.954 (9)            & 3.937 (4)           & \textbf{2.250 (3)}  \\ \hline
40     & \textbf{2.500 (3)}  & 5.782 (4)            & \textbf{32.094 (6)}    & 571.031 (12)           & \textbf{32.469 (4)}   & 34.719 (4)           & 7200* (21)           & 7200* (17)             & \textbf{2.500 (3)}  & 2.562 (3)           \\ \hline
\end{tabular}}
\caption{Comparison of run times (seconds) and number of plans visited of algorithm \ref{alg:cncg_MMFNIPR} for delayed interdictions with and without augmentations, two attackers}
\label{tbl:compareTimeBaseD2}

\vspace{2\baselineskip}
\resizebox{\textwidth}{!}{
\begin{tabular}{|c|c|c|c|c|c|c|c|c|c|c|}
\hline
Budget & Data1, Base  & Data1, Aug            & Data2, Base           & Data2, Aug             & Data3, Base            & Data3, Aug            & Data4, Base   & Data4, Aug             & Data5, Base         & Data5, Aug          \\ \hline
8      & 77.078 (9)   & \textbf{39.297 (6)}   & 92.656 (10)           & \textbf{50.157 (5)}    & \textbf{383.265 (14)}  & 495.062 (7)           & 17.812 (6)    & \textbf{10.563 (3)}    & \textbf{11.937 (7)} & 52.343 (7)          \\ \hline
12     & 64.329 (9)   & \textbf{34.172 (6)}   & 1783.781 (18)         & \textbf{173.828 (6)}   & 263.438 (9)            & \textbf{202.204 (6)}  & 492.625 (9)   & \textbf{176.500 (4)}   & 7.125 (4)           & \textbf{2.860 (2)}  \\ \hline
16     & 27.484 (6)   & \textbf{7.734 (4)}    & 790.297 (10)          & \textbf{352.578 (8)}   & 7200* (15)             & \textbf{1095.687 (6)} & 275.500 (8)   & \textbf{50.250 (4)}    & 42.891 (10)         & \textbf{6.672 (3)}  \\ \hline
20     & 153.922 (11) & \textbf{37.281 (5)}   & 4496.516 (23)         & \textbf{390.344 (8)}   & 1889.781 (12)          & \textbf{458.078 (6)}  & 5703.500 (18) & \textbf{3395.328 (11)} & \textbf{8.047 (4)}  & 10.281 (4)          \\ \hline
24     & 37.953 (8)   & \textbf{12.672 (4)}   & 7200* (21)            & 7200* (13)             & 3934.318 (17)          & \textbf{2142.922 (9)} & 6817.485 (21) & \textbf{1811.766 (10)} & 59.594 (6)          & \textbf{6.375 (3)}  \\ \hline
28     & 86.469 (9)   & \textbf{48.203 (6)}   & 7200* (30)            & \textbf{6942.609 (20)} & \textbf{3491.125 (13)} & 3576.328 (8)          & 2413.125 (14) & \textbf{1128.203 (8)}  & 154.578 (17)        & \textbf{42.953 (6)} \\ \hline
32     & 203.562 (14) & \textbf{14.313 (4)}   & 2200.047 (18)         & \textbf{1454.110 (11)} & 2365.125 (13)          & \textbf{897.406 (8)}  & 2729.718 (14) & \textbf{486.703 (6)}   & 9.531 (4)           & \textbf{5.563 (3)}  \\ \hline
36     & 238.141 (15) & \textbf{100.015 (10)} & 818.328 (15)          & \textbf{351.484 (8)}   & 890.188 (10)           & \textbf{284.032 (6)}  & 826.594 (11)  & \textbf{181.890 (5)}   & 11.703 (5)          & \textbf{4.328 (4)}  \\ \hline
40     & 115.562 (20) & \textbf{86.985 (16)}  & \textbf{507.312 (13)} & 1362.281 (12)          & 377.203 (9)            & \textbf{64.078 (4)}   & 7200* (24)    & \textbf{2704.344 (11)} & \textbf{10.313 (4)} & 10.640 (6)          \\ \hline
\end{tabular}}
\end{center}
\caption{Comparison of run times (seconds) and number of plans visited of algorithm \ref{alg:cncg_MMFNIPR} for the base formulation with upfront interdictions with and without augmentations, two attackers}
\label{tbl:compareTimeBaseU2}
\end{sidewaystable}

\begin{sidewaystable}[]
\begin{center}
\resizebox{\textwidth}{!}{
\begin{tabular}{|c|c|c|c|c|c|c|c|c|c|c|}
\hline
Budget & Data1, Base         & Data1, Aug           & Data2, Base   & Data2, Aug             & Data3, Base           & Data3, Aug            & Data4, Base         & Data4, Aug             & Data5, Base        & Data5, Aug          \\ \hline
8      & 51.515 (10)         & \textbf{9.781 (6)}   & 101.985 (11)  & \textbf{16.375 (5)}    & 185.360 (12)          & \textbf{119.578 (6)}  & \textbf{7.8282 (5)} & 13.531 (4)             & 3.750 (7)          & \textbf{3.438 (5)}  \\ \hline
12     & 39.922 (10)         & \textbf{22.235 (7)}  & 628.718 (18)  & \textbf{61.641 (5)}    & 115.203 (10)          & \textbf{44.016 (4)}   & 250.797 (11)        & \textbf{67.615 (5)}    & 10.110 (8)         & \textbf{2.765 (3)}  \\ \hline
16     & \textbf{4.297 (4)}  & 9.250 (5)            & 242.766 (12)  & \textbf{45.406 (5)}    & 1466.735 (16)         & \textbf{570.734 (8)}  & 110.766 (8)         & \textbf{13.219 (3)}    & 52.422 (11)        & \textbf{4.735 (3)}  \\ \hline
20     & 55.859 (11)         & \textbf{7.187 (4)}   & 576.250 (15)  & \textbf{203.750 (10)}  & 884.078 (12)          & \textbf{240.141 (6)}  & 2753.535 (21)       & \textbf{1138.469 (13)} & 5.843 (5)          & \textbf{4.828 (4)}  \\ \hline
24     & 19.766 (9)          & \textbf{6.250 (4)}   & 2452.266 (22) & \textbf{2021.250 (15)} & \textbf{624.765 (12)} & 662.359 (10)          & 1893.531 (19)       & \textbf{487.640 (9)}   & 57.782 (9)         & \textbf{5.140 (4)}  \\ \hline
28     & 52.516 (16)         & \textbf{5.250 (4)}   & 3102.656 (31) & \textbf{1410.938 (16)} & 1220.032 (15)         & \textbf{817.469 (10)} & 454.047 (12)        & \textbf{169.297 (6)}   & 90.125 (16)        & \textbf{22.141 (7)} \\ \hline
32     & 122.687 (16)        & \textbf{4.860 (4)}   & 1581.187 (25) & \textbf{632.562 (13)}  & 1317.968 (17)         & \textbf{169.390 (6)}  & 497.828 (11)        & \textbf{245.109 (7)}   & 4.359 (5)          & \textbf{3.437 (3)}  \\ \hline
36     & \textbf{22.860 (9)} & 61.578 (11)          & 220.891 (12)  & \textbf{131.797 (8)}   & 395.188 (12)          & \textbf{115.219 (6)}  & 294.079 (11)        & \textbf{80.360 (5)}    & \textbf{1.422 (3)} & 2.563 (3)           \\ \hline
40     & 54.390 (18)         & \textbf{47.125 (14)} & 385.656 (15)  & \textbf{98.234 (7)}    & 59.390 (7)            & \textbf{19.172 (4)}   & 4905.890 (29)       & \textbf{949.578 (12)}  & 23.312 (12)        & \textbf{1.375 (2)}  \\ \hline
\end{tabular}
}
\end{center}
\caption{Comparison of run times (seconds) and number of plans visited of algorithm \ref{alg:cncg_MMFNIPR} for the network phase formulation with upfront interdictions with and without augmentations, two attackers}
\label{tbl:compareTimeBaseUS2}
\end{sidewaystable}

Here, we see that, while the method with augmentations is less effective with two attackers than with one, it often solves the problem faster than the method without augmentations. We expect that the augmentations would be less effective in these instances, since the second attacker can prevent the recruitment of new participants, which limits the benefits of the augmentation that projects the recruitment of new participants. Now, $28$ of the instances are solved faster without including augmentations. In instances with two attackers, the base method is unable to solve $8$ instances, and the method with augmentations is unable to solve $4$ instances. We note that every instance that is solved by the base method is also solved by the method with augmentations. Comparing Tables \ref{tbl:compareTimeBaseU2} and \ref{tbl:compareTimeBaseUS2}, we again see significant benefits from the network phase formulation for upfront interdiction. The base formulation only outperforms the network phase formulation in $3$ instances, all of which are solved in under $60$ seconds. For the network phase formulation, $54$ of the $90$ instances are solved in at most $50\%$ of the solve time of the base formulation. Additionally, the network phase formulation is able to solve every instance, where the base formulation is unable to solve $5$ instances.

We now compare the quality of the bounds determined in unsolved instances. For instances with delayed interdiction for one attacker and two attackers, we report the relative optimality gap, the difference between the upper and lower bound divided by the upper bound, in Tables \ref{tbl:relgapD} and \ref{tbl:relgapD2}. For instances with upfront interdiction for one attacker and two attackers, we report the relative optimality gap in Tables \ref{tbl:relgapU} and \ref{tbl:relgapU2}. These results show that, with one exception, the quality of the bounds identified using the method with augmentations are higher quality than the base method, and that the bounds identified by the network phase model for upfront interdiction are higher quality than the base model. Because the models with augmentations both solve more instances and identify higher quality solutions in unsolved instances, the inclusion of these augmentations better ensures that the recommendations provided by our model appropriately account for how the traffickers will react to those decisions.

\begin{table}[h!]
\begin{center}
{\scriptsize %
\begin{tabular}{|c|c|c|}
\hline
Instance          & Base    & Aug \\ \hline
Data2, Budget 28  & 5.469\% & -   \\ \hline
Data2, Budget 32 & 8.333\% & -   \\ \hline
Data2, Budget 36  & 7.207\% & -   \\ \hline
Data3, Budget 12  & 0.562\% & -   \\ \hline
Data3, Budget 16  & 2.907\% & -   \\ \hline
Data3, Budget 24  & 1.987\% & -   \\ \hline
Data3, Budget 28  & 3.471\% & -   \\ \hline
Data3, Budget 32  & 3.008\% & -   \\ \hline
Data3, Budget 36  & 4.800\% & -   \\ \hline
\end{tabular} 
}%
\caption{Relative optimality gap of unsolved instances with delayed interdiction, one attacker}
\label{tbl:relgapD}
\end{center}
\end{table}

\begin{table}[h!]
\begin{center}
{\scriptsize %
\begin{tabular}{|c|c|c|}
\hline
Instance          & Base    & Aug     \\ \hline
Data2, Budget 24  & 3.604\% & 0.926\% \\ \hline
Data2, Budget 28 & 3.030\% & -       \\ \hline
Data4, Budget 24  & 2.326\% & 3.817\% \\ \hline
Data4, Budget 40  & 5.495\% & 1.149\% \\ \hline
\end{tabular}
}%
\caption{Relative optimality gap of unsolved instances with delayed interdiction, two attackers}
\label{tbl:relgapD2}
\end{center}
\end{table}

\begin{table}[h!]
\begin{center}
{\scriptsize %
\begin{tabular}{|c|c|c|c|c|}
\hline
Instance         & Base     & Aug & Phase   & Phase Aug \\ \hline
Data2, Budget 40 & 7.059\%  & -   & 2.439\% & -         \\ \hline
Data3, Budget 24 & 2.837\%  & -   & 2.817\% & -         \\ \hline
Data3, Budget 28 & 8.696\%  & -   & 4.511\% & -         \\ \hline
Data3, Budget 32 & 8.730\%  & -   & 3.306\% & -         \\ \hline
Data3, Budget 36 & 13.333\% & -   & -       & -         \\ \hline
Data4, Budget 40 & 7.143\%  & -   & -       & -         \\ \hline
\end{tabular}
}%
\caption{Relative optimality gap of unsolved instances with upfront interdiction, one attacker}
\label{tbl:relgapU}
\end{center}
\end{table}

\begin{table}[h!]
\begin{center}
{\scriptsize %
\begin{tabular}{|c|c|c|c|c|}
\hline
Instance          & Base     & Aug    & Phase & Phase Aug \\ \hline
Data2, Budget 24  & 5.000\%  & 4.00\% & -     & -         \\ \hline
Data2, Budget 28 & 4.651\%  & -      & -     & -         \\ \hline
Data3, Budget 16  & 5.479\%  & -      & -     & -         \\ \hline
Data4, Budget 40  & 1.471 \% & -      & -     & -         \\ \hline
\end{tabular}
}%
\caption{Relative optimality gap of unsolved instances with upfront interdiction, two attackers}
\label{tbl:relgapU2}
\end{center}
\end{table}

\subsection{Application-Based Analysis}
We now compare the policy recommendations of a multi-period MFNIP model without restructuring against those recommended by MP-MFNIP-R. For unsolved instances, we present the solution associated with the upper bound from the model and method that resulted in the smallest upper bound. This is because this solution is a bilevel feasible solution, and thus is what we want to minimize. We present the results of two sample networks, network $1$ and $3$, as the other three follow similar trends to one of these two; those results are included in Appendix \ref{app:fullresults}. We first present the results on flow through the network. Figures \ref{fig:flowplots_d_1} and \ref{fig:flowplots_d_2} present the flows with delayed interdictions with one and two attackers, respectively. Figures \ref{fig:flowplots_u_1} and \ref{fig:flowplots_u_2} present the flows with upfront interdictions with one and two attackers, respectively. In these figures, the black dotted line represents the total flow through the un-interdicted network, which is the total number of victims (bottoms included) multiplied by the number of time periods. The blue circles are the optimal flow as determined by multi-period MFNIP, and the red stars are the optimal flow after restructuring responding to multi-period MFNIP's recommended interdictions. The green triangles are the optimal flow as determined by MP-MFNIP-R.

\begin{figure}[h!]
\centering
\begin{subfigure}{.45\textwidth}
  \centering
  \includegraphics[width=\linewidth]{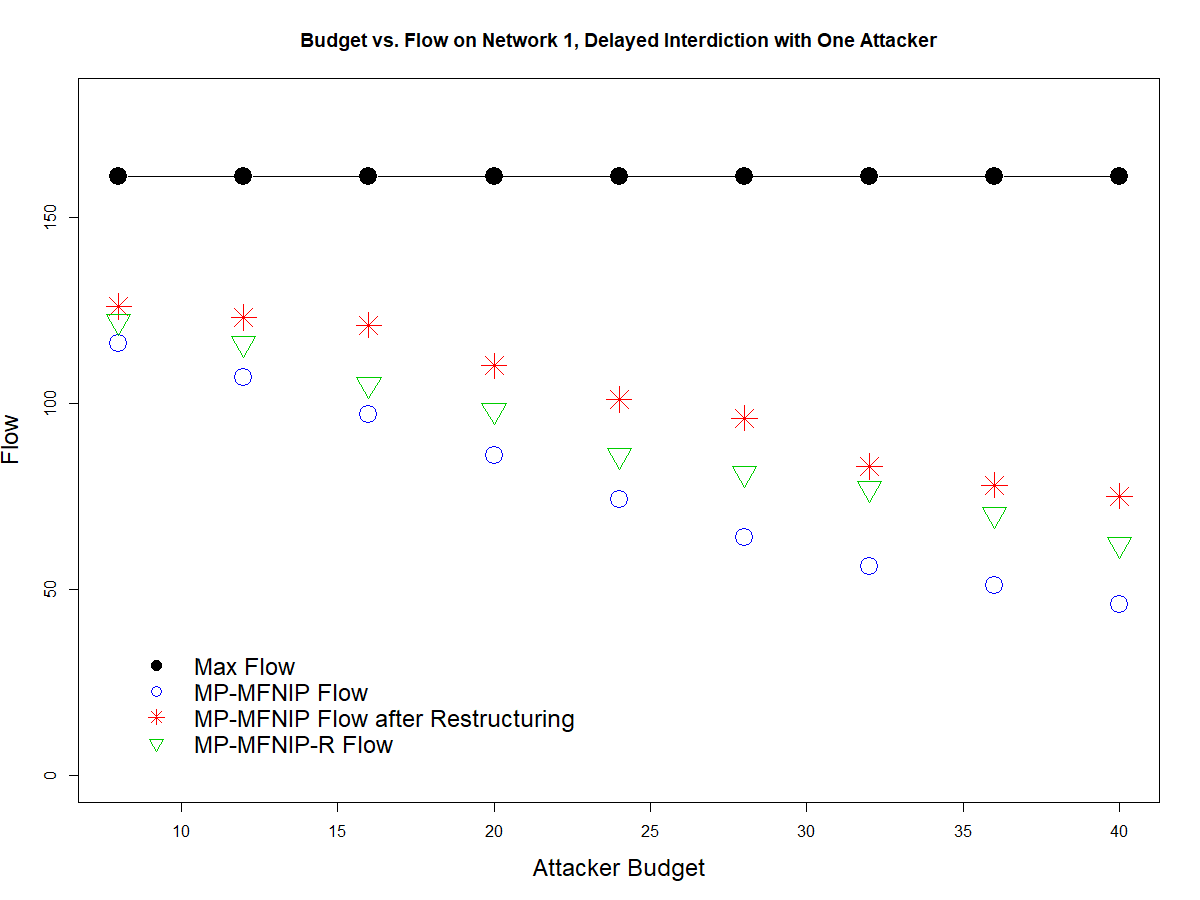}
  \caption{Network 1}
  \label{fig:net1_d_1_flow}
\end{subfigure}%
\begin{subfigure}{.45\textwidth}
  \centering
  \includegraphics[width=\linewidth]{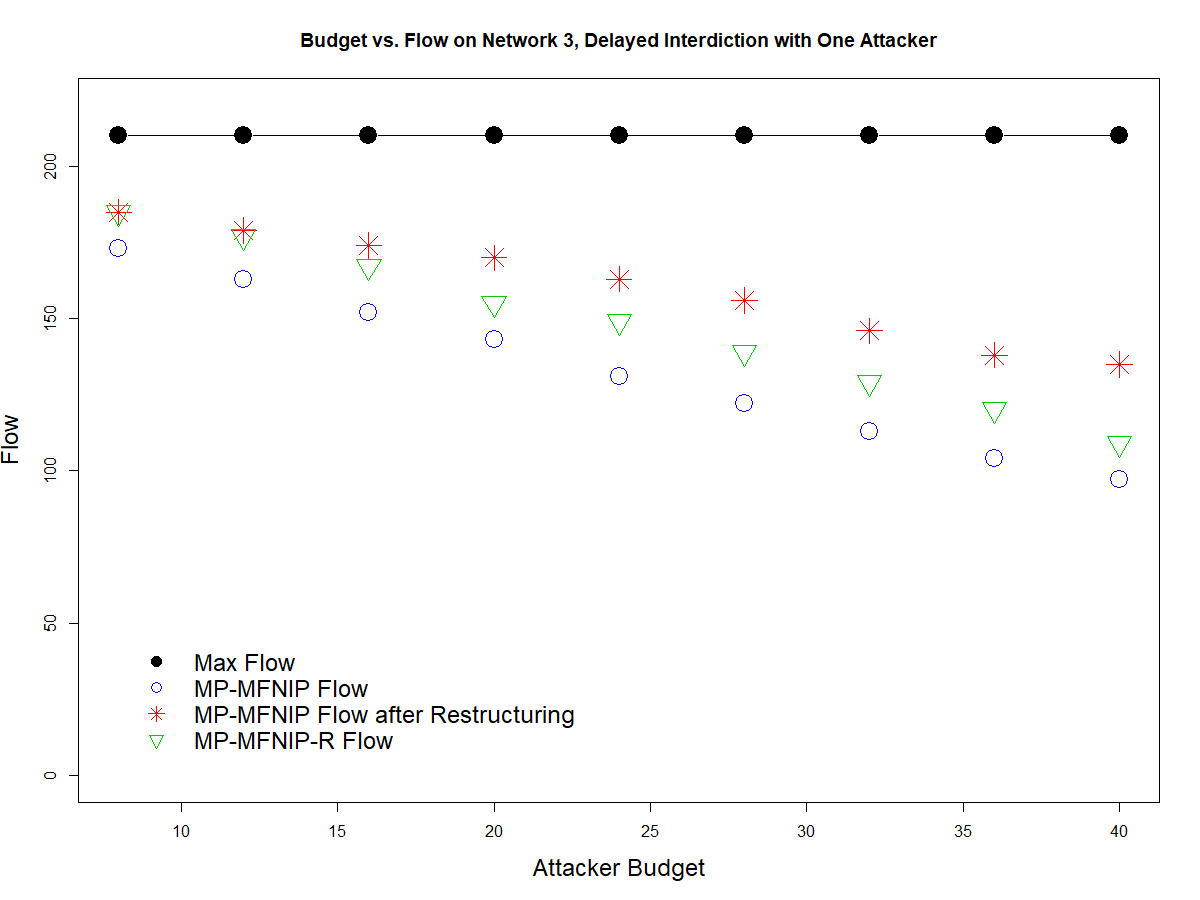}
  \caption{Network 3}
  \label{fig:net3_d_1_flow}
\end{subfigure}
\caption{Interdicted and restructured flows with delayed interdiction and one attacker over varying attacker budgets}
\label{fig:flowplots_d_1}
\end{figure}

\begin{figure}[h!]
\centering
\begin{subfigure}{.45\textwidth}
  \centering
  \includegraphics[width=\linewidth]{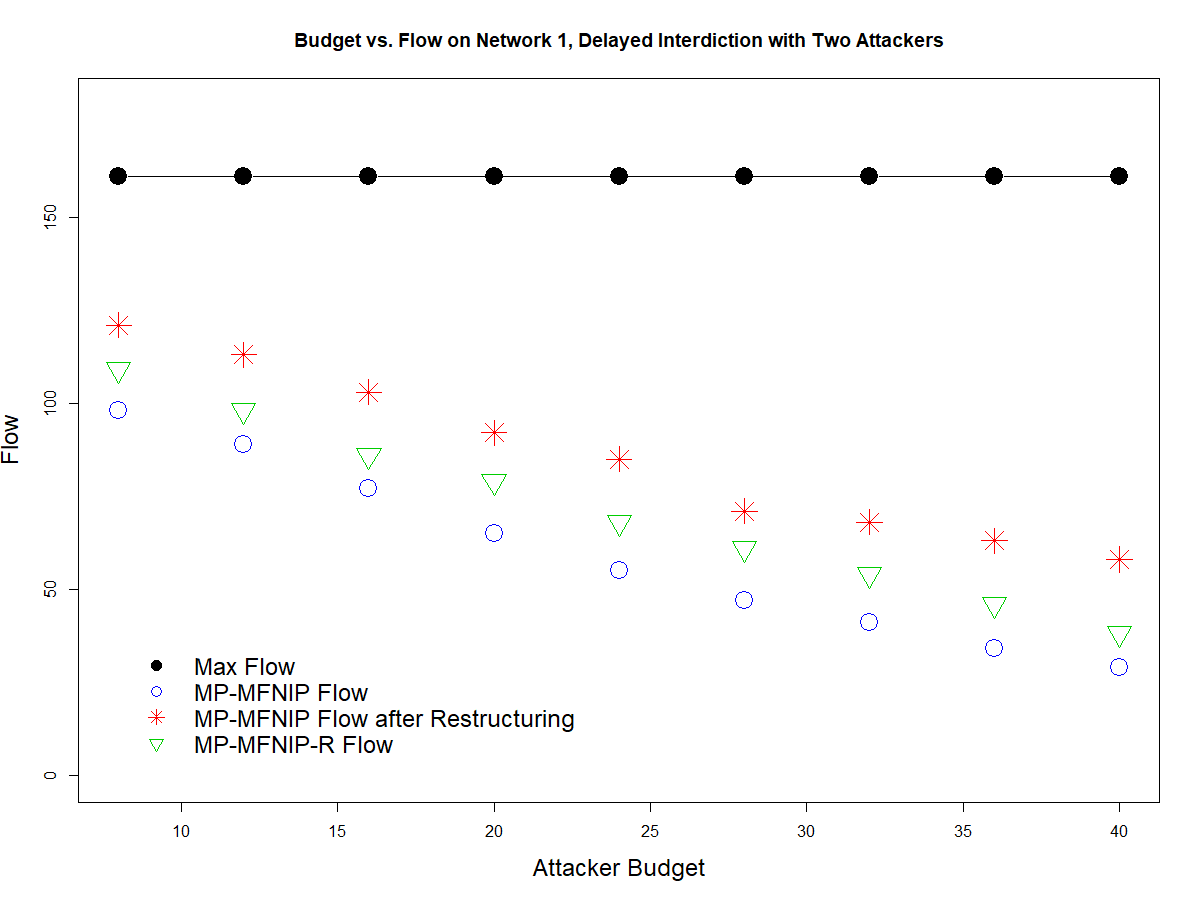}
  \caption{Network 1}
  \label{fig:net1_d_2_flow}
\end{subfigure}%
\begin{subfigure}{.45\textwidth}
  \centering
  \includegraphics[width=\linewidth]{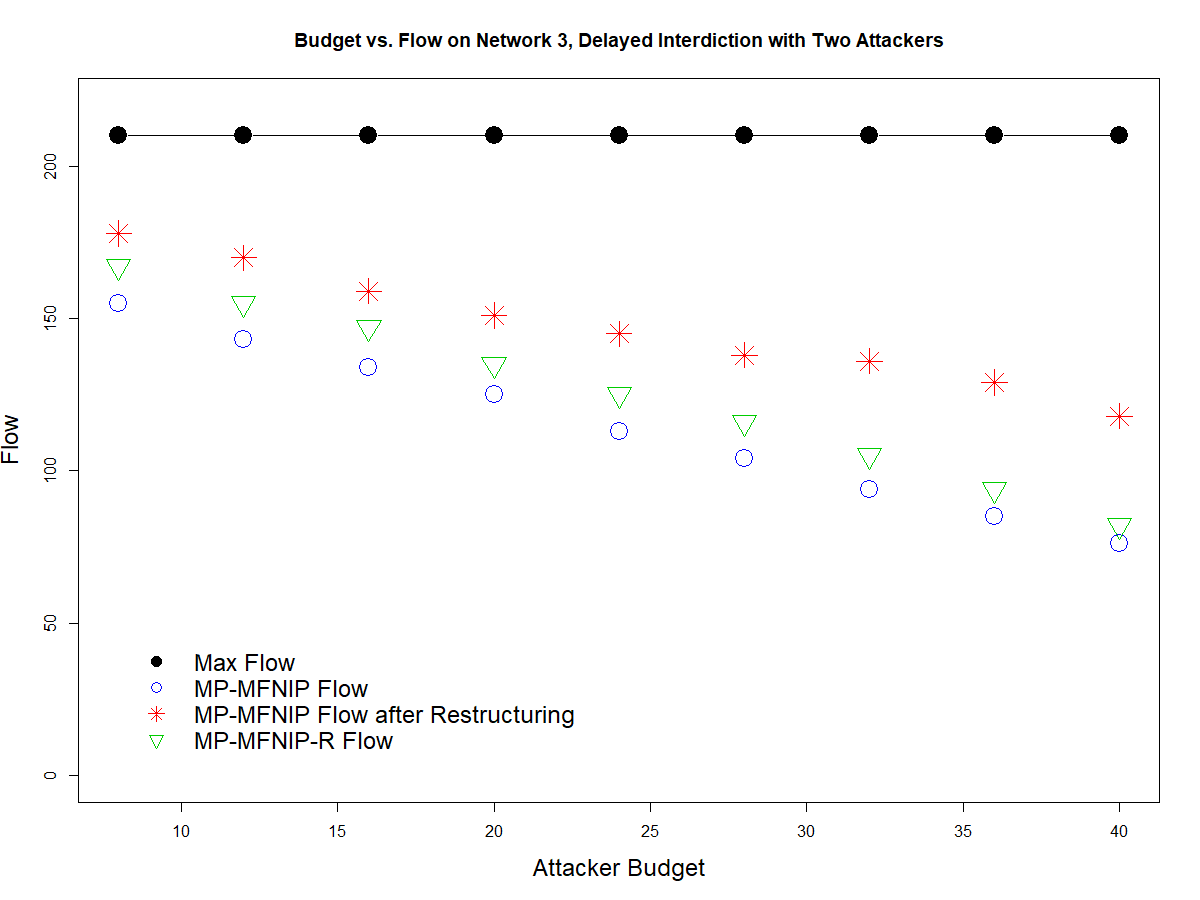}
  \caption{Network 3}
  \label{fig:net3_d_2_flow}
\end{subfigure}
\caption{Interdicted and restructured flows with delayed interdiction and two attackers over varying attacker budgets}
\label{fig:flowplots_d_2}
\end{figure}

\begin{figure}[h!]
\centering
\begin{subfigure}{.45\textwidth}
  \centering
  \includegraphics[width=\linewidth]{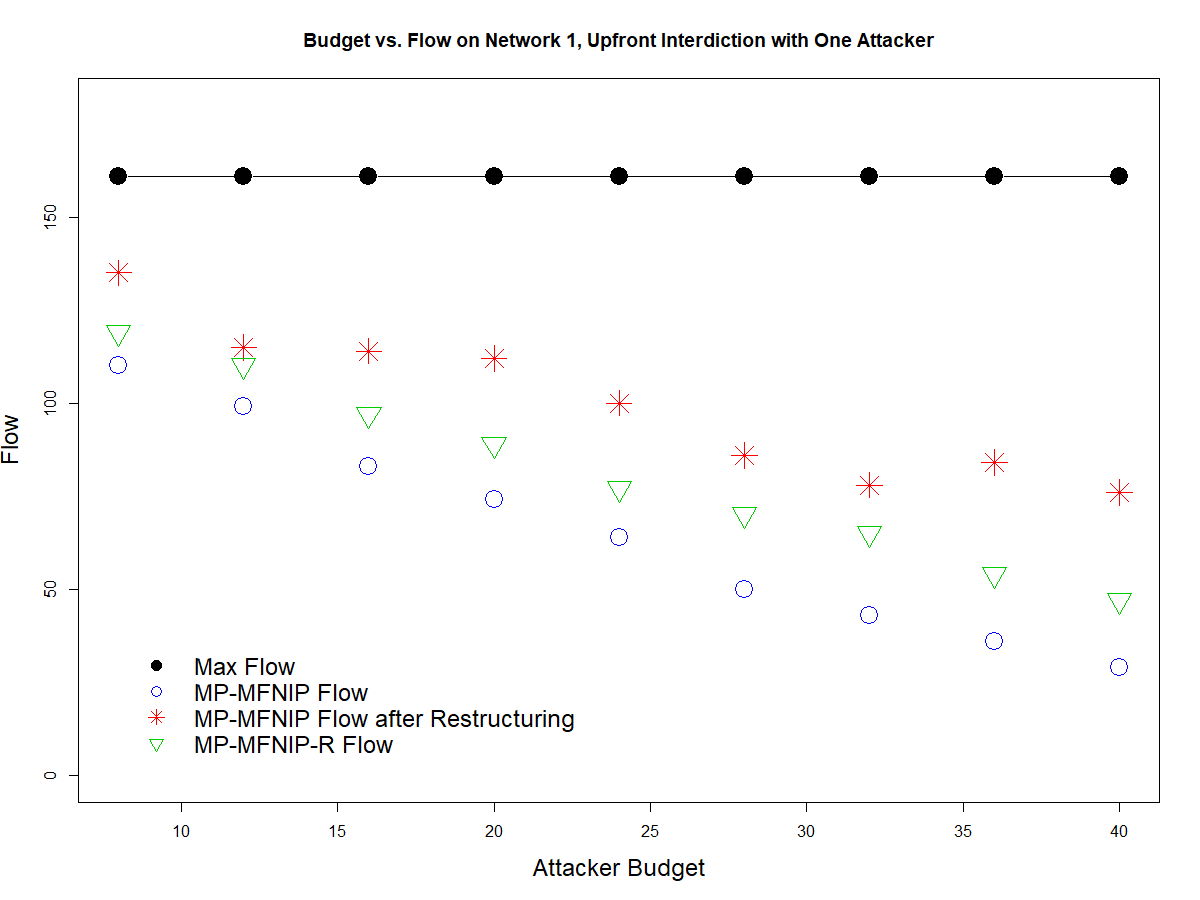}
  \caption{Network 1}
  \label{fig:net1_u_1_flow}
\end{subfigure}%
\begin{subfigure}{.45\textwidth}
  \centering
  \includegraphics[width=\linewidth]{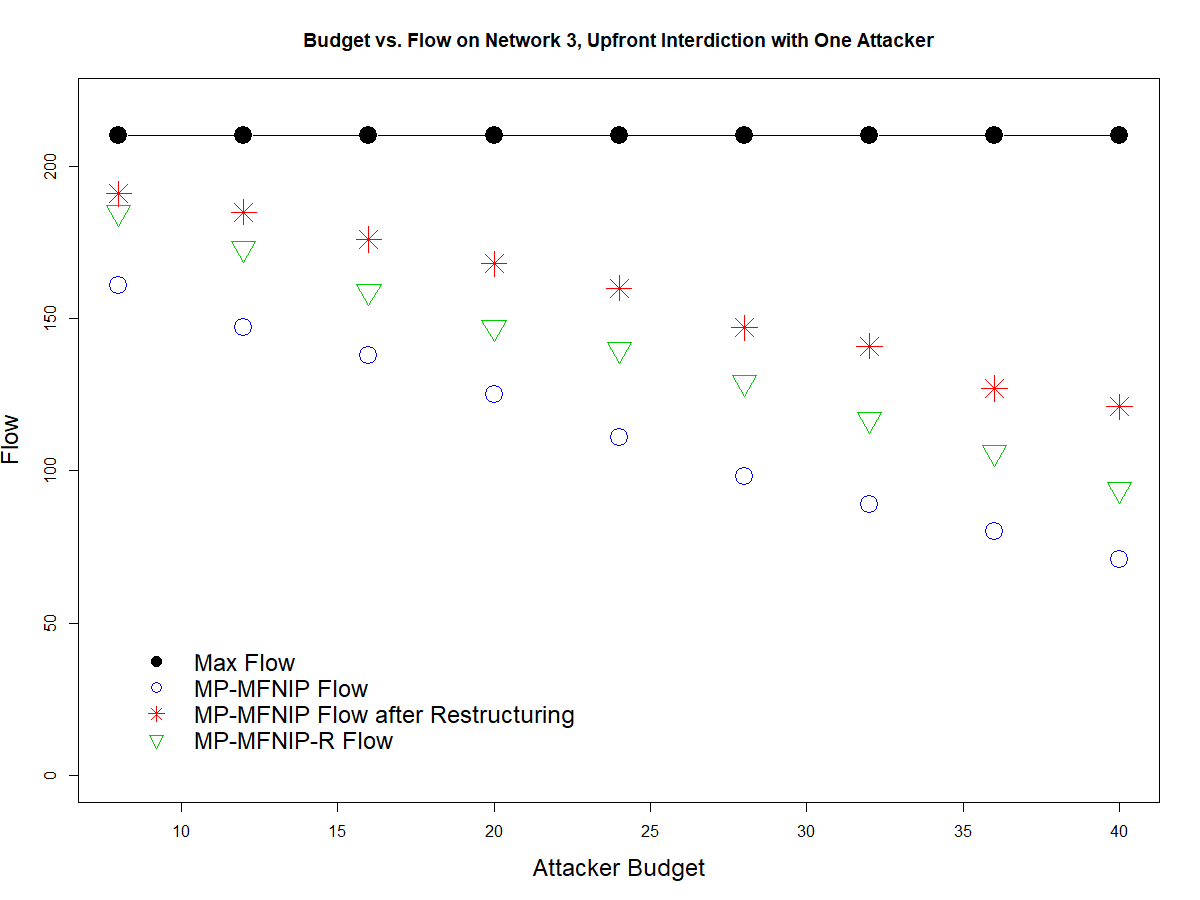}
  \caption{Network 3}
  \label{fig:net3_u_1_flow}
\end{subfigure}
\caption{Interdicted and restructured flows with upfront interdiction and one attacker over varying attacker budgets}
\label{fig:flowplots_u_1}
\end{figure}

\begin{figure}[h!]
\centering
\begin{subfigure}{.45\textwidth}
  \centering
  \includegraphics[width=\linewidth]{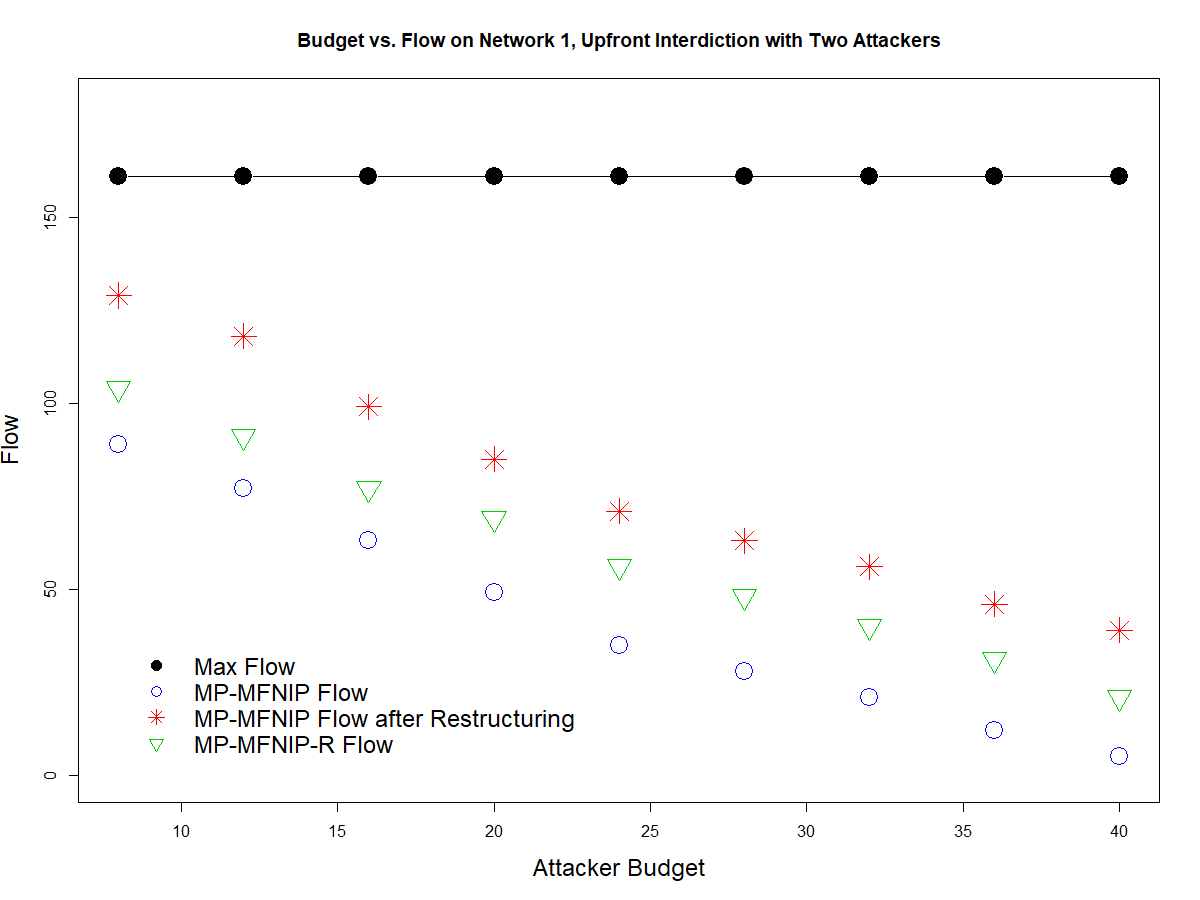}
  \caption{Network 1}
  \label{fig:net1_u_2_flow}
\end{subfigure}%
\begin{subfigure}{.45\textwidth}
  \centering
  \includegraphics[width=\linewidth]{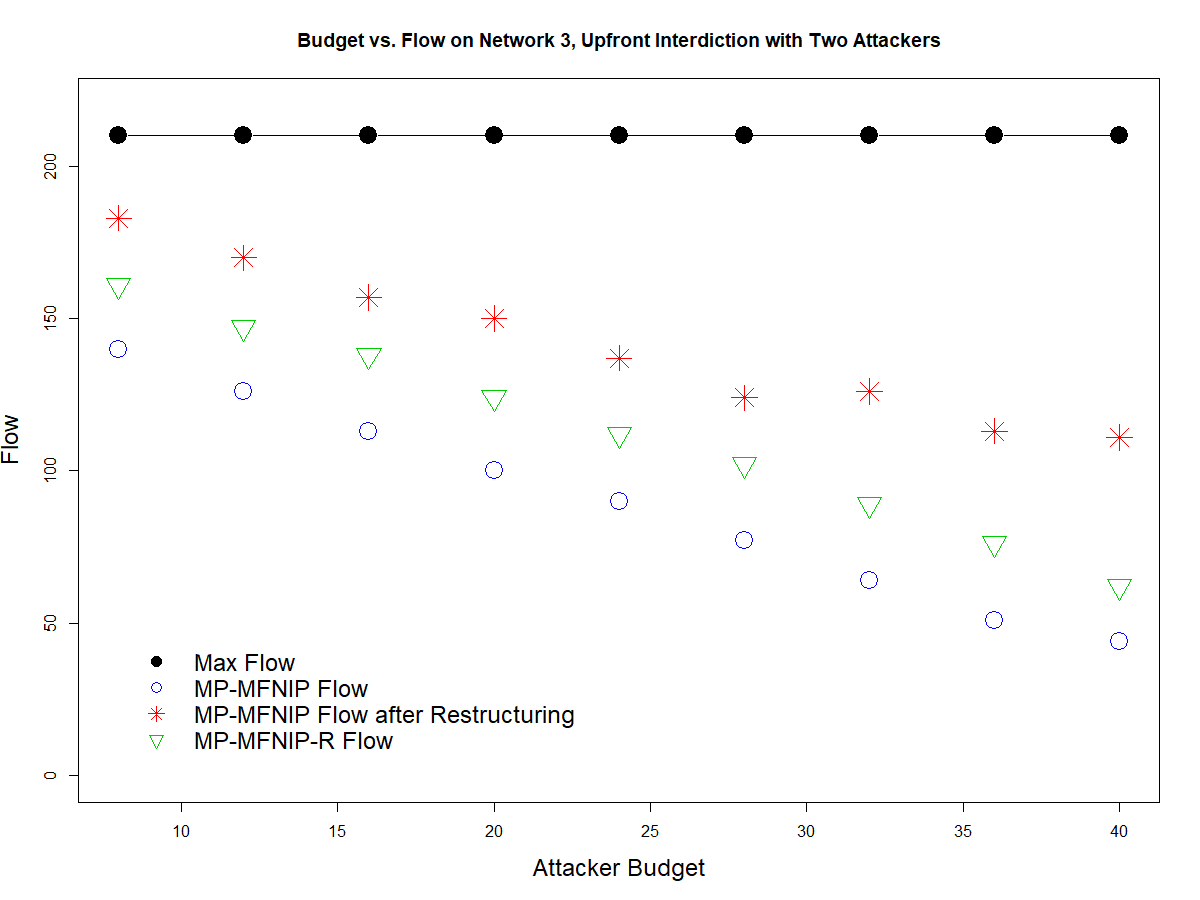}
  \caption{Network 3}
  \label{fig:net3_u_2_flow}
\end{subfigure}
\caption{Interdicted and restructured flows with upfront interdiction and two attackers over varying attacker budgets}
\label{fig:flowplots_u_2}
\end{figure}

Across all instances, the defender is able to regain a significant amount of flow after restructuring in response to the interdictions recommended by multi-period MFNIP. The interdictions recommended by MP-MFNIP-R are able to account for these restructurings, resulting in a lower optimal flow. However, in instances with only one attacker, these flows are closer to the restructured flows than the flows projected by multi-period MFNIP. This is particularly noticeable in network $1$ at higher budget levels, and network $3$ at lower budget levels. However, with the inclusion of a second attacker that can prevent the recruitment of new participants, the amount of flow restructured in MP-MFNIP-R is significantly limited, especially in the case of network $3$. By preventing recruitment, more victims can be interdicted, while also preventing them from being replaced, resulting in effectively reducing the total flow.

In comparing the flow differences between delayed interdiction and upfront interdiction, the difference in flows between the two models starts small, and the decrease in flow from upfront interdictions grows as the attacker budget increases. This aligns with our expectations, since at larger budget levels, the attacker with upfront interdictions can disrupt more of the network sooner. Even though this allows for restructurings to be implemented in earlier time periods, the original network is able to operate better than the restructured network.

We next present results on how many nodes of each type were recommended to be interdicted. Tables \ref{tbl:net1_d_1int_who_text} and \ref{tbl:net1_d_2int_who_text} present the recommended interdictions for network $1$ with delayed interdictions and one and two attackers, respectively. Tables \ref{tbl:net3_d_1int_who_text} and \ref{tbl:net3_d_2int_who_text} likewise present these results on network $3$. Tables \ref{tbl:net1_u_1int_who_text} - \ref{tbl:net3_u_2int_who_text} present these results with upfront interdictions. Columns with ``MP-MFNIP" report the results from multi-period MFNIP, and columns with ``MP-MFNIP-R" report the results from MMNFIP-R. Columns with ``2nd" report the decisions of the second attacker.

\begin{table}[h!]
\begin{center}
{\scriptsize %
\begin{tabular}{|c|c|c|c|c|c|c|}
\hline
Budget & \begin{tabular}[c]{@{}c@{}}MP-MFNIP\\ Int Trafficker\end{tabular} & \begin{tabular}[c]{@{}c@{}}MP-MFNIP\\ Int Bottom\end{tabular} & \begin{tabular}[c]{@{}c@{}}MP-MFNIP\\ Int Victim\end{tabular} & \begin{tabular}[c]{@{}c@{}}MP-MNFIP-R\\ Int Trafficker\end{tabular} & \begin{tabular}[c]{@{}c@{}}MP-MFNIP-R\\ Int Bottom\end{tabular} & \begin{tabular}[c]{@{}c@{}}MP-MFNIP-R\\ Int Victim\end{tabular} \\ \hline
8 & 0 & 1 & 2 & 0 & 1 & 2 \\ \hline
12 & 0 & 1 & 4 & 0 & 2 & 2 \\ \hline
16 & 0 & 2 & 4 & 1 & 0 & 6 \\ \hline
20 & 0 & 1 & 8 & 1 & 0 & 8 \\ \hline
24 & 0 & 1 & 10 & 1 & 0 & 10 \\ \hline
28 & 0 & 2 & 10 & 1 & 1 & 10 \\ \hline
32 & 0 & 1 & 14 & 0 & 0 & 16 \\ \hline
36 & 0 & 2 & 14 & 0 & 0 & 18 \\ \hline
40 & 0 & 2 & 16 & 0 & 0 & 20 \\ \hline
\end{tabular}
}%
\caption{Recommended delayed interdictions on network 1 with one attacker}
\label{tbl:net1_d_1int_who_text}
\end{center}
\end{table}

\begin{table}[h!]
\begin{center}
{\scriptsize %
\begin{tabular}{|c|c|c|c|c|c|c|}
\hline
Budget & \begin{tabular}[c]{@{}c@{}}MP-MFNIP\\ Int Trafficker\end{tabular} & \begin{tabular}[c]{@{}c@{}}MP-MFNIP\\ Int Bottom\end{tabular} & \begin{tabular}[c]{@{}c@{}}MP-MFNIP\\ Int Victim\end{tabular} & \begin{tabular}[c]{@{}c@{}}MP-MFNIP-R\\ Int Trafficker\end{tabular} & \begin{tabular}[c]{@{}c@{}}MP-MFNIP-R\\ Int Bottom\end{tabular} & \begin{tabular}[c]{@{}c@{}}MP-MFNIP-R\\ Int Victim\end{tabular} \\ \hline
8 & 0 & 2 & 0 & 0 & 2 & 0 \\ \hline
12 & 0 & 3 & 0 & 0 & 0 & 6 \\ \hline
16 & 0 & 2 & 4 & 1 & 0 & 6 \\ \hline
20 & 0 & 2 & 6 & 1 & 0 & 8 \\ \hline
24 & 1 & 2 & 6 & 0 & 0 & 12 \\ \hline
28 & 2 & 2 & 5 & 1 & 0 & 12 \\ \hline
32 & 2 & 2 & 7 & 1 & 0 & 14 \\ \hline
36 & 2 & 2 & 9 & 1 & 0 & 16 \\ \hline
40 & 1 & 2 & 14 & 1 & 0 & 18 \\ \hline
\end{tabular}
}%
\caption{Recommended delayed interdictions on network 3 with one attacker}
\label{tbl:net3_d_1int_who_text}
\end{center}
\end{table}

\begin{sidewaystable}[]
\begin{center}
\resizebox{\textwidth}{!}{
\begin{tabular}{|c|c|c|c|c|c|c|c|c|c|c|}
\hline
Budget & \begin{tabular}[c]{@{}c@{}}MP-MFNIP\\ Int Trafficker\end{tabular} & \begin{tabular}[c]{@{}c@{}}MP-MFNIP\\ Int Bottom\end{tabular} & \begin{tabular}[c]{@{}c@{}}MP-MFNIP\\ Int Victim\end{tabular} & \begin{tabular}[c]{@{}c@{}}MP-MFNIP 2nd\\ Int Current Vic\end{tabular} & \begin{tabular}[c]{@{}c@{}}MP-MFNIP 2nd\\ Int Prosp Vic\end{tabular} & \begin{tabular}[c]{@{}c@{}}MP-MFNIP-R\\ Int Trafficker\end{tabular} & \begin{tabular}[c]{@{}c@{}}MP-MFNIP-R\\ Int Bottom\end{tabular} & \begin{tabular}[c]{@{}c@{}}MP-MFNIP-R\\ Int Victim\end{tabular} & \begin{tabular}[c]{@{}c@{}}MP-MFNIP-R 2nd\\ Int Current Victim\end{tabular} & \begin{tabular}[c]{@{}c@{}}MP-MFNIP-R 2nd\\ Int Prosp Victim\end{tabular} \\ \hline
8 & 0 & 1 & 2 & 3 & 1 & 0 & 0 & 4 & 2 & 3 \\ \hline
12 & 0 & 1 & 4 & 2 & 3 & 0 & 0 & 6 & 2 & 3 \\ \hline
16 & 0 & 1 & 6 & 3 & 1 & 0 & 0 & 8 & 2 & 3 \\ \hline
20 & 0 & 1 & 8 & 3 & 1 & 0 & 0 & 10 & 2 & 3 \\ \hline
24 & 0 & 2 & 12 & 2 & 3 & 0 & 0 & 12 & 2 & 4 \\ \hline
28 & 0 & 1 & 14 & 2 & 3 & 0 & 0 & 14 & 2 & 4 \\ \hline
32 & 0 & 1 & 14 & 2 & 4 & 0 & 0 & 16 & 2 & 4 \\ \hline
36 & 0 & 2 & 14 & 2 & 4 & 0 & 0 & 18 & 2 & 4 \\ \hline
40 & 0 & 3 & 14 & 2 & 4 & 0 & 0 & 20 & 0 & 8 \\ \hline
\end{tabular}}
\caption{Recommended delayed interdictions on network 1 with two attackers}
\label{tbl:net1_d_2int_who_text}

\vspace{2\baselineskip}
\resizebox{\textwidth}{!}{
\begin{tabular}{|c|c|c|c|c|c|c|c|c|c|c|}
\hline
Budget & \begin{tabular}[c]{@{}c@{}}MP-MFNIP\\ Int Trafficker\end{tabular} & \begin{tabular}[c]{@{}c@{}}MP-MFNIP\\ Int Bottom\end{tabular} & \begin{tabular}[c]{@{}c@{}}MP-MFNIP\\ Int Victim\end{tabular} & \begin{tabular}[c]{@{}c@{}}MP-MFNIP\\ Int Current Vic\end{tabular} & \begin{tabular}[c]{@{}c@{}}MP-MFNIP 2nd\\ Int Prosp Vic\end{tabular} & \begin{tabular}[c]{@{}c@{}}MP-MFNIP-R 2nd\\ Int Trafficker\end{tabular} & \begin{tabular}[c]{@{}c@{}}MP-MFNIP-R\\ Int Bottom\end{tabular} & \begin{tabular}[c]{@{}c@{}}MP-MFNIP-R\\ Int Victim\end{tabular} & \begin{tabular}[c]{@{}c@{}}MP-MFNIP-R 2nd\\ Int Current Victim\end{tabular} & \begin{tabular}[c]{@{}c@{}}MP-MFNIP-R 2nd\\ Int Prosp Victim\end{tabular} \\ \hline
8 & 0 & 2 & 0 & 3 & 1 & 0 & 0 & 4 & 2 & 3 \\ \hline
12 & 0 & 2 & 2 & 3 & 1 & 0 & 0 & 6 & 2 & 3 \\ \hline
16 & 0 & 2 & 4 & 3 & 1 & 0 & 0 & 8 & 2 & 4 \\ \hline
20 & 0 & 2 & 6 & 3 & 1 & 0 & 0 & 10 & 2 & 4 \\ \hline
24 & 1 & 2 & 6 & 3 & 1 & 0 & 0 & 12 & 2 & 4 \\ \hline
28 & 0 & 1 & 12 & 2 & 4 & 0 & 0 & 14 & 2 & 4 \\ \hline
32 & 1 & 3 & 8 & 2 & 4 & 0 & 0 & 16 & 2 & 4 \\ \hline
36 & 2 & 3 & 7 & 2 & 4 & 0 & 0 & 18 & 1 & 7 \\ \hline
40 & 1 & 1 & 14 & 2 & 4 & 0 & 0 & 20 & 1 & 7 \\ \hline
\end{tabular}
}
\caption{Recommended delayed interdictions on network 3 with two attackers}
\label{tbl:net3_d_2int_who_text}
\end{center}
\end{sidewaystable}

\begin{table}[h!]
\begin{center}
{\scriptsize %
\begin{tabular}{|c|c|c|c|c|c|c|}
\hline
Budget & \begin{tabular}[c]{@{}c@{}}MP-MFNIP\\ Int Trafficker\end{tabular} & \begin{tabular}[c]{@{}c@{}}MP-MFNIP\\ Int Bottom\end{tabular} & \begin{tabular}[c]{@{}c@{}}MP-MFNIP\\ Int Victim\end{tabular} & \begin{tabular}[c]{@{}c@{}}MP-MFNIP-R\\ Int Trafficker\end{tabular} & \begin{tabular}[c]{@{}c@{}}MP-MFNIP-R\\ Int Bottom\end{tabular} & \begin{tabular}[c]{@{}c@{}}MP-MFNIP-R\\ Int Victim\end{tabular} \\ \hline
8 & 0 & 2 & 0 & 0 & 0 & 4 \\ \hline
12 & 1 & 1 & 1 & 1 & 0 & 4 \\ \hline
16 & 1 & 2 & 1 & 1 & 0 & 6 \\ \hline
20 & 1 & 3 & 1 & 1 & 0 & 8 \\ \hline
24 & 0 & 2 & 8 & 1 & 0 & 10 \\ \hline
28 & 0 & 2 & 10 & 1 & 1 & 10 \\ \hline
32 & 0 & 2 & 12 & 2 & 2 & 8 \\ \hline
36 & 1 & 2 & 10 & 2 & 2 & 10 \\ \hline
40 & 1 & 2 & 12 & 2 & 3 & 10 \\ \hline
\end{tabular}
}%
\caption{Recommended upfront interdictions on network 1 with one attacker}
\label{tbl:net1_u_1int_who_text}
\end{center}
\end{table}

\begin{table}[h!]
\begin{center}
{\scriptsize %
\begin{tabular}{|c|c|c|c|c|c|c|}
\hline
Budget & \begin{tabular}[c]{@{}c@{}}MP-MFNIP\\ Int Trafficker\end{tabular} & \begin{tabular}[c]{@{}c@{}}MP-MFNIP\\ Int Bottom\end{tabular} & \begin{tabular}[c]{@{}c@{}}MP-MFNIP\\ Int Victim\end{tabular} & \begin{tabular}[c]{@{}c@{}}MP-MFNIP-R\\ Int Trafficker\end{tabular} & \begin{tabular}[c]{@{}c@{}}MP-MFNIP-R\\ Int Bottom\end{tabular} & \begin{tabular}[c]{@{}c@{}}MP-MFNIP-R\\ Int Victim\end{tabular} \\ \hline
8 & 0 & 2 & 0 & 0 & 0 & 4 \\ \hline
12 & 0 & 3 & 0 & 1 & 1 & 2 \\ \hline
16 & 0 & 2 & 4 & 1 & 0 & 6 \\ \hline
20 & 1 & 2 & 3 & 1 & 0 & 8 \\ \hline
24 & 1 & 2 & 5 & 1 & 0 & 10 \\ \hline
28 & 2 & 2 & 5 & 1 & 0 & 12 \\ \hline
32 & 2 & 3 & 5 & 1 & 0 & 14 \\ \hline
36 & 2 & 2 & 9 & 1 & 0 & 16 \\ \hline
40 & 2 & 3 & 9 & 1 & 0 & 18 \\ \hline
\end{tabular}
}%
\caption{Recommended upfront interdictions on network 3 with one attacker}
\label{tbl:net3_u_1int_who_text}
\end{center}
\end{table}

\begin{sidewaystable}[]
\begin{center}
\resizebox{\textwidth}{!}{
\begin{tabular}{|c|c|c|c|c|c|c|c|c|c|c|}
\hline
Budget & \begin{tabular}[c]{@{}c@{}}MP-MFNIP\\ Int Trafficker\end{tabular} & \begin{tabular}[c]{@{}c@{}}MP-MFNIP\\ Int Bottom\end{tabular} & \begin{tabular}[c]{@{}c@{}}MP-MFNIP\\ Int Victim\end{tabular} & \begin{tabular}[c]{@{}c@{}}MP-MFNIP\\ Int Current Vic\end{tabular} & \begin{tabular}[c]{@{}c@{}}MP-MFNIP\\ Int Prosp Vic\end{tabular} & \begin{tabular}[c]{@{}c@{}}MP-MFNIP-R\\ Int Trafficker\end{tabular} & \begin{tabular}[c]{@{}c@{}}MP-MFNIP-R\\ Int Bottom\end{tabular} & \begin{tabular}[c]{@{}c@{}}MP-MFNIP-R\\ Int Victim\end{tabular} & \begin{tabular}[c]{@{}c@{}}MP-MFNIP-R 2nd\\ Int Current Victim\end{tabular} & \begin{tabular}[c]{@{}c@{}}MP-MFNIP-R 2nd\\ Int Prosp Victim\end{tabular} \\ \hline
8 & 0 & 2 & 0 & 3 & 1 & 0 & 0 & 4 & 2 & 3 \\ \hline
12 & 0 & 2 & 2 & 2 & 2 & 0 & 0 & 6 & 2 & 3 \\ \hline
16 & 0 & 2 & 4 & 2 & 3 & 0 & 0 & 8 & 2 & 3 \\ \hline
20 & 0 & 2 & 6 & 2 & 3 & 0 & 0 & 10 & 2 & 3 \\ \hline
24 & 0 & 2 & 8 & 2 & 3 & 0 & 0 & 12 & 2 & 4 \\ \hline
28 & 0 & 2 & 10 & 2 & 3 & 0 & 0 & 14 & 2 & 4 \\ \hline
32 & 0 & 3 & 10 & 2 & 3 & 0 & 0 & 16 & 2 & 4 \\ \hline
36 & 0 & 2 & 14 & 2 & 4 & 0 & 0 & 18 & 2 & 4 \\ \hline
40 & 0 & 3 & 14 & 2 & 4 & 0 & 0 & 20 & 0 & 8 \\ \hline
\end{tabular}
}
\caption{Recommended upfront interdictions on network 1 with two attackers}
\label{tbl:net1_u_2int_who_text}

\vspace{2\baselineskip}
\resizebox{\textwidth}{!}{
\begin{tabular}{|c|c|c|c|c|c|c|c|c|c|c|}
\hline
Budget & \begin{tabular}[c]{@{}c@{}}MP-MFNIP\\ Int Trafficker\end{tabular} & \begin{tabular}[c]{@{}c@{}}MP-MFNIP\\ Int Bottom\end{tabular} & \begin{tabular}[c]{@{}c@{}}MP-MFNIP\\ Int Victim\end{tabular} & \begin{tabular}[c]{@{}c@{}}MP-MFNIP\\ Int Current Vic\end{tabular} & \begin{tabular}[c]{@{}c@{}}MP-MFNIP\\ Int Prosp Vic\end{tabular} & \begin{tabular}[c]{@{}c@{}}MP-MFNIP-R\\ Int Trafficker\end{tabular} & \begin{tabular}[c]{@{}c@{}}MP-MFNIP-R\\ Int Bottom\end{tabular} & \begin{tabular}[c]{@{}c@{}}MP-MFNIP-R\\ Int Victim\end{tabular} & \begin{tabular}[c]{@{}c@{}}MP-MFNIP-R\\ Int Current Victim\end{tabular} & \begin{tabular}[c]{@{}c@{}}MP-MFNIP-R\\ Int Prosp Victim\end{tabular} \\ \hline
8 & 0 & 2 & 0 & 3 & 1 & 0 & 0 & 4 & 2 & 3 \\ \hline
12 & 0 & 2 & 2 & 3 & 1 & 0 & 0 & 6 & 2 & 3 \\ \hline
16 & 1 & 2 & 1 & 3 & 1 & 0 & 0 & 8 & 2 & 4 \\ \hline
20 & 2 & 2 & 1 & 3 & 1 & 0 & 0 & 10 & 2 & 4 \\ \hline
24 & 1 & 2 & 5 & 3 & 1 & 0 & 0 & 12 & 2 & 4 \\ \hline
28 & 2 & 2 & 5 & 3 & 1 & 0 & 0 & 14 & 2 & 4 \\ \hline
32 & 1 & 3 & 7 & 2 & 4 & 0 & 0 & 16 & 2 & 4 \\ \hline
36 & 2 & 3 & 7 & 2 & 4 & 0 & 0 & 18 & 1 & 7 \\ \hline
40 & 2 & 3 & 9 & 1 & 5 & 0 & 0 & 20 & 1 & 7 \\ \hline
\end{tabular}
}
\caption{Recommended upfront interdictions on network 3 with two attackers}
\label{tbl:net3_u_2int_who_text}
\end{center}
\end{sidewaystable}

In the case of delayed interdiction with one attacker, multi-period MFNIP prioritizes interdicting bottoms in both networks, while focusing more on interdicting victims in network $1$, and more on interdicting traffickers in network $3$. Since network $3$ has more victims, it would be reasonable that disrupting traffickers would reduce the flow more, even accounting for the time delay. However, in both networks, MP-MFNIP-R consistently allocates a significant portion of the attacker budget to disrupting victims. This is likely due to the inclusion of back-up traffickers that the multi-period MFNIP is unable to account for. These trends continue for the case with two attackers. At lower budget levels, the second attacker spends more budget on interdicting current victims, and shifts their focus to preventing the recruitment of new victims at higher budget levels. In both networks, the number of prospective victims that the second attacker prevents their recruitment is double that of the number of current victims they interdict when the primary attacker has a budget of at least $24$. As the attacker budget increases, the first attacker is able to disrupt more victims, allowing for more opportunities for recruitment, making the prevention of recruitment a higher priority. This accounts for the decrease in restructured flow seen in Figures \ref{fig:flowplots_d_2} and \ref{fig:flowplots_u_2}.

In the case of upfront interdiction with one attacker, the recommended interdictions from MP-MFNIP-R in network $1$ drastically shift towards interdicting traffickers and bottoms, while the recommended interdictions from MP-MFNIP-R in network $3$ are almost completely consistent with the recommendations from the delayed interdiction model. We additionally have that, when there are two attackers, the first attacker's interdictions increasingly focus on disrupting victims currently in the network, while the second attacker again focuses on the prevention of recruitment. We again observe that, in both networks, the number of prospective victims that the second attacker prevents their recruitment is double that of the number of current victims they interdict when the primary attacker has a budget of at least $24$. This suggests that disrupting the ability of traffickers to recruit new victims will be key in effectively disrupting their operations. Currently, macroscopic models focus on how victims are transported between locations to move from their initial location to where demand is. Our results suggest that macroscopic network interdiction models should also focus on understanding which communities victims of trafficking are more prominently recruited from and how resources should be allocated to those susceptible communities to reduce the population's vulnerabilities to being trafficked, as opposed to solely focusing on movement between locations.

\section{Conclusion and Future Work}
\label{sec:conc}
We introduced the multi-period max flow network interdiction problem with restructuring (MP-MFNIP-R), where flow is sent from the source node to sink node in each time period, and interdictions and restructurings are decided upon upfront and implemented throughout the time horizon. We motivated this problem with applications in disrupting domestic sex trafficking networks, where flow is defined as the ability of a trafficker to control their victims. We modeled this problem as a BMILP, and derived a column-and-constraint generation (C\&CG) algorithm to solve this problem. Modeling-specific augmentations were incorporated in the C\&CG algorithm to significantly improve the solve time. For models where the interdictions are also implemented upfront, we proposed equivalent models that are able to significantly reduce the size of the problem. We additionally proposed additional models that include a second attacker that has the ability to interdict victims currently in the network, as well as to prevent the recruitment of new victims, modeling the ability of social services to disrupt a sex trafficking network. 

We tested our model on validated synthetic domestic sex trafficking networks with $5$ single-trafficker operations. We note that the benefits of all interdictions occurring upfront was more noticeable at higher attacker budget levels. Our work also supports that coordination between anti-trafficking stakeholders results in more effective disruption of flow. The inclusion of a second attacker with the ability to prevent recruitment is key proved vital to successfully reducing the flow after restructurings. By preventing recruitment, the resources spent on disrupting victims means that these victims cannot be replaced by traffickers. This suggested that future research, both qualitative and quantitative, should focus on exploring how to disrupt recruitment in sex trafficking networks. This work is speculative and more empirical data is needed to understand the true impacts of our mathematical analysis.

There is much work needed to more accurately apply network interdiction models to disrupting domestic sex trafficking networks. A first avenue of research is improving the dynamics proposed in this model to allow for interdiction and restructurings to occur in any time period, as decided by the attacker and defender. Bilevel mixed integer linear programs with integer decision variables in both levels of the problem are already computationally challenging, so extending this model to a multi-level mixed integer linear program with integer decision variables in every level will require significant computational advances. Additionally, as with any illicit network, there is significant uncertainty involved in learning the network structure. Future work can integrate a learning aspect to the model, where as the attacker learns more about the participants in the network and how it adapts as interdictions are implemented.

\section*{Acknowledgements}
We recognize that our research cannot capture all the complexities of the lived experiences of trafficking victims and survivors.  We acknowledge the significant contributions of our survivor-centered advisory group including: Tonique Ayler, Breaking Free, Housing Advocate, Survivor Leader; Terry Forliti, Independent Consultant, Survivor Leader;  Joy Friedman, Independent Consultant, Survivor Leader; Mikki Mariotti, Director, The PRIDE Program of the Family Partnership; Christine Nelson, Independent Consultant, Fellow of Survivor Alliance, Survivor Leader; Lorena Nevile, Vice President of Programs, The Family Partnership; and Drea Sortillon, Witness-Victim Division, Hennepin County Attorney’s Office.  They have provided critical expertise in understanding the complexities of trafficking networks.

\section*{Declarations}
\textbf{Funding} This material is based upon work supported by the National Science Foundation (NSF) under Grant No. 1838315 and the National Institute of Justice (NIJ) under Grant No. 2020-R2-CX-0022. The opinions expressed in the paper do not necessarily reflect the views of the NSF or NIJ.

\noindent\textbf{Conflicts of interest} The authors have no relevant financial or non-financial interests to disclose.

\bibliographystyle{agsm}
\bibliography{ref.bib}

@article{fedina2015use,
  title={Use and misuse of research in books on sex trafficking: Implications for interdisciplinary researchers, practitioners, and advocates},
  author={Fedina, Lisa},
  journal={Trauma, Violence, \& Abuse},
  volume={16},
  number={2},
  pages={188--198},
  year={2015},
  publisher={Sage Publications Sage CA: Los Angeles, CA},
  note={\url{https://doi.org/10.1177/1524838014523337}}
}

@book{ilo2014,
    title = "Profits and Poverty: the Economics of Forced Labour",
    author = "{de Cock}, Micha{\"e}lle and Maame Woode",
    year = "2014",
    language = "English",
    isbn = "9789221287810",
    publisher = "International Labour Office",
    address="Geneva, Switzerland"
}

@misc{polaris,
  title={The typology of modern slavery: Defining sex and labor trafficking in the {U}nited {S}tates},
  author={Anthony, Brittany and Penrose, Jennifer Kimball and Jakiel, Sarah and Couture, Tessa and Crowe, Sara and Fowler, Megan and Keyhan, Rochelle and Sorenson, Keeli and Myles Bradley and Badavi, Mary Ann},
  year={2017},
  note={Retrieved from \url{https://polarisproject.org/wp-content/uploads/2019/09/Polaris-Typology-of-Modern-Slavery-1.pdf}}
}

@misc{carpenter2016nature,
  title={The nature and extent of gang involvement in sex trafficking in {San} {Diego} {County}},
  author={Carpenter, Ami and Gates, Jamie},
  year={2016},
  note={Retrieved from \url{https://www.ojp.gov/ncjrs/virtual-library/abstracts/nature-and-extent-gang-involvement-sex-trafficking-san-diego-county}}
}

@article{preble2019under,
  title={Under their 'control': Perceptions of traffickers’ power and coercion among international female trafficking survivors during exploitation},
  author={Preble, Kathleen M},
  journal={Victims \& Offenders},
  volume={14},
  number={2},
  pages={199--221},
  year={2019},
  publisher={Taylor \& Francis},
  note={\url{https://doi.org/10.1080/15564886.2019.1567637}}
}

@unpublished{martin2014mapping,
  title={Mapping the Market for Sex with Minor Trafficked Girls in {M}inneapolis: Structures Functions, and Patterns},
  author={Martin, Lauren E and Pierce, Alexandra and Peyton, Stephen and Gabilondo, Anna Isobel and Tulpule, Girija},
  year={2014},
  note={Community Report, University of Minnesota, Minneapolis, MN}
}

@misc{martin2017mapping,
  title={Mapping the demand: Sex buyers in the state of {Minnesota}},
  author={Martin, Lauren and Melander, Christina and Karnik, Harshada and Nakamura, Corelle},
  year={2017},
  note={Retrieved from \url{https://conservancy.umn.edu/bitstream/handle/11299/226521/MappingtheDemand-FullReport\%20-\%20FINAL\%20July\%2031\%202017.pdf?sequence=1}}
}

@article{marcus2016pimping,
  title={Pimping and profitability: Testing the economics of trafficking in street sex markets in {Atlantic} {City}, {New} {Jersey}},
  author={Marcus, Anthony and Sanson, Jo and Horning, Amber and Thompson, Efram and Curtis, Ric},
  journal={Sociological Perspectives},
  volume={59},
  number={1},
  pages={46--65},
  year={2016},
  publisher={SAGE Publications Sage CA: Los Angeles, CA},
  note={\url{https://doi.org/10.1177/0731121416628552}}
}

@misc{dank2014estimating,
  title={Estimating the size and structure of the underground commercial sex economy in eight major {US} cities},
  author={Dank, Meredith and Khan, B and Downey, PM and Kotonias, C and Mayer, D and Owens, C and Yu, L},
  year={2014},
  note={Retrieved from \url{https://www.ojp.gov/pdffiles1/nij/grants/245295.pdf}}
}

@article{fedina2019risk,
  title={Risk factors for domestic child sex trafficking in the {U}nited {S}tates},
  author={Fedina, Lisa and Williamson, Celia and Perdue, Tasha},
  journal={Journal of Interpersonal Violence},
  volume={34},
  number={13},
  pages={2653--2673},
  year={2019},
  publisher={Sage Publications Sage CA: Los Angeles, CA},
  \note={\url{https://doi.org/10.1177/2F0886260516662306}}
}

@article{gerassi2017design,
  title={Design strategies from sexual exploitation and sex work studies among women and girls: Methodological considerations in a hidden and vulnerable population},
  author={Gerassi, Lara and Edmond, Tonya and Nichols, Andrea},
  journal={Action Research},
  volume={15},
  number={2},
  pages={161--176},
  year={2017},
  publisher={SAGE Publications Sage UK: London, England},
  note={\url{https://doi.org/10.1177/1476750316630387}}
}

@article{weitzer2014new,
  title={New directions in research on human trafficking},
  author={Weitzer, Ronald},
  journal={The ANNALS of the American Academy of Political and Social Science},
  volume={653},
  number={1},
  pages={6--24},
  year={2014},
  publisher={Sage Publications Sage CA: Los Angeles, CA},
  note={\url{https://doi.org/10.1177/0002716214521562}}
}

@article{denton2016anatomy,
  title={Anatomy of offending: Human trafficking in the {U}nited {S}tates, 2006--2011},
  author={Denton, Erin},
  journal={Journal of Human Trafficking},
  volume={2},
  number={1},
  pages={32--62},
  year={2016},
  publisher={Taylor \& Francis},
  note={\url{https://doi.org/10.1080/23322705.2016.1136540}}
}

@book{cockbain2018offender,
  title={Offender and Victim Networks in Human Trafficking},
  author={Cockbain, Ella},
  year={2018},
  publisher={Routledge},
  address={London, England, UK}
}

@article{martin2014benefit,
author = { Lauren Martin and Richard Lotspeich },
title = {A benefit-cost framework for early intervention to prevent sex trading},
journal = {Journal of Benefit Cost Study},
volume = {5},
number = {1},
pages = {43-87},
year  = {2014},
note={\url{https://doi.org/10.1515/jbca-2013-0021}}
}

@article{konrad2017overcoming,
  title={Overcoming human trafficking via operations research and analytics: Opportunities for methods, models, and applications},
  author={Konrad, Renata A and Trapp, Andrew C and Palmbach, Timothy M and Blom, Jeffrey S},
  journal={Eur. J. Oper. Res.},
  volume={259},
  number={2},
  pages={733--745},
  year={2017},
  publisher={Elsevier},
  note={\url{https://doi.org/10.1016/j.ejor.2016.10.049}}
}

@article{caulkins2019call,
  title={A Call to the Engineering Community to Address Human Trafficking},
  author={Caulkins, Jonathan P and Kammer-Kerwick, Matt and Konrad, Renata and Maass, Kayse Lee and Martin, Lauren and Sharkey, Thomas C.},
  journal={The Bridge},
  volume = {49},
  number = {3},
  pages = {67-73},
  year={2019},
  publisher={National Academy of Engineering}
}

@book{dewey2018routledge,
  title={Routledge international handbook of sex industry research},
  author={Dewey, Susan and Crowhurst, Isabel and Izugbara, Chimaraoke},
  year={2018},
  publisher={Routledge}
}

@article{ulloa2016prevalence,
  title={Prevalence and correlates of sex exchange among a nationally representative sample of adolescents and young adults},
  author={Ulloa, Emilio and Salazar, Marissa and Monjaras, Lidia},
  journal={Journal of Child Sexual Abuse},
  volume={25},
  number={5},
  pages={524--537},
  year={2016},
  publisher={Taylor \& Francis},
  note={\url{https://doi.org/10.1080/10538712.2016.1167802}}
}

@article{franchino2021vulnerabilities,
  title={Vulnerabilities relevant for commercial sexual exploitation of children/domestic minor sex trafficking: A systematic review of risk factors},
  author={Franchino-Olsen, Hannabeth},
  journal={Trauma, Violence, \& Abuse},
  volume={22},
  number={1},
  pages={99--111},
  year={2021},
  publisher={SAGE Publications Sage CA: Los Angeles, CA},
  note={\url{https://doi.org/10.1177/1524838018821956}}
}

@incollection{farrell2020measuring,
  title={Measuring the nature and prevalence of human trafficking},
  author={Farrell, Amy and de Vries, Ieke},
  booktitle={The Palgrave International Handbook of Human Trafficking},
  editor={Winterdyk, John and Jones, Jackie},
  pages={147--162},
  year={2020},
  publisher={Palgrave Macmillan},
  address={London, England, UK}
}

@article{farrell2015police,
  title={Police perceptions of human trafficking},
  author={Farrell, Amy and Pfeffer, Rebecca and Bright, Katherine},
  journal={Journal of Crime and Justice},
  volume={38},
  number={3},
  pages={315--333},
  year={2015},
  publisher={Taylor \& Francis},
  note={\url{https://doi.org/10.1080/0735648X.2014.995412}}
}

@article{roby2017federal,
  title={Federal and state responses to domestic minor sex trafficking: the evolution of policy},
  author={Roby, Jini L and Vincent, Melanie},
  journal={Social Work},
  volume={62},
  number={3},
  pages={201--210},
  year={2017},
  publisher={Oxford University Press},
  note={\url{https://doi.org/10.1093/sw/swx026}}
}

@article{surtees2008traffickers,
author = {Rebecca Surtees},
title ={Traffickers and Trafficking in Southern and Eastern {E}urope: Considering the Other Side of Human Trafficking},
journal = {European Journal of Criminology},
volume = {5},
number = {1},
pages = {39-68},
year = {2008},
note={\url{https://doi.org/10.1177/1477370807084224}}
}

@unpublished{dimas2021survey,
  title={A Survey of Operations Research and Analytics Literature Related to Anti-Human Trafficking},
  author={Dimas, Geri L and Konrad, Renata A and Maass, Kayse Lee and Trapp, Andrew C},
  note={preprint, arXiv:2103.16476 [cs.CY]},
  year={2021}
}

@book{stackelberg1952theory,
  title={Theory of the Market Economy},
  author={Stackelberg, Heinrich von},
  year={1952},
  publisher={Oxford University Press},
  address={Oxford, England, UK}
}

@article{salmeron2009worst,
  title={Worst-case interdiction analysis of large-scale electric power grids},
  author={Salmeron, Javier and Wood, Kevin and Baldick, Ross},
  journal={IEEE Transactions on Power Systems},
  volume={24},
  number={1},
  pages={96--104},
  year={2009},
  publisher={IEEE},
  note={\url{https://doi.org/10.1109/TPWRS.2008.2004825}}
}

@article{baycik2018interdicting,
author = {Nail Orkun Baycik and Thomas C. Sharkey and Chase E. Rainwater},
title = {Interdicting layered physical and information flow networks},
journal = {IISE Transactions},
volume = {50},
number = {4},
pages = {316-331},
year  = {2018},
publisher = {Taylor & Francis},
note = {\url{https://doi.org/10.1080/24725854.2017.1401754}}
}

@inproceedings{alderson2011solving,
  title={Solving defender-attacker-defender models for infrastructure defense},
  author={Alderson, David L and Brown, Gerald G and Carlyle, W Matthew and Wood, R Kevin},
  booktitle={12th INFORMS Computing Society Conference, Monterey, CA, USA},
  year={2011},
  publisher={INFORMS},
  pages={28-4-9}
}

@article{malaviya2012multi,
author = { Ajay   Malaviya  and  Chase   Rainwater  and  Thomas C.  Sharkey },
title = {Multi-period network interdiction problems with applications to city-level drug enforcement},
journal = {IIE Transactions},
volume = {44},
number = {5},
pages = {368-380},
year  = {2012},
publisher = {Taylor & Francis},
note = {\url{https://doi.org/10.1080/0740817X.2011.602659}}
}

@article{smith2020survey,
  title={A survey of network interdiction models and algorithms},
  author={Smith, J Cole and Song, Yongjia},
  journal={European Journal of Operational Research},
  volume={283},
  number={3},
  pages={797-811},
  year={2020},
  publisher={Elsevier},
  note={\url{https://doi.org/10.1016/j.ejor.2019.06.024}}
}

@unpublished{kosmas2020interdicting,
  title={Interdicting Restructuring Networks with Applications in Illicit Trafficking},
  author={Kosmas, Daniel and Sharkey, Thomas C and Mitchell, John E and Maass, Kayse Lee and Martin, Lauren},
  note={preprint, arXiv:2011.07093 [math.OC]},
  year={2020}
}

@unpublished{kosmas2022generating,
  title={Generating Synthetic but Realistic Human Trafficking Networks for Modeling Disruptions through Transdisciplinary and Community-Based Action Research},
  author={Kosmas, Daniel and Melander, Christina and Singerhouse, Emily and Sharkey, Thomas C and Maass, Kayse Lee and Barrick, Kelle and Martin, Lauren},
  note={preprint, arXiv:2203.01893 [cs.SI]},
  year={2022}
}

@article{wood1993deterministic,
 author = {Wood, R.Kevin},
 title = {Deterministic Network Interdiction},
 journal = {Mathematical and Computer Modelling},
 issue_date = {January, 1993},
 volume = {17},
 number = {2},
 month = jan,
 year = {1993},
 issn = {0895-7177},
 pages = {1-18},
 numpages = {18},
 note = {\url{http://dx.doi.org/10.1016/0895-7177(93)90236-R}},
 doi = {10.1016/0895-7177(93)90236-R},
 acmid = {2262505},
 publisher = {Elsevier Science Publishers B. V.},
 address = {Amsterdam, The Netherlands, The Netherlands},
}

@techreport{derbes1997efficiently,
  title={Efficiently Interdicting a Time-Expanded Transshipment Network.},
  author={Derbes, H Dan},
  year={1997},
  institution={Naval Postgraduate School, Monterey, CA, USA}
}

@article{rad2013maximum,
  title={Maximum dynamic network flow interdiction problem: New formulation and solution procedures},
  author={Rad, M Afshari and Kakhki, H Taghizadeh},
  journal={Computers \& Industrial Engineering},
  volume={65},
  number={4},
  pages={531-536},
  year={2013},
  publisher={Elsevier},
  note={\url{https://doi.org/10.1016/j.cie.2013.04.014}}
}

@inproceedings{zheng2012stochastic,
  title={Stochastic dynamic network interdiction games},
  author={Zheng, Jiefu and Casta{\~n}{\'o}n, David A},
  booktitle={2012 American Control Conference (ACC)},
  pages={1838--1844},
  year={2012},
  organization={IEEE}
}

@article{soleimani2017solving,
  title={Solving multi-period interdiction via generalized Bender’s decomposition},
  author={Soleimani-Alyar, Maryam and Ghaffari-Hadigheh, Alireza},
  journal={Acta Mathematicae Applicatae Sinica, English Series},
  volume={33},
  number={3},
  pages={633--644},
  year={2017},
  publisher={Springer},
  note={\url{https://doi.org/10.1007/s10255-017-0687-9}}
}

@article{soleimani2018dynamic,
  title={Dynamic network interdiction problem with uncertain data},
  author={Soleimani-Alyar, Maryam and Ghaffari-Hadigheh, Alireza},
  journal={International Journal of Uncertainty, Fuzziness and Knowledge-Based Systems},
  volume={26},
  number={02},
  pages={327--342},
  year={2018},
  publisher={Springer},
  note={\url{https://doi.org/10.1142/S0218488518500174}}
}

@article{jabarzare2020dynamic,
  title={Dynamic interdiction networks with applications in illicit supply chains},
  author={Jabarzare, Ziba and Zolfagharinia, Hossein and Najafi, Mehdi},
  journal={Omega},
  volume={96},
  pages={1--22},
  year={2020},
  publisher={Elsevier},
  note={\url{https://doi.org/10.1016/j.omega.2019.05.005}}
}

@article{holzmann2019shortest,
  title={Shortest path interdiction problem with arc improvement recourse: A multiobjective approach},
  author={Holzmann, Tim and Smith, J Cole},
  journal={Naval Research Logistics (NRL)},
  volume={66},
  number={3},
  pages={230--252},
  year={2019},
  publisher={Wiley Online Library},
  note={\url{https://doi.org/10.1002/nav.21839}}
}

@incollection{mayorga2019countering,
  title={Countering human trafficking Using {ISE/OR} techniques},
  author={Mayorga, Maria and Tateosian, Laura and Velasquez, German and Amindarbari, Reza and Caltagirone, Sherrie},
  booktitle={Emerging Frontiers in Industrial and Systems Engineering},
  editor={Nembhard, Harriet B and Cudney, Elizabeth A and Coperich, Katherine M},
  pages={237--257},
  year={2019},
  publisher={CRC Press, Boca Raton, FL, USA}
}

@misc{tezcan2020human,
  title={Human Trafficking Interdiction with Decision Dependent Success},
  author={Tezcan, Baris and Maass, Kayse Lee},
  year={2020},
  howpublished={\emph{engrXiv preprint 1068}}
}

@article{xie2022interdependent,
  title={An interdependent network interdiction model for disrupting sex trafficking networks},
  author={Xie, Xiaodan and Aros-Vera, Felipe},
  journal={Production and Operations Management},
  year={2022},
  publisher={Wiley Online Library},
  note={\url{https://doi.org/10.1111/poms.13713}}
}

@article{zeng2014solving,
  title={Solving bilevel mixed integer program by reformulations and decomposition},
  author={Zeng, Bo and An, Yu},
  journal={Optimization online},
  pages={1-34},
  year={2014}
}

@article{yue2019projection,
  title={A projection-based reformulation and decomposition algorithm for global optimization of a class of mixed integer bilevel linear programs},
  author={Yue, Dajun and Gao, Jiyao and Zeng, Bo and You, Fengqi},
  journal={Journal of Global Optimization},
  volume={73},
  number={1},
  pages={27-57},
  year={2019},
  publisher={Springer},
  note={\url{https://doi.org/10.1007/s10898-018-0679-1}}
}

@article{mccormick1976computability,
  title={Computability of global solutions to factorable nonconvex programs: {P}art {I}—{C}onvex underestimating problems},
  author={McCormick, Garth P},
  journal={Mathematical Programming},
  volume={10},
  number={1},
  pages={147--175},
  year={1976},
  publisher={Springer},
  note={\url{https://doi.org/10.1007/BF01580665}}
}

@article{sefair2016dynamic,
  title={Dynamic shortest-path interdiction},
  author={Sefair, Jorge A and Smith, J Cole},
  journal={Networks},
  volume={68},
  number={4},
  pages={315--330},
  year={2016},
  publisher={Wiley Online Library},
  note={\url{https://doi.org/10.1002/net.21712}}
}

@article{hounmenou2019review,
  title={A review and critique of the US responses to the commercial sexual exploitation of children},
  author={Hounmenou, Charles and O'Grady, Caitlin},
  journal={Children and Youth Services Review},
  volume={98},
  pages={188--198},
  year={2019},
  publisher={Elsevier},
  note={\url{https://doi.org/10.1016/j.childyouth.2019.01.005}}
}

@article{macy2021scoping,
  title={A scoping review of human trafficking screening and response},
  author={Macy, Rebecca J and Klein, LB and Shuck, Corey A and Rizo, Cynthia Fraga and Van Deinse, Tonya B and Wretman, Christopher J and Luo, Jia},
  journal={Trauma, Violence, \& Abuse},
  pages={1--18},
  year={2021},
  publisher={SAGE Publications Sage CA: Los Angeles, CA},
  note={\url{https://doi.org/10.1177/15248380211057273}}
}

@incollection{belles2018defining,
  title={Defining {S}ex {T}rafficking},
  author={Belles, Nita},
  booktitle={Handbook of Sex Trafficking},
  pages={3--8},
  year={2018},
  publisher={Springer}
}

@article{roe2015sexual,
  title={The sexual exploitation of girls in the United States: The role of female pimps},
  author={Roe-Sepowitz, Dominique Eve and Gallagher, James and Risinger, Markus and Hickle, Kristine},
  journal={Journal of Interpersonal Violence},
  volume={30},
  number={16},
  pages={2814--2830},
  year={2015},
  publisher={Sage Publications Sage CA: Los Angeles, CA},
  note={\url{https://doi.org/10.1177/0886260514554292}}
}

@book{foot2015collaborating,
  title={Collaborating against human trafficking: Cross-sector challenges and practices},
  author={Foot, Kirsten},
  year={2015},
  publisher={Rowman \& Littlefield}
}

@article{foot2021outcome,
  title={An Outcome-Centered Comparative Analysis of Counter-Human Trafficking Coalitions in the Global South},
  author={Foot, Kirsten and Sworn, Helen and Alejano-Steele, AnnJanette},
  journal={Management Communication Quarterly},
  volume={35},
  number={3},
  pages={418--444},
  year={2021},
  publisher={SAGE Publications Sage CA: Los Angeles, CA},
  note={\url{https://doi.org/10.1177/08933189211017925}}
}

@article{pajon2022importance,
  title={The importance of multi-agency collaborations during human trafficking criminal investigations},
  author={Paj{\'o}n, Laura and Walsh, Dave},
  journal={Policing and Society},
  pages={1--19},
  year={2022},
  publisher={Taylor \& Francis},
  note={\url{https://doi.org/10.1080/10439463.2022.2106984}}
}

@book{david2008trafficking,
  title={Trafficking of women for sexual purposes},
  author={David, Fiona},
  year={2008},
  publisher={Australian Institute of Criminology Canberra},
  note={Retrieved from \url{https://www.dss.gov.au/sites/default/files/documents/05_2012/rrp95_trafficking_of_women.pdf}}
}

@misc{clawson2008prosecuting,
  title={Prosecuting human trafficking cases: Lessons learned and promising practices},
  author={Clawson, Heather J and Dutch, Nicole and Lopez, Susan and Tiapula, Suzanna},
  year={2008},
  note={Retrieved from \url{https://www.ojp.gov/library/publications/prosecuting-human-trafficking-cases-lessons-learned-and-promising-practices}}
}

@article{moynihan2018interventions,
  title={Interventions that foster healing among sexually exploited children and adolescents: A systematic review},
  author={Moynihan, Melissa and Pitcher, Claire and Saewyc, Elizabeth},
  journal={Journal of child sexual abuse},
  volume={27},
  number={4},
  pages={403--423},
  year={2018},
  publisher={Taylor \& Francis},
  note={\url{https://doi.org/10.1080/10538712.2018.1477220}}
}

@inproceedings{wilt2019measuring,
  title={Measuring the impact of coordination in disrupting illicit trafficking networks},
  author={Wilt, John and Sharkey, Thomas C},
  booktitle={IIE Annual Conference. Proceedings},
  pages={767--772},
  year={2019},
  organization={Institute of Industrial and Systems Engineers (IISE)}
}

@article{sreekumaran2021equilibrium,
  title={Equilibrium strategies for multiple interdictors on a common network},
  author={Sreekumaran, Harikrishnan and Hota, Ashish R and Liu, Andrew L and Uhan, Nelson A and Sundaram, Shreyas},
  journal={European Journal of Operational Research},
  volume={288},
  number={2},
  pages={523--538},
  year={2021},
  publisher={Elsevier},
  note={\url{https://doi.org/10.1016/j.ejor.2020.06.002}}
}

@article{sharkey2021better,
  title={Better together: A transdisciplinary approach to disrupt human trafficking},
  author={Sharkey, Thomas C and Barrick, Kelle and Farrell, Amy and Maass, Kayse Lee and Martin, Lauren and Song, Yongjia},
  journal={ISE Magazine},
  volume = {51},
  number = {11},
  pages={34--39},
  year={2021}
}

@article{martin2022learning,
  title={Learning Each Other’s Language and Building Trust: Community-Engaged Transdisciplinary Team Building for Research on Human Trafficking Operations and Disruption},
  author={Martin, Lauren and Gupta, Mahima and Maass, Kayse L and Melander, Christina and Singerhouse, Emily and Barrick, Kelle and Samad, Tariq and Sharkey, Thomas C and Ayler, Tonique and Forliti, Teresa and others},
  journal={International Journal of Qualitative Methods},
  volume={21},
  pages={1--15},
  year={2022},
  publisher={SAGE Publications Sage CA: Los Angeles, CA}
}

\newpage
\begin{appendices}
\section{Summary of Notation}
\label{note}

\begin{table}[H]
    \centering
    {\scriptsize %
    \begin{tabular}{|c|l|}
        \hline
         Set & Description of Set \\ \hline
         $N$ & set of nodes \\ \hline
         $T$ & trafficker nodes \\ \hline
         $B$ & bottom nodes \\ \hline
         $V$ & victim nodes \\ \hline
         $T^R$ & back-up trafficker nodes \\ \hline
         $B^R$ & victims nodes that can be promoted to the role of bottom \\ \hline
         $V^R$ & recruitable victim nodes \\ \hline
         $A$ & set of arcs currently in the network \\ \hline
         $A^R$ & set of arcs that can be restructured \\ \hline
         $A^{R,out}$ & set of arcs that can be restructured from the tail node \\ \hline
         $A^{R,in}$ & set of arcs that can be restructured from the head node \\ \hline
         $Y$ & set of all feasible interdiction plans \\ \hline
         $Z(y)$ & set of all feasible restructuring plans responding to interdiction plan $y$ \\ \hline
         $S^k$ & set of network phases for restructuring plan $k$ in upfront interdiction model \\ \hline
         $P^k$ & set of recruitable victims that were not recruited in restructuring plan $k$ \\ \hline
         $C^k$ & set of traffickers that have not taken all of their actions in restructuring plan $k$ \\ \hline
         $A^{k,rec}$ & set of restructurable arcs between traffickers in $C^k$ and recruitable victims in $P^k$ \\ \hline
         \end{tabular}
         }%
    \caption{Description of notation for sets}
    \label{tab:noteset}
\end{table}

\begin{table}[H]
    \centering
    {\scriptsize %
    \begin{tabular}{|c|l|}
        \hline
         Variable & Description of Variable \\ \hline
         $y_i$ & indicator of whether node $i$ has been interdicted \\ \hline
         $\gamma_{it}$ & indicator of if node $i$ has been interdicted in or before time $t$ \\ \hline
         $z^{out}$ & indicator of whether arc $(i,j)$ has been ``out" restructured  \\ \hline
         $\zeta^{out}_{ijt} $& indicator of if arc $(i,j)$ has been ``out" restructured in or before time $t$ \\ \hline
         $z^{in}$ & indicator of whether arc $(i,j)$ has been ``in" restructured \\ \hline
         $\zeta^{in}_{ijt} $& indicator of if arc $(i,j)$ has been ``in" restructured in or before time $t$ \\ \hline
         $x_{ijt}$ & amount of flow across arc $(i,j)$ at time $t$\\ \hline
         $x_{it}$ & amount of flow across node $i$ at time $t$ \\ \hline
         $\tilde{r}_i$ & adjusted cost to interdict trafficker $i$ accounting for reductions from other interdictions \\ \hline
         $y'_i$ & indicator of whether node $i$ has been interdicted by second attacker \\ \hline
         $\pi^{+^k}_{it}$ & dual variable for inflow conservation constraint for node $i$ at time $t$ for restructuring plan $k$ \\ \hline
         $\pi^{-^k}_{it}$ & dual variable for outflow conservation constraint for node $i$ at time $t$ for restructuring plan $k$ \\ \hline
         $\theta^k_{ijt}$ & indicator if arc $(i,j)$ in in the minimum cut at time $t$ with restructuring plan $k$\\ \hline
         $\theta^k_{it}$ & indicator if node $i$ in in the minimum cut at time $t$ with restructuring plan $k$\\ \hline
         $w^{out,k}_{ijt}$ & \begin{tabular}[l]{@{}l@{}}indicator of if arc $(i,j)$ is ``out" restructured at or after time period $t$,\\ where $(i,j)$ was ``out" restructured in restructuring plan $k$\end{tabular}\\ \hline
         $w^{in,k}_{ijt}$ & \begin{tabular}[l]{@{}l@{}}indicator of if arc $(i,j)$ is ``in" restructured at or after time period $t$,\\ where $(i,j)$ was ``in" restructured in restructuring plan $k$\end{tabular}\\ \hline
         $w^{out,k}_{ij}$ & \begin{tabular}[l]{@{}l@{}}indicator of if arc $(i,j)$ is ``out" restructured\\ where $(i,j)$ was ``out" restructured in restructuring plan $k$ for upfront interdiction model\end{tabular}\\ \hline
         $w^{in,k}_{ij}$ & \begin{tabular}[l]{@{}l@{}}indicator of if arc $(i,j)$ is ``in" restructured\\ where $(i,j)$ was ``in" restructured in restructuring plan for upfront interdiction model $k$\end{tabular}\\ \hline
         $w^{out,k}_{ijs}$ & \begin{tabular}[l]{@{}l@{}}indicator of if arc $(i,j)$ is ``out" restructured at or after network phase $s$,\\ where $(i,j)$ was ``out" restructured in restructuring plan $k$\end{tabular}\\ \hline
         $w^{in,k}_{ijs}$ & \begin{tabular}[l]{@{}l@{}}indicator of if arc $(i,j)$ is ``in" restructured at or after network phase $s$,\\ where $(i,j)$ was ``in" restructured in restructuring plan $k$\end{tabular}\\ \hline
         $\phi^k_{ijt}$ & \begin{tabular}[l]{@{}l@{}}indicator of if arc $(i,j)$ can be restructured at or after time period $t$\\ in addition to the feasible restructurings in restructuring plan $k$ \end{tabular}\\ \hline
         $\nu_{ki}$ & indicator if trafficker $i$ has performed $c_i^{out}$ actions in augmented restructuring plan $k$ \\ \hline
         $\xi_{ki}$ & indicator if trafficker $i$ has performed all feasible restructurings in augmented restructuring plan $k$ \\ \hline
    \end{tabular}
    }%
    \caption{Description of notation for variables}
    \label{tab:notevar}
\end{table}

\begin{table}[H]
    \centering    
    {\scriptsize %
    \begin{tabular}{|c|l|}
        \hline
         Parameter & Description of Parameter\\ \hline
         $\alpha$ & source node  \\ \hline
         $\omega$ & sink node \\ \hline
         $\tau$ & number of time periods in time horizon\\ \hline
         $\delta_i^y$ & number of time periods needed to interdict node $i$ \\ \hline
         $\delta_{ij}^z$ & number of time periods needed to restructure arc $(i,j)$ \\ \hline
         $u_{ij}$ & capacity of arc $(i,j)$ \\ \hline
         $u_i$ & capacity of node $i$ \\ \hline
         $\tilde{u}_i$ & capacity increase of victim node $i$ being promoted to the role of bottom\\ \hline
         $r_i$ & cost to interdict node $i$ \\ \hline
         $d_{il}$ & reduction in cost to interdict trafficker node $i$ if victim (or bottom) node $l$ is also interdicted\\ \hline
         $r_i^{min}$ & minimum cost to interdict trafficker node $i$ accounting for reductions from other interdictions\\ \hline
         $b$ & attacker budget \\ \hline
         $r'_i$ & cost for second attacker to interdict node $i$ \\ \hline
         $b'$ & attacker budget for second attacker \\ \hline
         $\delta^{min,out}_i$ & the minimum number of time periods need before trafficker node $i$ can initiate a restructuring\\ \hline
         $\delta^{max,out}_i$ & the maximum number of time periods need before trafficker node $i$ can initiate a restructuring\\ \hline
         $\delta^{min,in}_j$ & the minimum number of time periods need before victim node $j$ can initiate a restructuring\\ \hline
         $\delta^{max,in}_j$ & the maximum number of time periods need before victim node $j$ can initiate a restructuring\\ \hline
         $\bar{y}$ & a feasible interdiction plan \\  \hline
         $c^{out}_i$ & number of actions trafficker $i$ can take \\ \hline
         $c^{in}_j$ & number of actions victim $j$ can take \\ \hline
         $\lambda^{out,k}_{it}$ & indicator of if trafficker node $i$ is able to perform all restructurings performed in restructuring plan $k$ by time $t$ \\ \hline
         $\lambda^{in,k}_{jt}$ & indicator of if victim node $j$ is able to perform all restructurings performed in restructuring plan $k$ by time $t$ \\ \hline
         $z^{out,k}_{ij}$ & indicator of whether arc $(i,j)$ has been ``out" restructured in restructuring plan $k$ \\ \hline
         $\zeta^{out,k}_{ijt} $& indicator of if arc $(i,j)$ has been ``out" restructured in or before time $t$ in restructuring plan $k$ \\ \hline
         $z^{in,k}_{ij}$ & indicator of whether arc $(i,j)$ has been ``in" restructured in restructuring plan $k$ \\ \hline
         $\zeta^{in,k}_{ijt} $& indicator of if arc $(i,j)$ has been ``in" restructured in or before time $t$ in restructuring plan $k$ \\ \hline
         $M$ & number of restructuring plans in $\bigcup_{y \in Y} Z(y)$ \\ \hline
         $M_y$ & number of restructuring plans in $Z(y)$ \\ \hline
         $n$ & number of restructuring plans considered in \eqref{tlfFin} \\ \hline
         $U$ & upper bound on objective value of MP-MFNIP-R \\ \hline
         $L$ & lower bound on objective value of MP-MFNIP-R \\ \hline
         $\bar{t}$ & the latest time period any recruitable victim can be recruited in \\ \hline
         $a^k_i$ & number of actions taken be trafficker $i$ in restructuing plan $k$ \\ \hline
         $B$ & arbitrarily large parameter for big-$M$ style constraints \\ \hline
    \end{tabular}
    }%
    \caption{Description of notation for data/parameters}
    \label{tab:notepar}
\end{table} 

\begin{table}[H]
    \centering
    {\scriptsize %
    \begin{tabular}{|c|l|}
        \hline
        Parameter & Description of Parameter \\ \hline
        $\tau_s^k$ & number of time periods spent in network phase $s$ in restructuring plan $k$ \\ \hline
         $\zeta^{out,k}_{ijs} $& indicator of if arc $(i,j)$ has been ``out" restructured in or before network phase $s$ in restructuring plan $k$ \\ \hline
         $\zeta^{in,k}_{ijs} $& indicator of if arc $(i,j)$ has been ``in" restructured in or before network phase $s$ in restructuring plan $k$ \\ \hline
    \end{tabular}
    }%
    \caption{Description of notation for data/parameters specific to network phase model}
    \label{tab:notepar2}
\end{table}

\newpage
\section{Full Model Derivation}
\label{fullderiv}
As per standard column-and-constraint generation, we separate the maximization problem into two problems, where the innermost problem in only the continuous variables, the $x$ variables. This allows us to replace the problem with the dual minimization problem: 

\begin{singlespace}
\begin{subequations}
\label{tlf1}
\footnotesize\begin{align}
    \min_{y, \gamma \in Y} \max_{z, \zeta} \min_{\pi,\theta} ~~~ & \begin{multlined}
    \sum_{t = 1}^{\tau} [ \sum_{i \in N \setminus \{\alpha,\omega\}} u_i \gamma_{it} \theta_{it} + \sum_{(i,j) \in A} u_{ij} \theta_{ijt} + \sum_{(i,j) \in A^{R,out}} u_{ij} \zeta^{out}_{ijt} \theta_{ijt} \\+ \sum_{(i,j) \in A^{R,in}} u_{ij} \zeta^{in}_{ijt} \theta_{ijt} + \sum_{(i,j) \in B^R} \tilde{u_j} \zeta^{out}_{\alpha jt} \theta_{\alpha jt} ]\end{multlined} \nonumber\\ 
    \text{s.t.} \hspace{.050cm }& \pi_{jt}^+ + \theta_{\alpha jt} \ge 1 \\
    &\omit\hfill$ \text{ for } (\alpha,j) \in A \cup A^R, t = 1, \ldots, \tau$\nonumber\\
    & \pi_{jt}^+ - \pi_{it}^- + \theta_{ijt} \ge 0 \\
    & \omit\hfill$\text{ for } (i,j) \in A \cup A^R \text{ s.t. } i \ne \alpha, j \ne \omega, t = 1, \ldots, \tau$\nonumber\\
    & \pi_{it}^- - \pi_{it}^+ + \theta_{it} \ge 0 \\
    & \omit\hfill$\text{ for } i \in N\setminus \{\alpha,\omega\}, t = 1, \ldots, \tau $\nonumber\\
    & -\pi_{it}^- + \theta_{i\omega t} \ge 0 \\
    & \omit\hfill$\text{ for } (i,\omega) \in A \cup A^R, t = 1, \ldots, \tau$\\
    & \theta \ge 0 \\
    & y \in \{0,1\}^{|N|\setminus \{\alpha,\omega\}}  \\
    & \gamma \in \{0,1\}^{|N|\setminus \{\alpha,\omega\} \times \tau}\\
    & \text{Constraints } \eqref{con:intBudgetInModel} - \eqref{con:intTimeInModel} \nonumber
\end{align}
\end{subequations}
\end{singlespace}

As in \cite{kosmas2020interdicting}, we use an equivalent formulation to reduce the number of bilinear terms in the objective function.

\begin{singlespace}
\begin{subequations}
\label{tlf2}
\footnotesize\begin{align}
    \min_{y, \gamma} \max_{z, \zeta} \min_{\pi,\theta} ~~~ & \sum_{t = 1}^{\tau} [\sum_{i \in N \setminus \{\alpha,\omega\}} u_i \theta_{it} + \sum_{(i,j) \in A \cup A^{R,out} \cup A^{R,in}} u_{ij} \theta_{ijt} + \sum_{(i,j) \in B^R} \tilde{u_j} \zeta^{out}_{\alpha jt} \theta_{\alpha jt} ]\nonumber\\  
    \text{s.t.} \hspace{.050cm }& \pi_{jt}^+ + \theta_{\alpha jt} \ge 1 \\
    & \omit\hfill$\text{ for } (\alpha,j) \in A, t = 1, \ldots, \tau$\nonumber\\
    & \pi_{jt}^+ - \pi_{it}^- + \theta_{ijt} \ge 0 \\
    & \omit\hfill$\text{ for } (i,j) \in A \text{ s.t. } i \ne \alpha, j \ne \omega, k = 1, \ldots, n, t = 1, \ldots, \tau$\nonumber\\
    & \pi_{it}^- - \pi_{it}^+ + \theta_{it} \ge - \gamma_{it} \\
    & \omit\hfill$\text{ for } i \in N\setminus \{\alpha,\omega\}, t = 1, \ldots, \tau $\nonumber\\
    & -\pi_{it}^- + \theta_{i\omega t} \ge 0 \\
    & \omit\hfill$\text{ for } (i,\omega) \in A, k = 1, \ldots, n, t = 1, \ldots, \tau$ \nonumber\\
    & \pi_{jt}^+ + \theta_{\alpha jt} \ge \zeta^{out}_{\alpha jt} \\
    & \omit\hfill$\text{ for } (\alpha,j) \in A^{R, out}, t = 1, \ldots, \tau$ \nonumber\\
    & \pi_{jt}^+ - \pi_{it}^- + \theta_{ijt} \ge \zeta^{out}_{ijt} - 1 \\
    & \omit\hfill$\text{ for } (i,j) \in  A^{R, out} \text{ s.t. } i \ne \alpha, j \ne \omega, k = 1, \ldots, n, t = 1, \ldots, \tau$ \nonumber\\
    & \pi_{jt}^+ - \pi_{it}^- + \theta_{ijt} \ge \zeta^{in}_{ijt} - 1 \\
    & \omit\hfill$\text{ for } (i,j) \in  A^{R, in} \text{ s.t. } i \ne \alpha, j \ne \omega, k = 1, \ldots, n, t = 1, \ldots, \tau$ \nonumber\\
    & \theta \ge 0 \\
    & y \in \{0,1\}^{|N|\setminus \{\alpha,\omega\}} \\
    & \gamma \in \{0,1\}^{|N|\setminus \{\alpha,\omega\} \times \tau}\\
    & \text{Constraints } \eqref{con:intBudgetInModel} - \eqref{con:intTimeInModel} \nonumber
\end{align}
\end{subequations}
\end{singlespace}

Note that $Z(y)$ is a finite set, so we can reduce the tri-level problem into a single-minimization problem by enumerating over every feasible $(z,\zeta)$ solution, and enforcing that the objective value of the bilevel problem is at least as large as objective value associated with each $(z,\zeta)$ solution. Let $M_y = |Z(y)|$. The following is a non-standard formulation of the single-level minimization problem.

\begin{singlespace}
\begin{subequations}
\label{tlf3}
\footnotesize\begin{align}
    \min_{y, \gamma, \pi,\theta} ~~~ & \eta \\
    \text{s.t.} \hspace{.050cm } & \eta \ge\sum_{t = 1}^{\tau} [\sum_{i \in N \setminus \{\alpha,\omega\}} u_i \theta^k_{it} + \sum_{(i,j) \in A \cup A^{R,out} \cup A^{R,in}} u_{ij} \theta^k_{ijt} + \sum_{(i,j) \in B^R} \tilde{u_j}  \zeta^{out,k}_{\alpha jt} \theta_{\alpha jt} ] \\
    & \omit\hfill$\text{ for } k = 1, \ldots, M_y$\nonumber\\  
    & \pi_{jt}^{+^k} + \theta^k_{\alpha jt} \ge 1 \\
    & \omit\hfill$\text{ for } (\alpha,j) \in A, k = 1, \ldots, M_y, t = 1, \ldots, \tau$\nonumber\\
    & \pi_{jt}^{+^k} - \pi_{it}^{-^k} + \theta^k_{ijt} \ge 0 \\
    & \omit\hfill$\text{ for } (i,j) \in A \text{ s.t. } i \ne \alpha, j \ne \omega, k = 1, \ldots, M_y, t = 1, \ldots, \tau$ \nonumber\\
    & \pi_{it}^{-^k} - \pi_{it}^{+^k} + \theta^k_{it} \ge - \gamma_{it} \\
    & \omit\hfill$\text{ for } i \in N\setminus \{\alpha,\omega\}, k = 1, \ldots, M_y, t = 1, \ldots, \tau$ \nonumber\\
    & -\pi_{it}^{-^k} + \theta^k_{i\omega t} \ge 0 \\
    & \omit\hfill$\text{ for } (i,\omega) \in A, k = 1, \ldots, M_y, t = 1, \ldots, \tau$ \nonumber\\
    & \pi_{jt}^{+^k} + \theta^k_{\alpha jt} \ge \zeta^{out,k}_{\alpha jt} \\
    & \omit\hfill$\text{ for } (\alpha,j) \in A^{R, out}, k = 1, \ldots, M_y, t = 1, \ldots, \tau$ \nonumber\\
    & \pi_{jt}^{+^k} - \pi_{it}^{-^k} + \theta^k_{ijt} \ge \zeta^{out,k}_{ijt} - 1 \\
    & \omit\hfill$\text{ for } (i,j) \in  A^{R, out} \text{ s.t. } i \ne \alpha, j \ne \omega, k = 1, \ldots, M_y, t = 1, \ldots, \tau$ \nonumber\\
    & \pi_{jt}^{+^k} - \pi_{it}^{-^k} + \theta^k_{ijt} \ge \zeta^{in,k}_{ijt} - 1 \\
    & \omit\hfill$\text{ for } (i,j) \in  A^{R, in} \text{ s.t. } i \ne \alpha, j \ne \omega, k = 1, \ldots, M_y, t = 1, \ldots, \tau$ \nonumber\\
    & \theta \ge 0 \\
    & y \in \{0,1\}^{|N|\setminus \{\alpha,\omega\}} \\
    & \gamma \in \{0,1\}^{|N|\setminus \{\alpha,\omega\} \times \tau} \\
    & \text{Constraints } \eqref{con:intBudgetInModel} - \eqref{con:intTimeInModel} \nonumber
\end{align}
\end{subequations}
\end{singlespace}

To return to a standard formulation, we now include our partial information constraints and variables. The include of these constraints and variables enforce that, for a given $y$ and $(z, \zeta) \notin Z(y)$, we identify another solution $(\bar{z}, \bar{\zeta}) \in Z(y)$ that allows for nonzero components of $(z, \zeta)$ that are feasible with respect to $y$ will remain nonzero. These constraints are (57)-(66). Let $M = \bigcup_{y \in Y} Z(y)$. The following is a standard formulation of the single-level problem with partial information constraints.

\begin{singlespace}
\begin{subequations}
\label{tlf4}
\footnotesize\begin{align}
    \min_{y, \gamma, w \pi,\theta} ~~~ & \eta \nonumber \\
    \text{s.t.} \hspace{.050cm } & \eta \ge \sum_{t = 1}^{\tau} [\sum_{i \in N \setminus \{\alpha,\omega\}} u_i \theta^k_{it} + \sum_{(i,j) \in A \cup A^{R,out} \cup A^{R,in}} u_{ij} \theta^k_{ijt} + \sum_{(i,j) \in B^R} \tilde{u_j} z^{out,k}_{\alpha jt} w^{out,k}_{\alpha jt} \theta^k_{\alpha jt} ] \\
    & \omit\hfill$\text{ for } k = 1, \ldots,M$\nonumber\\  
    & \pi_{jt}^{+^k} + \theta^k_{\alpha jt} \ge 1 \\
    & \omit\hfill$\text{ for } (\alpha,j) \in A, k = 1, \ldots, M, t = 1, \ldots, \tau$ \nonumber\\
    & \pi_{jt}^{+^k} - \pi_{it}^{-^k} + \theta^k_{ijt} \ge 0 \\
    & \omit\hfill$\text{ for } (i,j) \in A \text{ s.t. } i \ne \alpha, j \ne \omega, k = 1, \ldots, M, t = 1, \ldots, \tau$ \nonumber\\
    & \pi_{it}^{-^k} - \pi_{it}^{+^k} + \theta^k_{it} \ge - \gamma_{it} \\
    & \omit\hfill$\text{ for } i \in N\setminus \{\alpha,\omega\},k = 1, \ldots, M,  t = 1, \ldots, \tau$ \nonumber \\
    & -\pi_{it}^{-^k} + \theta^k_{i\omega t} \ge 0 \\
    & \omit\hfill$\text{ for } (i,\omega) \in A, k = 1, \ldots, M, t = 1, \ldots, \tau$ \nonumber\\
    & \pi_{jt}^{+^k} + \theta^k_{\alpha jt} \ge w^{out,k}_{\alpha jt} + z^{out,k}_{\alpha j} - 1 \\
    & \omit\hfill$\text{ for } (\alpha,j) \in A^{R, out}, k = 1, \ldots, M, t = 1, \ldots, \tau$ \nonumber\\
    & \pi_{jt}^{+^k} - \pi_{it}^{-^k} + \theta^k_{ijt} \ge w^{out,k}_{ijt} + z^{out,k}_{ij} - 2 \\
    & \omit\hfill$\text{ for } (i,j) \in  A^{R, out} \text{ s.t. } i \ne \alpha, j \ne \omega, k = 1, \ldots, M, t = 1, \ldots, \tau$ \nonumber\\
    & \pi_{jt}^{+^k} - \pi_{it}^{-^k} + \theta^k_{ijt} \ge w^{in,k}_{ijt} + z^{in,k}_{ij} - 2 \\
    & \omit\hfill$\text{ for } (i,j) \in  A^{R, in} \text{ s.t. } i \ne \alpha, j \ne \omega, k = 1, \ldots, M, t = 1, \ldots, \tau$ \nonumber\\
    & \theta^k \ge 0 \\
    & \omit\hfill$\text{ for } k = 1, \ldots,M$ \nonumber\\
    & \mu^{out}_i \lambda^{out,k}_{it} + \sum_{(i,j) \in A^{R,out}: z^{out,k}_{ij}=1} w^{out,k}_{ij(t+\delta^{z, out}_{ij})} \ge \sum_{(i,h) \in A} \gamma_{ht} \\
    & \omit\hfill$\text{ for } k \in 1, \ldots,M, i \in T, t \in 1, \ldots, \tau-\delta^{max,out}_i $\nonumber\\
    & \mu^{in}_j \lambda^{in,k}_{jt} + \sum_{(i,j) \in A^{R,in}:  z^{in,k}_{ij}=1} w^{in,k}_{ij(t+ \delta^{z,in}_{ij})} \ge \sum_{(h,j) \in A} \gamma_{ht} \\
    & \omit\hfill$\text{ for } k \in 1, \ldots,M, j \in V, t \in 1, \ldots, \tau-\delta^{max,in}_i$ \nonumber\\
    & \lambda^{out,k}_{it} \le \frac{\sum_{(i,j) \in A^{R,out}:  z^{out,k}_{ij}=1} w^{out,k}_{ij(t+\delta^{z,out}_{ij})}}{\sum_{(i,j) \in A^{R,out}:  z^{out,k}_{ij}=1} z^{out,k}_{ij}} \\
    & \omit\hfill$\text{ for } k \in 1, \ldots,M, i \in T, t = 1, \ldots, \tau $ \nonumber\\
    & \lambda^{in,k}_{it} \le \frac{\sum_{(i,j) \in A^{R,in}: z^{in,k}_{ij}=1} w^{in,k}_{ij(t+\delta^{z,in}_{ij})}}{\sum_{(i,j) \in A^{R,in}:  z^{in,k}_{ij}=1} z^{in,k}_{ij}} \\
    & \omit\hfill$\text{ for } k = 1, \ldots,M, i \in V, t = 1, \ldots, \tau$ \nonumber \\
    & w^{out,k}_{ij} \ge w^{out,k}_{ij(t-1)} \\
    & \omit\hfill$\text{ for } k = 1, \ldots,M, (i,j) \in A^{R,out}, t \in 2 \ldots, \tau $ \nonumber\\
    & w^{in,k}_{ijt} \ge w^{in,k}_{ij(t-1)} \\
    & \omit\hfill$\text{ for } k = 1, \ldots,M, (i,j) \in A^{R,in}, t \in 2 \ldots, \tau$ \nonumber\\
    & w^{out,k}_{\alpha j t} \ge y_{i} \\
    & \omit\hfill$\text{ for } k = 1,\ldots,M,  (j,i) \in T^R, t \in \delta^y_{i} + \delta^{z,out}_{\alpha j}, \ldots, \tau$ \nonumber\\
    & w^{out,k}_{\alpha j (t + \delta^{z,out}_{\alpha j})} \ge y_i \\
    & \omit\hfill$\text{ for } k = 1, \ldots,M, (j,i) \in B^R \text{ s.t. } z^{out,k}_{\alpha j}=1, t \in \delta^y_i+1, \ldots, \tau - \delta^{z,out}_{\alpha j} $ \nonumber\\
    & w^{out,k}_{ijt} \ge z^{out,k}_{kij} \\
    & \omit\hfill$\text{ for } k = 1, \ldots,M, i \in B, j \in V \text{ s.t. } (i,j) \in A^{R,out}$ \nonumber\\
    & w^{out,k}_{ijt} \ge z^{out,k}_{ij} \\
    & \omit\hfill$\text{ for } k = 1, \ldots,M, i \in V, j \in V \text{ s.t. } (i,j) \in A^{R,out}, \exists l \in B, (i,l) \in B^R $\nonumber\\
    & y \in \{0,1\}^{|N|\setminus \{\alpha,\omega\}} \\
    & \gamma \in \{0,1\}^{|N|\setminus \{\alpha,\omega\} \times \tau} \\
    & \text{Constraints } \eqref{con:intBudgetInModel} - \eqref{con:intTimeInModel} \nonumber
\end{align}
\end{subequations}
\end{singlespace}

In general, $M$ will be very large, making the problem computationally difficult to solve. We can instead optimize over a subset of points $\{z^1, \zeta^1, \ldots, z^n, \zeta^n\} \subset \bigcup_{y \in Y} Z(y)$, where $n < M$ to identify a lower bound on the true objective value of the bilevel problem. We then iteratively identify new points $(\tilde{z},\tilde{\zeta})$ to include in the set we optimize over, until the true solution is identified. This is the minimization problem outlined in Section \ref{sec:modelderiv}.

\newpage
\section{Computational Results}
\label{app:fullresults}
\begin{table}[h]
\begin{center}
{\scriptsize %
\begin{tabular}{|c|c|c|c|c|c|}
\hline
Budget & Data1 & Data2 & Data3 & Data4 & Data5 \\ \hline
Base & 161 & 189 & 210 & 210 & 154 \\ \hline
8 & 122 & 164 & 185 & 181 & 130 \\ \hline
12 & 116 & 155 & 177 & 172 & 118 \\ \hline
16 & 105 & 146 & 167 & 160 & 114 \\ \hline
20 & 98 & 138 & 155 & 150 & 102 \\ \hline
24 & 86 & 128 & 149 & 143 & 94 \\ \hline
28 & 81 & 122 & 139 & 135 & 89 \\ \hline
32 & 77 & 111 & 129 & 125 & 82 \\ \hline
36 & 70 & 104 & 120 & 114 & 74 \\ \hline
40 & 62 & 96 & 109 & 107 & 69 \\ \hline
\end{tabular}
}%
\caption{MP-MFNIP-R flow with delayed interdiction and 1 attacker}
\label{tbl:flow_d_1int}
\end{center}
\end{table}

\begin{table}[h]
\begin{center}
{\scriptsize %
\begin{tabular}{|c|c|c|c|c|c|}
\hline
Budget & Data1 & Data2 & Data3 & Data4 & Data5 \\ \hline
Base & 161 & 189 & 210 & 210 & 154 \\ \hline
8 & 109 & 146 & 167 & 160 & 116 \\ \hline
12 & 98 & 138 & 155 & 152 & 104 \\ \hline
16 & 86 & 126 & 147 & 140 & 94 \\ \hline
20 & 79 & 116 & 135 & 134 & 82 \\ \hline
24 & 68 & 108* & 125 & 129* & 74 \\ \hline
28 & 61 & 97 & 116 & 114 & 64 \\ \hline
32 & 54 & 88 & 105 & 106 & 52 \\ \hline
36 & 46 & 74 & 94 & 94 & 46 \\ \hline
40 & 38 & 66 & 82 & 87* & 41 \\ \hline
\end{tabular}
}%
\caption{MP-MFNIP-R flow with delayed interdiction and 2 attackers}
\label{tbl:flow_d_2int}
\end{center}
\end{table}

\begin{table}[h]
\begin{center}
{\scriptsize %
\begin{tabular}{|c|c|c|c|c|c|}
\hline
Budget & Data1 & Data2 & Data3 & Data4 & Data5 \\ \hline
Base & 161 & 189 & 210 & 210 & 154 \\ \hline
8 & 119 & 161 & 185 & 183 & 126 \\ \hline
12 & 110 & 153 & 173 & 167 & 112 \\ \hline
16 & 97 & 139 & 159 & 154 & 108 \\ \hline
20 & 89 & 130 & 147 & 140 & 94 \\ \hline
24 & 77 & 119 & 140 & 133 & 84 \\ \hline
28 & 70 & 110 & 129 & 121 & 77 \\ \hline
32 & 65 & 98 & 117 & 110 & 71 \\ \hline
36 & 54 & 88 & 106 & 98 & 62 \\ \hline
40 & 47 & 80 & 94 & 91 & 54 \\ \hline
\end{tabular}
}%
\caption{MP-MFNIP-R flow with upfront interdiction and 1 attacker}
\label{tbl:flow_u_1int}
\end{center}
\end{table}

\begin{table}[h]
\begin{center}
{\scriptsize %
\begin{tabular}{|c|c|c|c|c|c|}
\hline
Budget & Data1 & Data2 & Data3 & Data4 & Data5 \\ \hline
Base & 161 & 189 & 210 & 210 & 154 \\ \hline
8 & 104 & 140 & 161 & 154 & 110 \\ \hline
12 & 91 & 131 & 147 & 145 & 96 \\ \hline
16 & 77 & 117 & 138 & 131 & 84 \\ \hline
20 & 69 & 105 & 124 & 124 & 70 \\ \hline
24 & 56 & 96 & 112 & 112 & 61 \\ \hline
28 & 48 & 83 & 102 & 98 & 49 \\ \hline
32 & 40 & 70 & 89 & 89 & 35 \\ \hline
36 & 31 & 56 & 76 & 75 & 27 \\ \hline
40 & 21 & 47 & 62 & 68 & 20 \\ \hline
\end{tabular}
}%
\caption{MP-MFNIP-R flow with upfront interdiction and 2 attackers}
\label{tbl:flow_u_2int}
\end{center}
\end{table}

\begin{table}[h]
\begin{center}
{\scriptsize %
\begin{tabular}{|c|c|c|c|c|c|c|}
\hline
Budget & \begin{tabular}[c]{@{}c@{}}MP-MFNIP\\ Int Trafficker\end{tabular} & \begin{tabular}[c]{@{}c@{}}MP-MFNIP\\ Int Bottom\end{tabular} & \begin{tabular}[c]{@{}c@{}}MP-MFNIP\\ Int Victim\end{tabular} & \begin{tabular}[c]{@{}c@{}}MP-MFNIP-R\\ Int Trafficker\end{tabular} & \begin{tabular}[c]{@{}c@{}}MP-MFNIP-R\\ Int Bottom\end{tabular} & \begin{tabular}[c]{@{}c@{}}MP-MFNIP-R\\ Int Victim\end{tabular} \\ \hline
8 & 0 & 1 & 2 & 0 & 1 & 2 \\ \hline
12 & 0 & 1 & 4 & 0 & 2 & 2 \\ \hline
16 & 0 & 2 & 4 & 1 & 0 & 6 \\ \hline
20 & 0 & 1 & 8 & 1 & 0 & 8 \\ \hline
24 & 0 & 1 & 10 & 1 & 0 & 10 \\ \hline
28 & 0 & 2 & 10 & 1 & 1 & 10 \\ \hline
32 & 0 & 1 & 14 & 0 & 0 & 16 \\ \hline
36 & 0 & 2 & 14 & 0 & 0 & 18 \\ \hline
40 & 0 & 2 & 16 & 0 & 0 & 20 \\ \hline
\end{tabular}
}%
\caption{Recommended delayed interdictions on network 1 with one attacker}
\label{tbl:net1_d_1int_who}
\end{center}
\end{table}

\begin{table}[h]
\begin{center}
{\scriptsize %
\begin{tabular}{|c|c|c|c|c|c|c|}
\hline
Budget & \begin{tabular}[c]{@{}c@{}}MP-MFNIP\\ Int Trafficker\end{tabular} & \begin{tabular}[c]{@{}c@{}}MP-MFNIP\\ Int Bottom\end{tabular} & \begin{tabular}[c]{@{}c@{}}MP-MFNIP\\ Int Victim\end{tabular} & \begin{tabular}[c]{@{}c@{}}MP-MFNIP-R\\ Int Trafficker\end{tabular} & \begin{tabular}[c]{@{}c@{}}MP-MFNIP-R\\ Int Bottom\end{tabular} & \begin{tabular}[c]{@{}c@{}}MP-MFNIP-R\\ Int Victim\end{tabular} \\ \hline
8 & 0 & 2 & 0 & 0 & 1 & 2 \\ \hline
12 & 0 & 2 & 2 & 0 & 2 & 2 \\ \hline
16 & 0 & 3 & 2 & 1 & 0 & 6 \\ \hline
20 & 0 & 2 & 6 & 1 & 0 & 8 \\ \hline
24 & 0 & 2 & 8 & 1 & 0 & 10 \\ \hline
28 & 0 & 2 & 10 & 2 & 0 & 10 \\ \hline
32 & 0 & 2 & 12 & 2 & 0 & 12 \\ \hline
36 & 0 & 3 & 12 & 0 & 0 & 18 \\ \hline
40 & 1 & 2 & 14 & 1 & 0 & 18 \\ \hline
\end{tabular}
}%
\caption{Recommended delayed interdictions on network 2 with one attacker}
\label{tbl:net2_d_1int_who}
\end{center}
\end{table}

\begin{table}[h]
\begin{center}
{\scriptsize %
\begin{tabular}{|c|c|c|c|c|c|c|}
\hline
Budget & \begin{tabular}[c]{@{}c@{}}MP-MFNIP\\ Int Trafficker\end{tabular} & \begin{tabular}[c]{@{}c@{}}MP-MFNIP\\ Int Bottom\end{tabular} & \begin{tabular}[c]{@{}c@{}}MP-MFNIP\\ Int Victim\end{tabular} & \begin{tabular}[c]{@{}c@{}}MP-MFNIP-R\\ Int Trafficker\end{tabular} & \begin{tabular}[c]{@{}c@{}}MP-MFNIP-R\\ Int Bottom\end{tabular} & \begin{tabular}[c]{@{}c@{}}MP-MFNIP-R\\ Int Victim\end{tabular} \\ \hline
8 & 0 & 2 & 0 & 0 & 2 & 0 \\ \hline
12 & 0 & 3 & 0 & 0 & 0 & 6 \\ \hline
16 & 0 & 2 & 4 & 1 & 0 & 6 \\ \hline
20 & 0 & 2 & 6 & 1 & 0 & 8 \\ \hline
24 & 1 & 2 & 6 & 0 & 0 & 12 \\ \hline
28 & 2 & 2 & 5 & 1 & 0 & 12 \\ \hline
32 & 2 & 2 & 7 & 1 & 0 & 14 \\ \hline
36 & 2 & 2 & 9 & 1 & 0 & 16 \\ \hline
40 & 1 & 2 & 14 & 1 & 0 & 18 \\ \hline
\end{tabular}
}%
\caption{Recommended delayed interdictions on network 3 with one attacker}
\label{tbl:net3_d_1int_who}
\end{center}
\end{table}

\begin{table}[h]
\begin{center}
{\scriptsize %
\begin{tabular}{|c|c|c|c|c|c|c|}
\hline
Budget & \begin{tabular}[c]{@{}c@{}}MP-MFNIP\\ Int Trafficker\end{tabular} & \begin{tabular}[c]{@{}c@{}}MP-MFNIP\\ Int Bottom\end{tabular} & \begin{tabular}[c]{@{}c@{}}MP-MFNIP\\ Int Victim\end{tabular} & \begin{tabular}[c]{@{}c@{}}MP-MFNIP-R\\ Int Trafficker\end{tabular} & \begin{tabular}[c]{@{}c@{}}MP-MFNIP-R\\ Int Bottom\end{tabular} & \begin{tabular}[c]{@{}c@{}}MP-MFNIP-R\\ Int Victim\end{tabular} \\ \hline
8 & 0 & 2 & 0 & 0 & 2 & 0 \\ \hline
12 & 1 & 1 & 2 & 0 & 0 & 6 \\ \hline
16 & 1 & 1 & 4 & 1 & 0 & 6 \\ \hline
20 & 2 & 2 & 1 & 1 & 0 & 8 \\ \hline
24 & 2 & 3 & 1 & 1 & 0 & 10 \\ \hline
28 & 1 & 1 & 10 & 2 & 1 & 8 \\ \hline
32 & 2 & 3 & 5 & 2 & 1 & 10 \\ \hline
36 & 2 & 2 & 9 & 2 & 0 & 14 \\ \hline
40 & 2 & 3 & 9 & 2 & 0 & 16 \\ \hline
\end{tabular}
}%
\caption{Recommended delayed interdictions on network 4 with one attacker}
\label{tbl:net4_d_1int_who}
\end{center}
\end{table}

\begin{table}[h]
\begin{center}
{\scriptsize %
\begin{tabular}{|c|c|c|c|c|c|c|}
\hline
Budget & \begin{tabular}[c]{@{}c@{}}MP-MFNIP\\ Int Trafficker\end{tabular} & \begin{tabular}[c]{@{}c@{}}MP-MFNIP\\ Int Bottom\end{tabular} & \begin{tabular}[c]{@{}c@{}}MP-MFNIP\\ Int Victim\end{tabular} & \begin{tabular}[c]{@{}c@{}}MP-MFNIP-R\\ Int Trafficker\end{tabular} & \begin{tabular}[c]{@{}c@{}}MP-MFNIP-R\\ Int Bottom\end{tabular} & \begin{tabular}[c]{@{}c@{}}MP-MFNIP-R\\ Int Victim\end{tabular} \\ \hline
8 & 0 & 1 & 2 & 0 & 0 & 4 \\ \hline
12 & 0 & 1 & 4 & 0 & 0 & 6 \\ \hline
16 & 0 & 1 & 6 & 0 & 0 & 8 \\ \hline
20 & 0 & 1 & 8 & 0 & 0 & 10 \\ \hline
24 & 0 & 1 & 10 & 1 & 0 & 10 \\ \hline
28 & 0 & 2 & 10 & 1 & 1 & 10 \\ \hline
32 & 0 & 1 & 14 & 0 & 0 & 16 \\ \hline
36 & 0 & 2 & 14 & 0 & 0 & 18 \\ \hline
40 & 0 & 3 & 14 & 1 & 1 & 16 \\ \hline
\end{tabular}
}%
\caption{Recommended delayed interdictions on network 5 with one attacker}
\label{tbl:net5_d_1int_who}
\end{center}
\end{table}

\begin{table}[h]
\begin{center}
{\scriptsize %
\begin{tabular}{|c|c|c|c|c|c|c|}
\hline
Budget & \begin{tabular}[c]{@{}c@{}}MP-MFNIP\\ Int Trafficker\end{tabular} & \begin{tabular}[c]{@{}c@{}}MP-MFNIP\\ Int Bottom\end{tabular} & \begin{tabular}[c]{@{}c@{}}MP-MFNIP\\ Int Victim\end{tabular} & \begin{tabular}[c]{@{}c@{}}MP-MFNIP-R\\ Int Trafficker\end{tabular} & \begin{tabular}[c]{@{}c@{}}MP-MFNIP-R\\ Int Bottom\end{tabular} & \begin{tabular}[c]{@{}c@{}}MP-MFNIP-R\\ Int Victim\end{tabular} \\ \hline
8 & 0 & 2 & 0 & 0 & 0 & 4 \\ \hline
12 & 1 & 1 & 1 & 1 & 0 & 4 \\ \hline
16 & 1 & 2 & 1 & 1 & 0 & 6 \\ \hline
20 & 1 & 3 & 1 & 1 & 0 & 8 \\ \hline
24 & 0 & 2 & 8 & 1 & 0 & 10 \\ \hline
28 & 0 & 2 & 10 & 1 & 1 & 10 \\ \hline
32 & 0 & 2 & 12 & 2 & 2 & 8 \\ \hline
36 & 1 & 2 & 10 & 2 & 2 & 10 \\ \hline
40 & 1 & 2 & 12 & 2 & 3 & 10 \\ \hline
\end{tabular}
}%
\caption{Recommended upfront interdictions on network 1 with one attacker}
\label{tbl:net1_u_1int_who}
\end{center}
\end{table}

\begin{table}[h]
\begin{center}
{\scriptsize %
\begin{tabular}{|c|c|c|c|c|c|c|}
\hline
Budget & \begin{tabular}[c]{@{}c@{}}MP-MFNIP\\ Int Trafficker\end{tabular} & \begin{tabular}[c]{@{}c@{}}MP-MFNIP\\ Int Bottom\end{tabular} & \begin{tabular}[c]{@{}c@{}}MP-MFNIP\\ Int Victim\end{tabular} & \begin{tabular}[c]{@{}c@{}}MP-MFNIP-R\\ Int Trafficker\end{tabular} & \begin{tabular}[c]{@{}c@{}}MP-MFNIP-R\\ Int Bottom\end{tabular} & \begin{tabular}[c]{@{}c@{}}MP-MFNIP-R\\ Int Victim\end{tabular} \\ \hline
8 & 0 & 2 & 0 & 0 & 0 & 4 \\ \hline
12 & 0 & 3 & 0 & 1 & 0 & 4 \\ \hline
16 & 0 & 4 & 0 & 1 & 0 & 6 \\ \hline
20 & 0 & 4 & 2 & 1 & 1 & 6 \\ \hline
24 & 1 & 4 & 1 & 1 & 0 & 10 \\ \hline
28 & 2 & 4 & 1 & 2 & 0 & 10 \\ \hline
32 & 2 & 4 & 3 & 2 & 0 & 12 \\ \hline
36 & 0 & 3 & 12 & 2 & 1 & 12 \\ \hline
40 & 0 & 4 & 12 & 2 & 1 & 14 \\ \hline
\end{tabular}
}%
\caption{Recommended upfront interdictions on network 2 with one attacker}
\label{tbl:net2_u_1int_who}
\end{center}
\end{table}

\begin{table}[h]
\begin{center}
{\scriptsize %
\begin{tabular}{|c|c|c|c|c|c|c|}
\hline
Budget & \begin{tabular}[c]{@{}c@{}}MP-MFNIP\\ Int Trafficker\end{tabular} & \begin{tabular}[c]{@{}c@{}}MP-MFNIP\\ Int Bottom\end{tabular} & \begin{tabular}[c]{@{}c@{}}MP-MFNIP\\ Int Victim\end{tabular} & \begin{tabular}[c]{@{}c@{}}MP-MFNIP-R\\ Int Trafficker\end{tabular} & \begin{tabular}[c]{@{}c@{}}MP-MFNIP-R\\ Int Bottom\end{tabular} & \begin{tabular}[c]{@{}c@{}}MP-MFNIP-R\\ Int Victim\end{tabular} \\ \hline
8 & 0 & 2 & 0 & 0 & 0 & 4 \\ \hline
12 & 0 & 3 & 0 & 1 & 1 & 2 \\ \hline
16 & 0 & 2 & 4 & 1 & 0 & 6 \\ \hline
20 & 1 & 2 & 3 & 1 & 0 & 8 \\ \hline
24 & 1 & 2 & 5 & 1 & 0 & 10 \\ \hline
28 & 2 & 2 & 5 & 1 & 0 & 12 \\ \hline
32 & 2 & 3 & 5 & 1 & 0 & 14 \\ \hline
36 & 2 & 2 & 9 & 1 & 0 & 16 \\ \hline
40 & 2 & 3 & 9 & 1 & 0 & 18 \\ \hline
\end{tabular}
}%
\caption{Recommended upfront interdictions on network 3 with one attacker}
\label{tbl:net3_u_1int_who}
\end{center}
\end{table}

\begin{table}[h]
\begin{center}
{\scriptsize %
\begin{tabular}{|c|c|c|c|c|c|c|}
\hline
Budget & \begin{tabular}[c]{@{}c@{}}MP-MFNIP\\ Int Trafficker\end{tabular} & \begin{tabular}[c]{@{}c@{}}MP-MFNIP\\ Int Bottom\end{tabular} & \begin{tabular}[c]{@{}c@{}}MP-MFNIP\\ Int Victim\end{tabular} & \begin{tabular}[c]{@{}c@{}}MP-MFNIP-R\\ Int Trafficker\end{tabular} & \begin{tabular}[c]{@{}c@{}}MP-MFNIP-R\\ Int Bottom\end{tabular} & \begin{tabular}[c]{@{}c@{}}MP-MFNIP-R\\ Int Victim\end{tabular} \\ \hline
8 & 0 & 2 & 0 & 0 & 0 & 4 \\ \hline
12 & 1 & 1 & 2 & 1 & 0 & 4 \\ \hline
16 & 1 & 2 & 1 & 1 & 0 & 6 \\ \hline
20 & 2 & 2 & 1 & 1 & 0 & 8 \\ \hline
24 & 2 & 3 & 1 & 1 & 1 & 8 \\ \hline
28 & 2 & 3 & 3 & 2 & 1 & 8 \\ \hline
32 & 2 & 3 & 5 & 2 & 1 & 10 \\ \hline
36 & 2 & 3 & 7 & 2 & 0 & 14 \\ \hline
40 & 2 & 3 & 9 & 2 & 1 & 14 \\ \hline
\end{tabular}
}%
\caption{Recommended upfront interdictions on network 4 with one attacker}
\label{tbl:net4_u_1int_who}
\end{center}
\end{table}

\begin{table}[h]
\begin{center}
{\scriptsize %
\begin{tabular}{|c|c|c|c|c|c|c|}
\hline
Budget & \begin{tabular}[c]{@{}c@{}}MP-MFNIP\\ Int Trafficker\end{tabular} & \begin{tabular}[c]{@{}c@{}}MP-MFNIP\\ Int Bottom\end{tabular} & \begin{tabular}[c]{@{}c@{}}MP-MFNIP\\ Int Victim\end{tabular} & \begin{tabular}[c]{@{}c@{}}MP-MFNIP-R\\ Int Trafficker\end{tabular} & \begin{tabular}[c]{@{}c@{}}MP-MFNIP-R\\ Int Bottom\end{tabular} & \begin{tabular}[c]{@{}c@{}}MP-MFNIP-R\\ Int Victim\end{tabular} \\ \hline
8 & 0 & 1 & 2 & 0 & 0 & 4 \\ \hline
12 & 0 & 1 & 4 & 0 & 0 & 6 \\ \hline
16 & 0 & 1 & 6 & 0 & 0 & 8 \\ \hline
20 & 0 & 1 & 8 & 0 & 0 & 10 \\ \hline
24 & 0 & 2 & 8 & 1 & 0 & 10 \\ \hline
28 & 0 & 3 & 8 & 1 & 1 & 10 \\ \hline
32 & 0 & 4 & 8 & 0 & 2 & 12 \\ \hline
36 & 0 & 3 & 12 & 2 & 1 & 12 \\ \hline
40 & 0 & 3 & 14 & 2 & 1 & 14 \\ \hline
\end{tabular}
}%
\caption{Recommended upfront interdictions on network 5 with one attacker}
\label{tbl:net5_u_1int_who}
\end{center}
\end{table}

\begin{sidewaystable}[]
\begin{center}
\resizebox{\textwidth}{!}{
\begin{tabular}{|c|c|c|c|c|c|c|c|c|c|c|}
\hline
Budget & \begin{tabular}[c]{@{}c@{}}MP-MFNIP\\ Int Trafficker\end{tabular} & \begin{tabular}[c]{@{}c@{}}MP-MFNIP\\ Int Bottom\end{tabular} & \begin{tabular}[c]{@{}c@{}}MP-MFNIP\\ Int Victim\end{tabular} & \begin{tabular}[c]{@{}c@{}}MP-MFNIP 2nd\\ Int Current Vic\end{tabular} & \begin{tabular}[c]{@{}c@{}}MP-MFNIP 2nd\\ Int Prosp Vic\end{tabular} & \begin{tabular}[c]{@{}c@{}}MP-MFNIP-R\\ Int Trafficker\end{tabular} & \begin{tabular}[c]{@{}c@{}}MP-MFNIP-R\\ Int Bottom\end{tabular} & \begin{tabular}[c]{@{}c@{}}MP-MFNIP-R\\ Int Victim\end{tabular} & \begin{tabular}[c]{@{}c@{}}MP-MFNIP-R 2nd\\ Int Current Victim\end{tabular} & \begin{tabular}[c]{@{}c@{}}MP-MFNIP-R 2nd\\ Int Prosp Victim\end{tabular} \\ \hline
8 & 0 & 1 & 2 & 3 & 1 & 0 & 0 & 4 & 2 & 3 \\ \hline
12 & 0 & 1 & 4 & 2 & 3 & 0 & 0 & 6 & 2 & 3 \\ \hline
16 & 0 & 1 & 6 & 3 & 1 & 0 & 0 & 8 & 2 & 3 \\ \hline
20 & 0 & 1 & 8 & 3 & 1 & 0 & 0 & 10 & 2 & 3 \\ \hline
24 & 0 & 2 & 12 & 2 & 3 & 0 & 0 & 12 & 2 & 4 \\ \hline
28 & 0 & 1 & 14 & 2 & 3 & 0 & 0 & 14 & 2 & 4 \\ \hline
32 & 0 & 1 & 14 & 2 & 4 & 0 & 0 & 16 & 2 & 4 \\ \hline
36 & 0 & 2 & 14 & 2 & 4 & 0 & 0 & 18 & 2 & 4 \\ \hline
40 & 0 & 3 & 14 & 2 & 4 & 0 & 0 & 20 & 0 & 8 \\ \hline
\end{tabular}}
\caption{Recommended delayed interdictions on network 1 with two attackers}
\label{tbl:net1_d_2int_who}

\vspace{2\baselineskip}
\resizebox{\textwidth}{!}{
\begin{tabular}{|c|c|c|c|c|c|c|c|c|c|c|}
\hline
Budget & \begin{tabular}[c]{@{}c@{}}MP-MFNIP\\ Int Trafficker\end{tabular} & \begin{tabular}[c]{@{}c@{}}MP-MFNIP\\ Int Bottom\end{tabular} & \begin{tabular}[c]{@{}c@{}}MP-MFNIP\\ Int Victim\end{tabular} & \begin{tabular}[c]{@{}c@{}}MP-MFNIP 2nd\\ Int Current Vic\end{tabular} & \begin{tabular}[c]{@{}c@{}}MP-MFNIP 2nd\\ Int Prosp Vic\end{tabular} & \begin{tabular}[c]{@{}c@{}}MP-MFNIP-R\\ Int Trafficker\end{tabular} & \begin{tabular}[c]{@{}c@{}}MP-MFNIP-R\\ Int Bottom\end{tabular} & \begin{tabular}[c]{@{}c@{}}MP-MFNIP-R\\ Int Victim\end{tabular} & \begin{tabular}[c]{@{}c@{}}MP-MFNIP-R 2nd\\ Int Current Victim\end{tabular} & \begin{tabular}[c]{@{}c@{}}MP-MFNIP-R 2nd\\ Int Prosp Victim\end{tabular} \\ \hline
8 & 0 & 2 & 0 & 3 & 1 & 0 & 0 & 4 & 2 & 4 \\ \hline
12 & 0 & 2 & 2 & 3 & 1 & 0 & 0 & 6 & 2 & 4 \\ \hline
16 & 0 & 2 & 4 & 3 & 1 & 0 & 0 & 8 & 2 & 4 \\ \hline
20 & 0 & 2 & 6 & 2 & 3 & 0 & 0 & 10 & 1 & 6 \\ \hline
24 & 0 & 2 & 8 & 2 & 4 & 0 & 0 & 12 & 2 & 4 \\ \hline
28 & 0 & 2 & 10 & 2 & 4 & 0 & 0 & 14 & 1 & 7 \\ \hline
32 & 0 & 3 & 10 & 2 & 4 & 0 & 0 & 16 & 1 & 6 \\ \hline
36 & 1 & 2 & 12 & 2 & 4 & 0 & 0 & 18 & 1 & 6 \\ \hline
40 & 0 & 1 & 18 & 2 & 4 & 0 & 0 & 20 & 1 & 7 \\ \hline
\end{tabular}
}
\caption{Recommended delayed interdictions on network 2 with two attackers}
\label{tbl:net2_d_2int_who}
\end{center}
\end{sidewaystable}

\begin{sidewaystable}[]
\begin{center}
\resizebox{\textwidth}{!}{
\begin{tabular}{|c|c|c|c|c|c|c|c|c|c|c|}
\hline
Budget & \begin{tabular}[c]{@{}c@{}}MP-MFNIP\\ Int Trafficker\end{tabular} & \begin{tabular}[c]{@{}c@{}}MP-MFNIP\\ Int Bottom\end{tabular} & \begin{tabular}[c]{@{}c@{}}MP-MFNIP\\ Int Victim\end{tabular} & \begin{tabular}[c]{@{}c@{}}MP-MFNIP 2nd\\ Int Current Vic\end{tabular} & \begin{tabular}[c]{@{}c@{}}MP-MFNIP 2nd\\ Int Prosp Vic\end{tabular} & \begin{tabular}[c]{@{}c@{}}MP-MFNIP-R\\ Int Trafficker\end{tabular} & \begin{tabular}[c]{@{}c@{}}MP-MFNIP-R\\ Int Bottom\end{tabular} & \begin{tabular}[c]{@{}c@{}}MP-MFNIP-R\\ Int Victim\end{tabular} & \begin{tabular}[c]{@{}c@{}}MP-MFNIP-R 2nd\\ Int Current Victim\end{tabular} & \begin{tabular}[c]{@{}c@{}}MP-MFNIP-R 2nd\\ Int Prosp Victim\end{tabular} \\ \hline
8 & 0 & 2 & 0 & 3 & 1 & 0 & 0 & 4 & 2 & 3 \\ \hline
12 & 0 & 2 & 2 & 3 & 1 & 0 & 0 & 6 & 2 & 3 \\ \hline
16 & 0 & 2 & 4 & 3 & 1 & 0 & 0 & 8 & 2 & 4 \\ \hline
20 & 0 & 2 & 6 & 3 & 1 & 0 & 0 & 10 & 2 & 4 \\ \hline
24 & 1 & 2 & 6 & 3 & 1 & 0 & 0 & 12 & 2 & 4 \\ \hline
28 & 0 & 1 & 12 & 2 & 4 & 0 & 0 & 14 & 2 & 4 \\ \hline
32 & 1 & 3 & 8 & 2 & 4 & 0 & 0 & 16 & 2 & 4 \\ \hline
36 & 2 & 3 & 7 & 2 & 4 & 0 & 0 & 18 & 1 & 7 \\ \hline
40 & 1 & 1 & 14 & 2 & 4 & 0 & 0 & 20 & 1 & 7 \\ \hline
\end{tabular}
}
\caption{Recommended delayed interdictions on network 3 with two attackers}
\label{tbl:net3_d_2int_who}

\vspace{2\baselineskip}
\resizebox{\textwidth}{!}{
\begin{tabular}{|c|c|c|c|c|c|c|c|c|c|c|}
\hline
Budget & \begin{tabular}[c]{@{}c@{}}MP-MFNIP\\ Int Trafficker\end{tabular} & \begin{tabular}[c]{@{}c@{}}MP-MFNIP\\ Int Bottom\end{tabular} & \begin{tabular}[c]{@{}c@{}}MP-MFNIP\\ Int Victim\end{tabular} & \begin{tabular}[c]{@{}c@{}}MP-MFNIP 2nd\\ Int Current Vic\end{tabular} & \begin{tabular}[c]{@{}c@{}}MP-MFNIP 2nd\\ Int Prosp Vic\end{tabular} & \begin{tabular}[c]{@{}c@{}}MP-MFNIP-R\\ Int Trafficker\end{tabular} & \begin{tabular}[c]{@{}c@{}}MP-MFNIP-R\\ Int Bottom\end{tabular} & \begin{tabular}[c]{@{}c@{}}MP-MFNIP-R\\ Int Victim\end{tabular} & \begin{tabular}[c]{@{}c@{}}MP-MFNIP-R 2nd\\ Int Current Victim\end{tabular} & \begin{tabular}[c]{@{}c@{}}MP-MFNIP-R 2nd\\ Int Prosp Victim\end{tabular} \\ \hline
8 & 0 & 1 & 2 & 3 & 1 & 0 & 0 & 4 & 2 & 3 \\ \hline
12 & 1 & 1 & 2 & 3 & 1 & 0 & 0 & 6 & 2 & 4 \\ \hline
16 & 0 & 1 & 6 & 2 & 3 & 0 & 0 & 8 & 2 & 4 \\ \hline
20 & 2 & 2 & 1 & 3 & 1 & 0 & 0 & 10 & 2 & 4 \\ \hline
24 & 2 & 2 & 3 & 3 & 0 & 0 & 0 & 12 & 1 & 7 \\ \hline
28 & 1 & 2 & 8 & 2 & 4 & 1 & 0 & 12 & 2 & 3 \\ \hline
32 & 1 & 2 & 10 & 2 & 4 & 1 & 0 & 14 & 2 & 4 \\ \hline
36 & 0 & 0 & 18 & 1 & 7 & 1 & 0 & 16 & 2 & 4 \\ \hline
40 & 0 & 1 & 18 & 1 & 7 & 1 & 0 & 18 & 2 & 4 \\ \hline
\end{tabular}
}
\caption{Recommended delayed interdictions on network 4 with two attackers}
\label{tbl:net4_d_2int_who}
\end{center}
\end{sidewaystable}

\begin{sidewaystable}[]
\begin{center}
\resizebox{\textwidth}{!}{
\begin{tabular}{|c|c|c|c|c|c|c|c|c|c|c|}
\hline
Budget & \begin{tabular}[c]{@{}c@{}}MP-MFNIP\\ Int Trafficker\end{tabular} & \begin{tabular}[c]{@{}c@{}}MP-MFNIP\\ Int Bottom\end{tabular} & \begin{tabular}[c]{@{}c@{}}MP-MFNIP\\ Int Victim\end{tabular} & \begin{tabular}[c]{@{}c@{}}MP-MFNIP 2nd\\ Int Current Vic\end{tabular} & \begin{tabular}[c]{@{}c@{}}MP-MFNIP 2nd\\ Int Prosp Vic\end{tabular} & \begin{tabular}[c]{@{}c@{}}MP-MFNIP-R\\ Int Trafficker\end{tabular} & \begin{tabular}[c]{@{}c@{}}MP-MFNIP-R\\ Int Bottom\end{tabular} & \begin{tabular}[c]{@{}c@{}}MP-MFNIP-R\\ Int Victim\end{tabular} & \begin{tabular}[c]{@{}c@{}}MP-MFNIP-R 2nd\\ Int Current Victim\end{tabular} & \begin{tabular}[c]{@{}c@{}}MP-MFNIP-R 2nd\\ Int Prosp Victim\end{tabular} \\ \hline
8 & 0 & 1 & 2 & 3 & 1 & 0 & 0 & 4 & 3 & 1 \\ \hline
12 & 0 & 1 & 4 & 3 & 1 & 0 & 0 & 6 & 3 & 1 \\ \hline
16 & 0 & 1 & 6 & 3 & 1 & 0 & 0 & 8 & 2 & 3 \\ \hline
20 & 0 & 1 & 8 & 3 & 1 & 0 & 0 & 10 & 2 & 3 \\ \hline
24 & 0 & 1 & 10 & 2 & 4 & 0 & 0 & 12 & 2 & 4 \\ \hline
28 & 0 & 1 & 12 & 2 & 4 & 0 & 0 & 14 & 1 & 7 \\ \hline
32 & 0 & 1 & 14 & 1 & 7 & 0 & 0 & 16 & 1 & 7 \\ \hline
36 & 0 & 2 & 14 & 1 & 7 & 0 & 0 & 17 & 1 & 7 \\ \hline
40 & 0 & 3 & 14 & 1 & 7 & 0 & 1 & 17 & 1 & 7 \\ \hline
\end{tabular}
}
\caption{Recommended delayed interdictions on network 5 with two attackers}
\label{tbl:net5_d_2int_who}

\vspace{2\baselineskip}
\resizebox{\textwidth}{!}{
\begin{tabular}{|c|c|c|c|c|c|c|c|c|c|c|}
\hline
Budget & \begin{tabular}[c]{@{}c@{}}MP-MFNIP\\ Int Trafficker\end{tabular} & \begin{tabular}[c]{@{}c@{}}MP-MFNIP\\ Int Bottom\end{tabular} & \begin{tabular}[c]{@{}c@{}}MP-MFNIP\\ Int Victim\end{tabular} & \begin{tabular}[c]{@{}c@{}}MP-MFNIP 2nd\\ Int Current Vic\end{tabular} & \begin{tabular}[c]{@{}c@{}}MP-MFNIP 2nd\\ Int Prosp Vic\end{tabular} & \begin{tabular}[c]{@{}c@{}}MP-MFNIP-R\\ Int Trafficker\end{tabular} & \begin{tabular}[c]{@{}c@{}}MP-MFNIP-R\\ Int Bottom\end{tabular} & \begin{tabular}[c]{@{}c@{}}MP-MFNIP-R\\ Int Victim\end{tabular} & \begin{tabular}[c]{@{}c@{}}MP-MFNIP-R 2nd\\ Int Current Victim\end{tabular} & \begin{tabular}[c]{@{}c@{}}MP-MFNIP-R 2nd\\ Int Prosp Victim\end{tabular} \\ \hline
8 & 0 & 2 & 0 & 3 & 1 & 0 & 0 & 4 & 2 & 3 \\ \hline
12 & 0 & 2 & 2 & 2 & 2 & 0 & 0 & 6 & 2 & 3 \\ \hline
16 & 0 & 2 & 4 & 2 & 3 & 0 & 0 & 8 & 2 & 3 \\ \hline
20 & 0 & 2 & 6 & 2 & 3 & 0 & 0 & 10 & 2 & 3 \\ \hline
24 & 0 & 2 & 8 & 2 & 3 & 0 & 0 & 12 & 2 & 4 \\ \hline
28 & 0 & 2 & 10 & 2 & 3 & 0 & 0 & 14 & 2 & 4 \\ \hline
32 & 0 & 3 & 10 & 2 & 3 & 0 & 0 & 16 & 2 & 4 \\ \hline
36 & 0 & 2 & 14 & 2 & 4 & 0 & 0 & 18 & 2 & 4 \\ \hline
40 & 0 & 3 & 14 & 2 & 4 & 0 & 0 & 20 & 0 & 8 \\ \hline
\end{tabular}
}
\caption{Recommended upfront interdictions on network 1 with two attackers}
\label{tbl:net1_u_2int_who}
\end{center}
\end{sidewaystable}

\begin{sidewaystable}[]
\begin{center}
\resizebox{\textwidth}{!}{
\begin{tabular}{|c|c|c|c|c|c|c|c|c|c|c|}
\hline
Budget & \begin{tabular}[c]{@{}c@{}}MP-MFNIP\\ Int Trafficker\end{tabular} & \begin{tabular}[c]{@{}c@{}}MP-MFNIP\\ Int Bottom\end{tabular} & \begin{tabular}[c]{@{}c@{}}MP-MFNIP\\ Int Victim\end{tabular} & \begin{tabular}[c]{@{}c@{}}MP-MFNIP 2nd\\ Int Current Vic\end{tabular} & \begin{tabular}[c]{@{}c@{}}MP-MFNIP 2nd\\ Int Prosp Vic\end{tabular} & \begin{tabular}[c]{@{}c@{}}MP-MFNIP-R\\ Int Trafficker\end{tabular} & \begin{tabular}[c]{@{}c@{}}MP-MFNIP-R\\ Int Bottom\end{tabular} & \begin{tabular}[c]{@{}c@{}}MP-MFNIP-R\\ Int Victim\end{tabular} & \begin{tabular}[c]{@{}c@{}}MP-MFNIP-R 2nd\\ Int Current Victim\end{tabular} & \begin{tabular}[c]{@{}c@{}}MP-MFNIP-R 2nd\\ Int Prosp Victim\end{tabular} \\ \hline
8 & 0 & 2 & 0 & 3 & 1 & 0 & 0 & 4 & 2 & 4 \\ \hline
12 & 0 & 2 & 2 & 3 & 1 & 0 & 0 & 6 & 2 & 4 \\ \hline
16 & 0 & 4 & 0 & 3 & 1 & 0 & 0 & 8 & 2 & 4 \\ \hline
20 & 0 & 3 & 4 & 2 & 3 & 0 & 0 & 10 & 1 & 5 \\ \hline
24 & 0 & 3 & 6 & 2 & 4 & 0 & 0 & 12 & 1 & 7 \\ \hline
28 & 0 & 3 & 8 & 2 & 4 & 0 & 0 & 14 & 1 & 6 \\ \hline
32 & 0 & 3 & 10 & 2 & 4 & 0 & 0 & 16 & 1 & 7 \\ \hline
36 & 0 & 3 & 12 & 1 & 5 & 0 & 0 & 18 & 1 & 6 \\ \hline
40 & 1 & 3 & 12 & 2 & 4 & 0 & 0 & 20 & 1 & 7 \\ \hline
\end{tabular}
}
\caption{Recommended upfront interdictions on network 2 with two attackers}
\label{tbl:net2_u_2int_who}

\vspace{2\baselineskip}
\resizebox{\textwidth}{!}{
\begin{tabular}{|c|c|c|c|c|c|c|c|c|c|c|}
\hline
Budget & \begin{tabular}[c]{@{}c@{}}MP-MFNIP\\ Int Trafficker\end{tabular} & \begin{tabular}[c]{@{}c@{}}MP-MFNIP\\ Int Bottom\end{tabular} & \begin{tabular}[c]{@{}c@{}}MP-MFNIP\\ Int Victim\end{tabular} & \begin{tabular}[c]{@{}c@{}}MP-MFNIP 2nd\\ Int Current Vic\end{tabular} & \begin{tabular}[c]{@{}c@{}}MP-MFNIP 2nd\\ Int Prosp Vic\end{tabular} & \begin{tabular}[c]{@{}c@{}}MP-MFNIP-R\\ Int Trafficker\end{tabular} & \begin{tabular}[c]{@{}c@{}}MP-MFNIP-R\\ Int Bottom\end{tabular} & \begin{tabular}[c]{@{}c@{}}MP-MFNIP-R\\ Int Victim\end{tabular} & \begin{tabular}[c]{@{}c@{}}MP-MFNIP-R 2nd\\ Int Current Victim\end{tabular} & \begin{tabular}[c]{@{}c@{}}MP-MFNIP-R 2nd\\ Int Prosp Victim\end{tabular} \\ \hline
8 & 0 & 2 & 0 & 3 & 1 & 0 & 0 & 4 & 2 & 3 \\ \hline
12 & 0 & 2 & 2 & 3 & 1 & 0 & 0 & 6 & 2 & 3 \\ \hline
16 & 1 & 2 & 1 & 3 & 1 & 0 & 0 & 8 & 2 & 4 \\ \hline
20 & 2 & 2 & 1 & 3 & 1 & 0 & 0 & 10 & 2 & 4 \\ \hline
24 & 1 & 2 & 5 & 3 & 1 & 0 & 0 & 12 & 2 & 4 \\ \hline
28 & 2 & 2 & 5 & 3 & 1 & 0 & 0 & 14 & 2 & 4 \\ \hline
32 & 1 & 3 & 7 & 2 & 4 & 0 & 0 & 16 & 2 & 4 \\ \hline
36 & 2 & 3 & 7 & 2 & 4 & 0 & 0 & 18 & 1 & 7 \\ \hline
40 & 2 & 3 & 9 & 1 & 5 & 0 & 0 & 20 & 1 & 7 \\ \hline
\end{tabular}
}
\caption{Recommended upfront interdictions on network 3 with two attackers}
\label{tbl:net3_u_2int_who}
\end{center}
\end{sidewaystable}

\begin{sidewaystable}[]
\begin{center}
\resizebox{\textwidth}{!}{
\begin{tabular}{|c|c|c|c|c|c|c|c|c|c|c|}
\hline
Budget & \begin{tabular}[c]{@{}c@{}}MP-MFNIP\\ Int Trafficker\end{tabular} & \begin{tabular}[c]{@{}c@{}}MP-MFNIP\\ Int Bottom\end{tabular} & \begin{tabular}[c]{@{}c@{}}MP-MFNIP\\ Int Victim\end{tabular} & \begin{tabular}[c]{@{}c@{}}MP-MFNIP 2nd\\ Int Current Vic\end{tabular} & \begin{tabular}[c]{@{}c@{}}MP-MFNIP 2nd\\ Int Prosp Vic\end{tabular} & \begin{tabular}[c]{@{}c@{}}MP-MFNIP-R\\ Int Trafficker\end{tabular} & \begin{tabular}[c]{@{}c@{}}MP-MFNIP-R\\ Int Bottom\end{tabular} & \begin{tabular}[c]{@{}c@{}}MP-MFNIP-R\\ Int Victim\end{tabular} & \begin{tabular}[c]{@{}c@{}}MP-MFNIP-R 2nd\\ Int Current Victim\end{tabular} & \begin{tabular}[c]{@{}c@{}}MP-MFNIP-R 2nd\\ Int Prosp Victim\end{tabular} \\ \hline
8 & 0 & 2 & 0 & 3 & 1 & 0 & 0 & 4 & 2 & 4 \\ \hline
12 & 1 & 1 & 1 & 3 & 1 & 0 & 0 & 6 & 2 & 4 \\ \hline
16 & 1 & 2 & 2 & 3 & 1 & 0 & 0 & 8 & 2 & 4 \\ \hline
20 & 2 & 2 & 1 & 3 & 1 & 0 & 0 & 10 & 2 & 4 \\ \hline
24 & 2 & 3 & 1 & 2 & 2 & 1 & 0 & 10 & 2 & 3 \\ \hline
28 & 2 & 3 & 3 & 2 & 3 & 1 & 0 & 12 & 2 & 3 \\ \hline
32 & 2 & 3 & 5 & 2 & 4 & 1 & 0 & 14 & 2 & 4 \\ \hline
36 & 2 & 3 & 7 & 2 & 4 & 1 & 0 & 16 & 2 & 4 \\ \hline
40 & 2 & 3 & 9 & 2 & 4 & 0 & 0 & 20 & 0 & 10 \\ \hline
\end{tabular}
}
\caption{Recommended upfront interdictions on network 4 with two attackers}
\label{tbl:net4_u_2int_who}

\vspace{2\baselineskip}
\resizebox{\textwidth}{!}{
\begin{tabular}{|c|c|c|c|c|c|c|c|c|c|c|}
\hline
Budget & \begin{tabular}[c]{@{}c@{}}MP-MFNIP\\ Int Trafficker\end{tabular} & \begin{tabular}[c]{@{}c@{}}MP-MFNIP\\ Int Bottom\end{tabular} & \begin{tabular}[c]{@{}c@{}}MP-MFNIP\\ Int Victim\end{tabular} & \begin{tabular}[c]{@{}c@{}}MP-MFNIP 2nd\\ Int Current Vic\end{tabular} & \begin{tabular}[c]{@{}c@{}}MP-MFNIP 2nd\\ Int Prosp Vic\end{tabular} & \begin{tabular}[c]{@{}c@{}}MP-MFNIP-R\\ Int Trafficker\end{tabular} & \begin{tabular}[c]{@{}c@{}}MP-MFNIP-R\\ Int Bottom\end{tabular} & \begin{tabular}[c]{@{}c@{}}MP-MFNIP-R\\ Int Victim\end{tabular} & \begin{tabular}[c]{@{}c@{}}MP-MFNIP-R 2nd\\ Int Current Victim\end{tabular} & \begin{tabular}[c]{@{}c@{}}MP-MFNIP-R 2nd\\ Int Prosp Victim\end{tabular} \\ \hline
8 & 0 & 1 & 2 & 3 & 1 & 0 & 0 & 4 & 3 & 1 \\ \hline
12 & 0 & 1 & 4 & 3 & 1 & 0 & 0 & 6 & 3 & 1 \\ \hline
16 & 0 & 2 & 4 & 3 & 1 & 0 & 0 & 8 & 2 & 3 \\ \hline
20 & 0 & 2 & 6 & 2 & 2 & 0 & 0 & 10 & 2 & 4 \\ \hline
24 & 0 & 3 & 6 & 2 & 3 & 0 & 0 & 12 & 2 & 4 \\ \hline
28 & 0 & 3 & 8 & 2 & 4 & 0 & 0 & 14 & 1 & 5 \\ \hline
32 & 0 & 3 & 10 & 1 & 7 & 0 & 0 & 16 & 1 & 7 \\ \hline
36 & 0 & 3 & 12 & 1 & 7 & 0 & 1 & 16 & 1 & 7 \\ \hline
40 & 0 & 4 & 12 & 1 & 7 & 0 & 2 & 16 & 1 & 7 \\ \hline
\end{tabular}
}
\caption{Recommended upfront interdictions on network 5 with two attackers}
\label{tbl:net5_u_2int_who}
\end{center}
\end{sidewaystable}

\end{appendices}
\end{document}